\documentclass[reqno,10pt]{amsart}

\usepackage{amsmath,amssymb,bbm,dsfont,url,comment}

\usepackage{float}
\usepackage{mathtools}
\usepackage{amsthm}
\usepackage{enumerate}
\usepackage{multicol}
\usepackage{stmaryrd}
\SetSymbolFont{stmry}{bold}{U}{stmry}{m}{n}

\usepackage{bm}

\usepackage{bbm}
\usepackage{color,colordvi}
\usepackage[mathscr]{eucal}
\usepackage{amsthm}
\newtheorem{counter}{}[section]
\theoremstyle{definition}
\newtheorem{definition}         [counter]{Definition}

\theoremstyle{plain}
\newtheorem{lemma}              [counter]{Lemma}

\newtheorem{proposition}        [counter]{Proposition}
\newtheorem{theorem}            [counter]{Theorem}
\newtheorem{corollary}          [counter]{Corollary}
\newtheorem{question}           [counter]{Question}

\newtheorem*{theorem*}          {Theorem}
\theoremstyle{remark}
\newtheorem{remark}             [counter]{Remark}
\newtheorem*{remark*}           {Remark}
\newtheorem{example}            [counter]{Example}

\newtheoremstyle{named}{8pt}{8pt}{\itshape}{}{\bfseries}{.}{.5em}{#1 #3}
\theoremstyle{named}
\newtheorem*{namedtheorem}{Theorem}

\numberwithin{equation}{section}

\usepackage{mdframed} 
\usepackage[margin=1.05in]{geometry}
\usepackage[english]{babel}
\usepackage[latin1]{inputenc}   
\usepackage[T1]{fontenc}   
\usepackage{latexsym,amsxtra,amscd,ifthen,amsmath,color, multicol}
\usepackage{hyperref}
\usepackage{amsfonts,dsfont,mathrsfs}
\usepackage{verbatim,parallel}
\usepackage{amssymb}
\usepackage{relsize}
\usepackage{mathtools}
\usepackage[title]{appendix}
\usepackage{bbm}
\usepackage{stmaryrd}
\usepackage{soul}% for underlining
\usepackage[all,cmtip]{xy}
\usepackage[normalem]{ulem} 
\usepackage{enumerate}
\usepackage{tikz-cd}
\usetikzlibrary{matrix,arrows}
\usepackage{multirow}
\usepackage{xcolor}
\usepackage{etoolbox}%needed to patch commands \patchcmd 
\usepackage{paracol}
\usepackage{quiver} %for including diagrams from quiver website
\hypersetup{
  colorlinks   = true, %Colours links instead of ugly boxes
  urlcolor     = blue, %Colour for external hyperlinks
  linkcolor    = blue, %Colour of internal links
  citecolor   = black %Colour of citations
}

\usepackage{setspace}
\usepackage{graphicx}
\newcommand*{\Scale}[2][4]{\scalebox{#1}{$#2$}}%

\makeatletter
\renewcommand{\@secnumfont}{\bfseries}
\renewcommand\subsection{\@startsection{subsection}{2}%
  \z@{0.9\linespacing\@plus0.9\linespacing}{0.5em}%
  {\normalfont\bfseries}}
\renewcommand\subsubsection{\@startsection{subsubsection}{3}%
  \z@{0.5\linespacing\@plus0.7\linespacing}{0.1em}%
  {\normalfont\bfseries}}
\makeatother

\makeatletter
\patchcmd{\@maketitle}{\normalsize}{\Large}{}{}
\patchcmd{\@setauthors}{\MakeUppercase {\authors }}{\MakeUppercase{\large\authors}}{}{}
\patchcmd{\section}{\normalfont}{\normalfont\Large\bfseries}{}{}
\patchcmd{\subsection}{\normalfont}{\normalfont\large\bfseries}{}{}
\makeatother

\makeatletter
\let\origsection\section
\renewcommand\section{\@ifstar{\starsection}{\nostarsection}}

\newcommand\nostarsection[1]
{\sectionprelude\origsection{#1}\sectionpostlude}

\newcommand\starsection[1]
{\sectionprelude\origsection*{#1}\sectionpostlude}

\newcommand\sectionprelude{%space at end of section title
  \vspace{0.65em}
}

\newcommand\sectionpostlude{%space at start of section title
  \vspace{0.3em}
}
\makeatother

\makeatletter
\patchcmd{\@startsection}{\@afterindenttrue}{\@afterindentfalse}{}{}
\patchcmd{\@startsubsection}{\@afterindenttrue}{\@afterindentfalse}{}{}
\makeatother

\makeatletter
\def\l@subsection{\@tocline{2}{0pt}{2.5pc}{5pc}{}}
\makeatother

\author{David Jaklitsch}
\address{Department of Mathematics, University of Oslo, P.O. box 1053, Blindern, 0316 Oslo, Norway}
\email{dajak@uio.no}

\author{Harshit Yadav}
\address{Department of Mathematics, University of Alberta,
Edmonton, Alberta, T6G 2G1, Canada}
\email{hyadav3@ualberta.ca}

\makeatletter
\@namedef{subjclassname@2020}{%
  \textup{2020} Mathematics Subject Classification}
\makeatother

\DeclareFontFamily{U}{mathc}{}
\DeclareFontShape{U}{mathc}{m}{it}%
{<->s*[1.03] mathc10}{}
\DeclareMathAlphabet{\mathcal}{U}{mathc}{m}{it}

\allowdisplaybreaks

\setuldepth{Berlin} %to set depth of underlining

\makeatother

\DeclareMathOperator{\End}{End}
\DeclareMathOperator{\Aut}{Aut}

\newcommand{\Ind}{\mathtt{Ind}}
\newcommand{\bt}{\boxtimes}

\newcommand{\btD}{\bt_{\,\subD}}

\newcommand{\C}{\mathcal{C}}
\newcommand{\B}{\mathcal{B}}
\newcommand{\D}{\mathcal{D}}

\newcommand{\X}{\mathcal{X}}
\newcommand{\Y}{\mathcal{Y}}
\newcommand{\Z}{\mathcal{Z}}
\newcommand{\M}{\mathcal{M}}
\newcommand{\N}{\mathcal{N}}
\newcommand{\cL}{\mathcal{L}}

\newcommand{\R}{\mathcal{R}}

\newcommand{\dS}{\mathbb{S}}

\newcommand{\dD}{\mathbb{D}}
\newcommand{\dN}{\mathbb{N}}

\newcommand{\fp}{\mathfrak{p}}
\newcommand{\fq}{\mathfrak{q}}
\newcommand{\fs}{\mathfrak{s}}
\newcommand{\fn}{\mathfrak{n}}
\newcommand{\fu}{\mathfrak{u}}

\newcommand{\fsl}{\mathfrak{sl}}

\newcommand{\ff}{\footnote}
\newcommand{\kk}{\Bbbk}

\newcommand{\ok}{\otimes_{\kk}}

\newcommand{\oC}{\otimes_\subC}
\newcommand{\oD}{\otimes_\subD}
\newcommand{\id}{\textnormal{\textsf{Id}}}
\newcommand{\Rep}{\textnormal{\textsf{Rep}}}

\newcommand{\lv}{\hspace{.025cm}{}^{\textnormal{\Scale[0.8]{\vee\!}}} \hspace{-.025cm}}
\newcommand{\rv}{\hspace{.015cm}{}^{\textnormal{\Scale[0.8]{\vee}}} \hspace{-.03cm}}

\newcommand{\coev}{\textnormal{coev}}

\newcommand{\Vect}{\mathsf{Vec}}
\newcommand{\Hom}{\textnormal{Hom}}

\newcommand{\loc}{\textnormal{loc}}

\newcommand{\unit}{\mathds{1}}

\newcommand{\uHom}{\underline{\Hom}}
\newcommand{\uNat}{\textnormal{\underline{Nat}}}

\newcommand{\Rex}{{\sf Rex}}

\newcommand{\Fun}{{\sf Fun}}
\newcommand{\re}{{\sf re}}
\newcommand{\ra}{{\sf ra}}
\newcommand{\rra}{{\sf rra}}

\newcommand{\la}{{\sf la}}
\newcommand{\lla}{{\sf lla}}

\newcommand{\tr}{{\,{\Act}\,}}
\newcommand{\trF}{\,{{\Act}^{\!\scriptscriptstyle F}}\,}
\newcommand{\tlF}{\,{\triangleleft^{\scriptscriptstyle F}}\,}
\newcommand{\trG}{\,{{\Act}^{\!\scriptscriptstyle G}}\,}

\newcommand{\btr}{\raisebox{0.15ex}{\Scale[0.85]{\blacktriangleright}}}

\newcommand{\tl}{\triangleleft}

\newcommand{\op}{\textnormal{op}}

\definecolor{darkgreen}{RGB}{55,138,0}
\definecolor{burntorange}{RGB}{180,85,0}
\definecolor{navyblue}{RGB}{18,40,180}
\definecolor{cyan(process)}{rgb}{0.0, 0.6, 1.0}

\newcommand{\Frob}{{\texorpdfstring{$\otimes$}{}-Frobenius }}
\newcommand{\Frobc}{{$\otimes$-Frobenius,} }
\newcommand{\Frobp}{{$\otimes$-Frobenius.} }

\newcommand{\ogH}{\overline{\gH}}

\newcommand{\oalpha}{\overline{\alpha}}

\newcommand{\odD}{{\overline{\dD}}}

\newcommand{\subodD}{{\scriptscriptstyle \odD}}
\newcommand{\subdD}{{\scriptscriptstyle \dD}}

\newcommand{\moD}{\!\mathrm{mod}}

\def\Act           {{\triangleright}}   
\def\act           {\,{\Act}\,}         
\def\Actr          {{\triangleleft}}    
\def\actr          {\,{\Actr}\,}

\def\be            {\begin{equation}}
\def\Colon         {:\quad}

\def\dd            {{\rv\rv}}
\def\ldd           {{}{\rv\rv\!}}
\def\ee            {\end{equation}}
\def\Enumerate     {\def\leftmargini{1.34em}~\\[-1.42em]\begin{enumerate}}
\def\Enumeratei    {%\def\leftmargini{1.84em}
	~\\[-1.42em]\begin{enumerate}[{\rm (i)}]\addtolength\itemsep{+1pt}}
\def\Hom           {\mathrm{Hom}}

\def\FD             {{{}_\subF\D}}
\def\FM             {{{}_\subF\M}}
\def\GM             {{{}_\subG\M}}
\def\FN             {{{}_\subF\N}}

\def\id            {{\mathrm{id}}}

\def\iN            {\,{\in}\,}

\def\opp           {^{\rm opp}}              

\def\PsiN           {{\Psi_{\!\scriptscriptstyle\N}}}
\def\PsiM           {{\Phi_{\!\scriptscriptstyle\M}}}
\def\subFun           {{{\Fun_{\scriptscriptstyle\C}\scriptscriptstyle(\!\M,\scriptscriptstyle\N)}}}

\def\subF           {{\scriptscriptstyle F}}
\def\subG           {{\scriptscriptstyle G}}
\def\subM           {{\!\scriptscriptstyle \M}}
\def\subFM           {{\!\scriptscriptstyle {\FM}}}
\def\subN           {{\!\scriptscriptstyle \N}}
\def\subC           {{\!\scriptscriptstyle \C}}
\def\subX           {{\!\scriptscriptstyle \X}}
\def\subY           {{\!\scriptscriptstyle \Y}}
\def\subD           {{\!\scriptscriptstyle \D}}

\def\supC           {{\scriptscriptstyle \C}}
\def\supD           {{\scriptscriptstyle \D}}
\def\ot            {\,{\otimes}\,}        
\newcommand{\Se}{\mathbb{S}}              % right Serre functor
\newcommand{\lSe}{\overline{\mathbb{S}}}              % left Serre functor

\def\xsim         {\,{\xrightarrow{~\sim~}}\,}

\def\piv {\mathrm{piv}}

\def\ff {\mathsf{f}}

\newcommand{\gH}{\mathtt{g}_H}
\newcommand{\gHp}{\mathtt{g}_{H'}}

\newcommand{\tfp}{\widetilde{\fp}}
\newcommand{\tfq}{\widetilde{\fq}}
\newcommand{\tfu}{\,\widetilde{\fu}\,}

\def\ZFD{{\Z({}_\subF\D_\subF)}}
\def\FDF{{{}_\subF \D_\subF}}
\def\FDG{{{}_\subF \D_\subG}}

\newcommand{\chiFf}{\chi_{\scriptscriptstyle F_{\ff}}}

\newcommand{\sa}{\mathtt{a}}
\newcommand{\sg}{\mathtt{g}}

\begin{document}
    
\title[$\otimes$-Frobenius functors and exact module categories]
{\textnormal{$\otimes$-Frobenius functors and exact module categories}}

\begin{abstract}
We call a tensor functor $F:\mathcal{C}\to\mathcal{D}$ between finite tensor categories \emph{$\otimes$-Frobenius} if its left and right adjoints are isomorphic as $\mathcal{C}$-bimodule functors. We give several characterizations of this notion---most notably, $F$ is $\otimes$-Frobenius if and only if the centralizer $Z({}_{F}\!{\mathcal{D}}_{\!F})$ is unimodular. We use them to analyze how actions on module categories behave under pullback along $F$. For perfect functors, we show that twisting a $\mathcal{D}$-module category $\mathcal{M}$ along $F$ preserves exactness, and that pivotality, unimodularity, and sphericality are preserved whenever $F$ is $\otimes$-Frobenius (or, more generally, Frobenius with respect to $\mathcal{M}$). 

Applications include: (i) explicit criteria for $\otimes$-Frobenius functors arising from bialgebra maps $f\!:\!H'\!\to\!H$ between finite-dimensional Hopf algebras;  and (ii) criteria ensuring that objects of internal natural transformations are (symmetric) Frobenius algebras in $Z(\mathcal{C})$. Along the way we show that central tensor functors are Frobenius iff they are $\otimes$-Frobenius and that any tensor functor between separable fusion categories is $\otimes$-Frobenius, answering questions of Flake-Laugwitz-Posur from \cite{flake2024frobenius2}.
\end{abstract}

\maketitle

%%%%%%%%%%%%%%%%%%%%%%%%%%%%%%%%%%%%%%%%%%%%%%%%%%%%%%%%%%%%%%%
%%%%%%%%%%%%%%%%%%%%%%%%%%%%%%%%%%%%%%%%%%%%%%%%%%%%%%%%%%%%%%%
%%%%%%%%%%%%%%%%%%%%%%%%%%%%%%%%%%%%%%%%%%%%%%%%%%%%%%%%%%%%%%%
%%%%%%%%%%%%%%%%%%%%%%%%%%%%%%%%%%%%%%%%%%%%%%%%%%%%%%%%%%%%%%%
%%%%%%%%%%%%%%%%%%%%%%%%%%%%%%%%%%%%%%%%%%%%%%%%%%%%%%%%%%%%%%%
%%%%%%%%%%%%%%%%%%%%%%%%%%%%%%%%%%%%%%%%%%%%%%%%%%%%%%%%%%%%%%%

\section{Introduction}
Tensor categories have emerged as fundamental structures in the mathematical study of quantum field theory. They also offer a unified framework for a wide array of mathematical domains, including subfactor theory of von Neumann algebras, vertex operator algebras and Hopf algebras. In this article, we concentrate on the class of finite tensor categories and our results are particularly relevant beyond the semisimple setting. 

A typical approach to understanding an algebraic structure is via its representations. Module categories play this central role in the theory of tensor categories. About two decades ago, the work of Etingof and Ostrik 
\cite{etingof2004finite} showed that studying all module categories over a finite tensor category is generally 
intractable, suggesting a focus on the specific class of \textit{exact module categories}.

Physically, if $\C$ is a semisimple tensor category associated with a $2$-dimensional rational conformal field theory (CFT), semisimple (exact) $\C$-module categories supply the data necessary for obtaining the full local conformal field theory \cite{fuchs2002tft}. Extending beyond rational CFTs to the finite non-semisimple setting, Fuchs and Schweigert \cite{fuchs2021bulk} proposed the class of \textit{pivotal} exact $\C$-module categories (explored in \cite{schaumann2013traces,shimizu2023relative}) as the required data for constructing the full CFT. Since then, additional classes of module categories, such as {\it unimodular} \cite{yadav2023unimodular} and {\it spherical} \cite{spherical2025}, have been introduced, which hold significance for the construction of topological quantum field theories (TQFTs) \cite{fgjs2024manifestly}. 

The primary objective of this article is to construct new examples of unimodular, pivotal, and spherical $\C$-module categories. To achieve this, we investigate the process of twisting module actions by tensor functors. Specifically, given a tensor functor $F\colon \C\!\longrightarrow\!\D$, any $\D$-module category $\M$ can be regarded as a $\C$-module category by pulling back the $\D$-action along $F$. We establish necessary and sufficient conditions to ensure that $F$ preserves exactness, unimodularity, pivotality, and sphericality of module categories under twisting. We introduce a new class of functors, called \emph{\Frob functors}, which capture these preservation conditions.

\Frob functors are tensor functors $F\colon\C\!\to\!\D$ whose left and right adjoints are isomorphic in the $2$-category $\mathbf{Bimod}(\C)$ of $\C$-bimodules. This is a case of ambidextrous adjunction studied in \cite{lauda2006frobenius,street2004frobenius} and more recently in \cite{flake2024frobenius1}. We investigate \Frob functors, give several characterizations, and present examples from Hopf algebras. We further apply our results to show that, under suitable conditions, the objects called \emph{internal natural transformations} (introduced in \cite{fuchs2021internal}) form a (symmetric) Frobenius algebra in the Drinfeld center of a tensor category. These play an important role in constructing modular tensor categories and describing boundary conditions in $2$D-CFT \cite{fuchs2002tft}.

\subsubsection*{Main Results}
We refer the reader to \S\ref{sec:background} for the background definitions. A functor $F\colon\C\!\longrightarrow\!\D$ between tensor categories is called \textit{perfect} iff it preserves projective objects and \textit{Frobenius} in case its left and right adjoints are isomorphic. These are properties independent of a monoidal structure on $F$. A functor is called {\it Frobenius monoidal} if it admits a lax monoidal and an oplax monoidal structure compatible in a way akin to the product and coproduct of a Frobenius algebra.
In the presence of a tensor (that is, a $\kk$-linear, faithful and strong monoidal) structure on $F$, the Frobenius property on $F$ is related to its right adjoint $F^\ra$ admitting a Frobenius monoidal structure. In Proposition~\ref{prop:Frob-mon-implies-Frobenius} we show that if $F^\ra$ is Frobenius monoidal, then $F$ is Frobenius. Conversely, if $F$ is Frobenius and $F^\ra$ is exact and faithful, then $F^\ra$ is Frobenius monoidal.

A tensor functor $F\colon\C\!\longrightarrow\!\D$ induces a $\C$-bimodule category structure on $\D$ (which we denote as $\FDF$). Then $F$ can be seen as a $\C$-bimodule functor. Applying the center construction (see \S\ref{subsec:center}) to this bimodule functor we obtain a tensor functor $\Z(F)\colon\Z(\C)\!\longrightarrow\!\Z(\FDF)$. Moreover, the right adjoint $F^\ra$ of $F$ is also a $\C$-bimodule functor and its center $\Z(F^\ra)$ is a lax monoidal functor that is the right adjoint of $\Z(F)$.

We call $F$ {\it \Frob} \!\!if its adjoints are isomorphic as $\C$-bimodule functors. 
As the main result of \S\ref{sec:perfect-tensor}, we obtain various characterizations of \Frob functors, relating them to Frobenius functors, Frobenius monoidal functors and unimodular tensor categories:

\begin{namedtheorem}[A]$($\ref{thm:tFrob-char}$\,)$\label{thm:intro-A}
Let $F\colon\C\!\longrightarrow\!\D$ be a perfect tensor functor. Then the following are equivalent:
\begin{enumerate}[{\rm (i)}]
    \item $F$ is \Frob\!\!.
    \item The tensor functor $\Z(F):\Z(\C)\rightarrow\Z(\FDF)$ is Frobenius.
    \item The tensor category $\Z(\FDF)$ is unimodular.
    \item The lax monoidal functor $\Z(F^\ra):\Z(\FDF)\!\longrightarrow\!\Z(\C)$ is $($braided$)$ Frobenius monoidal.
\end{enumerate}
\end{namedtheorem}
Here the equivalence of parts (i), (ii) and (iii) is proved by utilizing results from  \cite{shimizu2017relative}, and the equivalence with (iv) follows from \cite{flake2024frobenius1}.

While tensor functors that are Frobenius are not in general \Frobc using this result, we show that central tensor functors are Frobenius if and only if they are \Frob (Proposition~\ref{prop:central-tensor-Frobenius}). Furthermore, we prove that any tensor functor between separable fusion categories is \Frob (Theorem~\ref{thm:separable-implies-tensor-Frob}), thereby answering \cite[Questions~5.22 \& 5.23]{flake2024frobenius2}.

In \S\ref{sec:F-twisted}, we apply the results of \S\ref{sec:perfect-tensor} to twisted actions on module categories. A tensor functor $F\colon\C\!\longrightarrow\!\D$ induces on a left $\D$-module category $\M$ a left $\C$-module action via:
\begin{equation*}
    c\trF m \coloneqq F(c)\act m 
\end{equation*}
for every $c\in\C$ and $m\in\M$. We denote the resulting $\C$-module category as $\FM$. 

When $\M$ is exact as a $\D$-module category, we show that $\FM$ is exact as a $\C$-module category if and only if $F$ is perfect. For perfect functors, the \textit{relative modular object} $\chi_\subF\in\D$, introduced in \cite{shimizu2017relative}, `measures' how much $F^\la$ and $F^\ra$ differ from one another. 
Moreover, $\chi_\subF$ is equipped with a \textit{$F$-half-braiding} $\sigma\colon\chi_{\subF}\otimes F(c)\to F(c)\otimes \chi_{\subF}$ such that $(\chi_\subF,\sigma) \in \Z(\FDF)$. In particular, using $\sigma$, $\chi_\subF$ induces a $\C$-module functor
\[ \chi_\subF \act - \Colon \FM \longrightarrow\FM\,.\]

To understand properties like pivotality, unimodularity, etc.\ of a $\C$-module category the main hurdle is to determine its so-called relative Serre functor. 
In Proposition~\ref{prop:FMSerre}, we provide a formula for the relative Serre functor of the $\C$-module category $\FM$ and show that it is described using the functor $\chi_\subF\act -$ and the relative Serre functor $\dS_\subM^\supD$ of the $\D$-module category $\M$ as follows:
\[ \dS_{\FM}^\supC \cong \chi_\subF \act \dS_\subM^\supD. \]
With this in mind, we say that a tensor functor $F$ is {\it Frobenius with respect to $\M$} if $\chi_F\act-$ is isomorphic to $\id_{\FM}$ as a $\C$-module functor. We show that \Frob functors are Frobenius with respect to every $\D$-module $\M$.
This allows us to obtain the following result: 

\begin{namedtheorem}[B]\label{thm:intro-B}
Let $F\colon\C\!\longrightarrow\!\D$ be a tensor functor and $\M$ a $\D$-module category. The table below summarizes the assumptions needed to transfer a property on $\M$ to the $\C$-module category $\FM$.
The `result' column shows that said property is transferred if $F$ obeys the respective assumptions.
\begin{table}[ht]
    \centering
    \begin{tabular}{c|c|c|c||c c c}
        \multirow{2}{*}{Reference} & \multicolumn{3}{c||}{Assumptions} & \multicolumn{3}{c}{Result} \\
        \cline{2-7}
         & $\C,\,\D $ & $\M$ & F & F &  & $\FM$ \\
        \hline
        Prop.~\ref{prop:perfect-iff} & FTC & exact & tensor & perfect & $\implies$ & exact \\
        %\hline
        Prop.~\ref{prop:uni-tensor-Frob} & FTC & unimodular & perfect & \Frob & $\implies$ & unimodular \\
        %\hline
        Prop.~\ref{prop:piv-tensor-Frob} & pivotal & pivotal & perfect, pivotal & \Frob\!\!  & $\implies$ & pivotal \\
        %\hline
        Prop~\ref{prop:sph-Frob} & spherical & spherical & perfect, pivotal & \Frob\!\! & $\implies$ & spherical 
    \end{tabular}
    \label{tab:example}
\end{table}
\end{namedtheorem}
\vspace{-0.5cm}
The preservation of pivotality, unimodularity and sphericality holds more generally in case $F$ is Frobenius with respect to $\M$.
A consequence of these results is that: if $F$ is a perfect tensor functor, then its right adjoint $F^{\ra}$ preserves exact algebras. Furthermore, if $\C$ and $\D$ are pivotal and $F$ is a pivotal \Frob functor, then $F^{\ra}$ preserves exact symmetric Frobenius algebras.

%\smallskip
In \S\ref{sec:Hopf-examples}, we illustrate our results for the categories $\Rep(H)$ of finite-dimensional representations of a finite-dimensional Hopf algebra $H$. We show that the tensor functor $F_\ff\colon \Rep(H)\!\longrightarrow\!\Rep(H')$, induced by a bialgebra map $\ff\colon H'\rightarrow H$, is \Frob if and only if $\ff(\sg_{H'})=\gH$ and $\alpha_H\circ \ff = \alpha_{H'}$, where $\sg$ and $\alpha$ stand for the distinguished grouplike elements of a Hopf algebra and its dual (Proposition~\ref{prop:rel-modular-Hopf}). 
We provide examples of such bialgebra maps and explain how our results explain some recent findings in \cite{flake2024frobenius2} obtained using different techniques. Moreover, for $L$ an exact left $H'$-comodule algebra, we show that $F_{\ff}$ is Frobenius with respect to $\Rep(L)$ if and only if it admits a certain invertible element called $\ff$-Frobenius element (Definition~\ref{defn:f-Frob-element}). We also explain the consequences of our results from \S\ref{sec:F-twisted} for pivotal (and unimodular) $\Rep(H)$-module categories.

In \S\ref{sec:internal-natural}, we provide another application of Theorem~B to generalize some algebraic results obtained in \cite{fuchs2021internal} inspired by logarithmic CFTs.
For a $\C$-module category $\M$, let $\C_\subM^*$ denote the category of $\C$-module endofunctors of $\M$. There is a perfect tensor functor $\Psi\colon\Z(\C) \!\longrightarrow\! \C_\subM^*$ which allows us to consider $\C_\subM^*$-module categories as $\Z(\C)$-module categories. We prove that $\Psi$ is \Frob if and only if $\M$ is unimodular (Proposition \ref{prop:psi-Frobenius}). Given a finite tensor category $\C$ and $\M$ and $\N$ exact left $\C$-module categories, we show that the $\Z(\C)$-module category $\Fun_{\C}(\M,\N)$ is exact, thereby generalizing \cite[Thm.\ 18]{fuchs2021internal}. A similar statement holds for unimodular module categories. If, in addition, $\C$ is pivotal, $\M$ and $\N$ are pivotal $\C$-modules and  $\M$ is unimodular, then $\Fun_{\C}(\M,\N)$ is a pivotal $\Z(\C)$-module category. 
Furthermore, $\Psi^{\ra}$ is a pivotal Frobenius monoidal functor.

For $G_1,G_2\in\Fun_{\C}(\M,\N)$, we show that the object $\uNat(G_1,G_2)\in\Z(\C)$ of internal natural transformations \cite{fuchs2021internal} is given by $\Psi^{\ra}(G_1^{\ra}\circ G_2)$. Thus, Theorem~B above gives criteria for when $\uNat(G,G)$ is a symmetric Frobenius algebra in $\Z(\C)$, generalizing the results of \cite{fuchs2021internal}.

\subsubsection*{Acknowledgments}
D.J.\ is supported by The Research Council of Norway - project 324944, and
H.Y.\ by a start-up grant from the University of Alberta and an NSERC Discovery Grant. We would like to thank J.~Fuchs and C.~Schweigert for discussions that encouraged us to start this project, and K.~Shimizu for  insights about Example \ref{ex:uqsl2-kG}. We are grateful to the anonymous referees for their detailed comments that significantly improved the manuscript. 

%%%%%%%%%%%%%%%%%%%%%%%%%%%%%%%%%%%%%%%%%%%%%%%%%%%%%%%%%%%%%%%
%%%%%%%%%%%%%%%%%%%%%%%%%%%%%%%%%%%%%%%%%%%%%%%%%%%%%%%%%%%%%%%
%%%%%%%%%%%%%%%%%%%%%%%%%%%%%%%%%%%%%%%%%%%%%%%%%%%%%%%%%%%%%%%
%%%%%%%%%%%%%%%%%%%%%%%%%%%%%%%%%%%%%%%%%%%%%%%%%%%%%%%%%%%%%%%
%%%%%%%%%%%%%%%%%%%%%%%%%%%%%%%%%%%%%%%%%%%%%%%%%%%%%%%%%%%%%%%
%%%%%%%%%%%%%%%%%%%%%%%%%%%%%%%%%%%%%%%%%%%%%%%%%%%%%%%%%%%%%%%

\section{Background}\label{sec:background}
In this article, we will work over an algebraically closed field $\kk$. $\Vect$ will denote the category of finite-dimensional $\kk$-vector spaces. Given a category $\X$, we will denote its opposite category as $\X^{\op}$. For a morphism $f\colon x\rightarrow x'$ in $\X$, the corresponding morphism in $\C^{\op}$ is denoted as $f^{\op}\colon x'\rightarrow x$.
Unless stated otherwise, we will follow the same conventions as in \cite{etingof2016tensor}.
%%%%%%%%%%%%%%%%%%%%%%%%%%%%%%%%%%%%%%%%%%%%%%%%%%%%%%%%%%%%%%%
\subsection{Monoidal categories and functors}
Owing to Mac Lane's coherence theorem, we can assume that all monoidal categories are strict. A generic monoidal category is denoted as a triple $(\C,\otimes_\subC,\unit_\subC)$ where $\oC$ stands for the tensor product and $\unit_{\subC}$ for the unit object. The \emph{monoidal opposite} $\overline{\C\,}$ of $\C$ is the monoidal category having the same underlying category as $\C$, but with reversed tensor product, i.e.\ $c \,{\otimes_ {\overline{\,\subC}}}\,d \coloneqq d \oC c$. Moreover, the opposite category $\C^{\op}$ is also endowed with a monoidal structure given by the same tensor product and unit as $\C$.

A monoidal category $\C$ is called \textit{rigid} if every object $c\in\C$ admits left and right duals.
We use the following conventions for dualities: 
the right dual $c\rv$ of an object $c\in\C$ comes equipped with evaluation and coevaluation morphisms
\begin{equation*}
{\rm ev}_c\Colon c\rv \otimes\, c\longrightarrow \unit \qquad{\rm and}\qquad
{\rm coev}_c\Colon \unit\longrightarrow c\otimes c\rv \,,
\end{equation*}
and the left dual $\rv c$ of $c\in\C$ comes with evaluation and coevaluation morphisms
\begin{equation*}
\widetilde{{\rm ev}_c}\Colon c \,\otimes \rv c\longrightarrow \unit \qquad{\rm and}\qquad
\widetilde{{\rm coev}_c}\Colon \unit\longrightarrow {}\rv c\otimes c \,.
\end{equation*}
This matches the conventions in \cite{fuchs2020eilenberg}. 
A rigid monoidal category is called \textit{pivotal} if it is equipped with a \textit{pivotal structure}, that is, a monoidal natural isomorphism $\fp\colon \id_{\,\subC} \xRightarrow{\;\sim\;} (-)\dd$.

An {\it algebra} in $\C$ is a triple $(A,m:A\otimes A\to A,u:\unit\to A)$ where $m$ satisfies $m\circ(\id_A\otimes m) = m\circ(m\otimes \id_A)$ and $u$ satisfies $m\circ(\id_A\otimes u) = \id_A = m\circ(u \otimes \id_A)$. A \textit{coalgebra} in $\C$ is a triple $(A,\Delta:A\to A\otimes A,\varepsilon: A\to\unit)$ such that $(A,\Delta^{\op},\varepsilon^{\op})$ is an algebra in $\C^{\op}$.
A {\it Frobenius algebra} is a tuple $(A,m,u,\Delta,\varepsilon)$ where $(A,m,u)$ forms an algebra, $(A,\Delta,\varepsilon)$ forms a coalgebra and $m,\Delta$ satisfy:
\[ (m \otimes \text{id}_A)\circ(\text{id}_A \otimes \Delta) = \Delta\circ m = (\text{id}_A \otimes m)\circ(\Delta \otimes \text{id}_A)\,. \]
Let $\C$ be a pivotal category with pivotal structure $\fp$. A \textit{symmetric Frobenius} algebra \cite{fuchs2008frobenius} is a Frobenius algebra $(A,m,u,\Delta,\varepsilon)$ satisfying:
\begin{equation*}%\label{eq:symalg}
    (\varepsilon \, m \otimes \id_{A\rv})(\id_A\otimes\coev_A) =   
    (\fp_{\lv A} \otimes \varepsilon \, m)(\widetilde{\coev_A} \otimes \id_A) \,.
\end{equation*}

Let $\C$ and $\D$ be monoidal categories. A \textit{$($lax$)$ monoidal} functor is a functor $F\colon\C\longrightarrow\D$ along with a morphism $F_0: \unit_\subD\rightarrow F(\unit_\subC)$ and a natural transformation $F_2(c,c')\colon F(c)\otimes_\subD F(c') \rightarrow F(c\otimes_\subC c')$ for $c,c'\in\C$ obeying certain associativity and unitality conditions. In the case that $F_0$ and $F_2(c,c')$ are isomorphisms we call $F$ \textit{strong monoidal}.

An \textit{oplax monoidal} functor is a functor $F\colon\C\longrightarrow\D$ equipped with a morphism $F^0\colon F(\unit_\subC) \rightarrow \unit_\subD$ and a natural transformation $F^2(c,c')\colon F(c\otimes_\subC c') \rightarrow F(c) \otimes_\subD F(c')$ such that $(F^{\op},(F^0)^{\op},(F^2)^{\op})$ is lax monoidal. 

A tuple $(F,F_0,F_2,F^0,F^2)$ is called \textit{Frobenius monoidal} if $(F,F_0,F_2)$ is lax monoidal, $(F,F^0,F^2)$ is oplax monoidal and for all $c,c',c''\in\C$, the following conditions are fulfilled \cite{day2008note}:
\begin{align*}
(\id_{F(c)}\otimes_{\subD} F_2(c',c'') )\;(F^2(c,c')\otimes_{\subD} \id_{F(c'')})  & \; =\; F^2(c,c'\otimes_{\subC} c'') F_2(c\otimes_{\subC} c',c''), \\
(F_2(c,c')\otimes_{\subD} \id_{F(c'')})\; ( \id_{F(c)}\otimes_{\subD} F^2(c',c'')) & \; = \;F^2(c\otimes_{\subC} c', c'') F_2(c, c' \otimes_{\subC} c'')\,.
\end{align*}
Any strong monoidal functor $(F,F_0,F_2)$ is Frobenius monoidal with $F^0=F_0^{-1}$ and $F^2=F_2^{-1}$. 

A Frobenius monoidal functor $F:\C\longrightarrow\D$ between rigid monoidal categories preserves dualities, that is, we have natural isomorphisms $\zeta_c:F(c\rv)\xsim F(c)\rv$. 
Let $(\C,\fp^{\supC})$ and $(\D,\fp^\supD)$ be two pivotal monoidal categories. We call such a Frobenius monoidal functor $F:\C\longrightarrow\D$ \textit{pivotal} if it satisfies:
\begin{equation}\label{eq:piv-functor}
    \zeta_c{\rv} \circ \fp^\supD_{F(c)} =  \zeta_{c{\rv}} \circ F(\fp_c^\supC)\Colon F(c) \rightarrow F(c{\rv}){\rv}.
\end{equation}
By applying $\zeta$ twice, we obtain a natural isomorphism $\xi_c:F(c\rv\rv)\xsim F(c)\rv\rv$.

Let $(F,F_0,F_2)$ be a strong monoidal functor from $\C$ to $\D$ with a right adjoint $F^\ra$ with unit $\eta$ and counit $\varepsilon$. The monoidal structure of $F$ induces a lax monoidal structure on $F^\ra$ via
\begin{equation*}
    F^\ra_2(d,d') \coloneqq F^\ra(\varepsilon_d\otimes \varepsilon_{d'})\circ F^\ra\left( F_2(F^\ra(d),F^\ra(d'))^{-1} \right) \circ \eta_{F^\ra(d)\otimes F^\ra(d')} , \;\;\;  F^\ra_0 = F^\ra(F^0) \circ \eta_{\unit_{\C}} \,.
\end{equation*}
By means of $\eta$, $\varepsilon$ and $F_2$, we can form two natural transformations called \textit{left and right coHopf operators}, respectively \cite{bruguieres2011hopf}:
\begin{equation}\label{eq:coHopf-operators}
     h^l_{d,c}\Colon F^\ra(d)\oC c\rightarrow F^\ra(d\oD F(c)) \qquad{\rm and}\qquad h^r_{d,c}\Colon c\oC F^\ra(d) \rightarrow F^\ra(F(c)\oD d)\,. 
\end{equation}
In a similar manner, when $F$ admits a left adjoint $F^\la$, we can equip $F^\la$ with an oplax monoidal structure $((F^\la)^2,(F^\la)^0)$. We use the unit, counit of the adjunction $F^\la\dashv F$ and $F_2$ to define natural transformations called \textit{left and right Hopf operators}, respectively \cite{bruguieres2011hopf}:
\begin{equation}\label{eq:Hopf-operators}
    H^l_{d,c}\Colon F^\la(d\oD F(c)) \rightarrow F^\la(d) \oC c \qquad{\rm and}\qquad H^r_{d,c}\Colon F^\la(F(c)\oD d) \rightarrow c\oC F^\la(d)\,.
\end{equation}  
If $\C$ and $\D$ are rigid, the maps $h^l, h^r$ and $H^l,H^r$ are all isomorphisms. We note that in references such as \cite{flake2024frobenius1,flake2024frobenius2,flake2024projection}, the co/Hopf operators are also called projection formula morphisms.

%%%%%%%%%%%%%%%%%%%%%%%%%%%%%%%%%%%%%%%%%%%%%%%%%%%%%%%%%%%%%%%
%%%%%%%%%%%%%%%%%%%%%%%%%%%%%%%%%%%%%%%%%%%%%%%%%%%%%%%%%%%%%%%
%%%%%%%%%%%%%%%%%%%%%%%%%%%%%%%%%%%%%%%%%%%%%%%%%%%%%%%%%%%%%%%

\subsection{Finite multitensor categories and module categories}\label{subsec:finite-tensor}

A \textit{finite $\kk$-linear category} is an abelian category that is equivalent to the category of finite-dimensional representations of a finite dimensional $\kk$-algebra. All categories considered in this article will be finite $\kk$-linear and all functors are assumed to be additive and $\kk$-linear, unless otherwise stated.
Every $\kk$-linear functor between finite $\kk$-linear categories that is left (right) exact admits a left (right) adjoint \cite[Cor.\ 1.9]{douglas2019balanced}.
The \textit{image} of a $\kk$-linear functor $F\colon\X\longrightarrow\Y$, denoted by $\mathrm{im}(F)$, is the full subcategory of $\Y$ consisting of subquotients of objects in the essential image of $F$.
A $\kk$-linear functor $F\colon\X\rightarrow\Y$ is called \textit{surjective} iff its image is equivalent to $\Y$.

A \textit{finite multitensor category} is a finite abelian $\kk$-linear, rigid monoidal category with $\kk$-bilinear tensor product. If, in addition, $\End_{\C}(\unit)\cong \kk$ holds, we call $\C$ a \textit{finite tensor category}.

A \textit{tensor functor} is a $\kk$-linear, exact, and strong monoidal functor. It is established that tensor functors between finite tensor categories are faithful \cite[Corollaire~2.10(ii)]{deligne2007categories}. The image $\mathrm{im}(F)$ of a tensor functor $F$ is a tensor subcategory of $\D$.
A fully faithful tensor functor is called \textit{injective}.

A \textit{left module category} over a finite tensor category $\C$ is a finite $\kk$-linear category $\M$ equipped with a module action functor $\tr:\C\times \M \rightarrow \M$ that is $\kk$-bilinear and exact in each variable, and natural isomorphisms $x\tr(y\tr m)\cong (x\otimes y)\tr m$ (for $x,y\in\C$, $m\in\M$) that satisfy certain coherence conditions. Right module categories and bimodule categories are similarly defined \cite[Def.\ 7.1.7]{etingof2016tensor}.
A module category ${}_\subC\M$ is called \textit{exact} iff for any $m\in \M$ and any projective object $p\in \C$, $p\tr m$ is projective in $\M$. A module category is called \textit{indecomposable} iff it is not equivalent to a direct sum of two non-trivial module categories. A $(\C,\D)$-bimodule category is \textit{exact} iff it is exact as left $\C$-module category and also exact as a right $\D$-module category. 
We call an algebra $A\in\C$ \textit{exact} (resp.~\textit{indecomposable}) if the left $\C$-module category $\moD_A(\C)$ of right $A$-modules in $\C$ is exact (resp.~indecomposable).

A \emph{$\C$-module functor} between two $\C$-module categories $\M$ and $\N$ consists of a functor $H\colon \M\to\N$ endowed with a \textit{module structure}, i.e.\ a natural isomorphism 
\[\varphi_{c,m}\Colon H(c\act m)\xsim c\act H(m)\] 
for $c\iN\C$ and $m\iN\M$ satisfying the appropriate pentagon axioms.
Module functors can be composed: given $\C$-module categories $\M$, $\N$ and $\cL$ and $\C$-module functors $H_1\colon\M\to\N$ and $H_2\colon\N\to\cL$ the composition functor $H_2\circ H_1\colon\M\to\cL$ inherits a $\C$-module functor structure via
\begin{equation}\label{eq:mod_fun_comp}
    H_2\circ H_1(c\tr m)\xrightarrow{H_2(\varphi^1_{c,m})} H_2(c\tr H_1(m))\xrightarrow{\varphi^2_{c,H_1(m)}} c\tr H_2(H_1(m))
\end{equation}
where the isomorphism $\varphi^i$ is the module structure of $H_i$ for $i=1,2$.

A \emph{module natural transformation} between module functors is a natural transformation between the underlying functors commuting with the respective module structures. Given $\C$-module categories $\M$ and $\N$, $\Fun_\C(\M,\N)$ denotes the category with $\C$-module functors between $\M$ and $\N$ as objects and module natural transformations as morphisms. 

Let $\C$ be a finite tensor category and $\M$ an indecomposable exact left $\C$-module category. The \textit{dual tensor category} of $\C$ with respect to $\M$ is the finite tensor category $\C_\subM^*:=\Fun_{\C}(\M,\M)$ of $\C$-module endofunctors with tensor product given by composition of module functors $H\otimes H'=H\circ H'$ and monoidal unit $\id_\subM$. In the case that $\M$ is a decomposable exact left $\C$-module category, then $\C_\subM^*$ is a multitensor category. Furthermore, given exact $\C$-module categories $\M$ and $\N$, the category of module functors $\Fun_{\C}(\M,\N)$ is endowed  with the structure of a $(\C_\subN^*,\C_\subM^*)$-bimodule category, via composition of module functors.

Let $\M_\subD$ be a right module category and ${}_\subD\N$ a left module category. A \textit{balanced functor} is a bilinear functor $F \colon \M\times\N\longrightarrow\Y$ into a $\kk$-linear category $\Y$ equipped with a natural isomorphism
  \begin{equation*}
  F(m\actr d,n) \xsim F(m,d\act n)
  \end{equation*}
for $d\in\D$, $m\in\M$ and $n\in\N$, obeying a pentagon coherence condition.

The \textit{relative Deligne product} of $\M_\subD$ and ${}_\subD\N$ is a $\kk$-linear category $\M\btD\N$ equipped with a right exact 
$\D$-balanced functor $\btD\colon \M\times\N \longrightarrow
\M\btD\N$ such that for every $\kk$-linear category $\Y$ the functor
\begin{equation*}
  \Rex(\M\btD\N,\Y) \xsim
  {\rm Bal}^\re(\M{\times}\N,\Y)\,,\qquad
  F\longmapsto F \,{\circ}\;\btD
\end{equation*}
is an equivalence of categories, where $\Rex$ denotes the category of right exact functors and ${\rm Bal}^\re$ the category of right exact balanced functors.

An exact module category ${}_\subD\M$ becomes an exact $(\D,\overline{\D_\subM^*}\,)$-bimodule category with action of the dual tensor category $\D_\subM^*$ given by module functor evaluation \cite[Lemma 7.12.7]{etingof2016tensor}. The category of $\D$-module functors
$\overline{\M}\coloneqq \Fun_\D(\M,\D)$ is naturally endowed with the structure of an exact $(\,\overline{\D_\subM^*},\D)$-bimodule category via
\begin{equation*}
F\act H \actr d\coloneqq H\circ F (-)\otimes d
\end{equation*}
for $F\in\D_\subM^*$, $H\in\overline{\M}$ and $d\in\D$. There are bimodule equivalences
\begin{equation}\label{eq:M_invertible}
\M\boxtimes_{\D_{\subM}^*}\overline{\M} \xsim \D \qquad \text{ and }\qquad
\overline{\M}\btD \M \xsim \D_{\subM}^*
\end{equation}
of $\D$-bimodule categories and of $\D_{\subM}^*$-bimodule categories, respectively \cite[Prop.\ 2.4.10]{douglas2018dualizable}.

Let $\M$ be a $\C$-module category. For an object $m\in\M$ the action functor $-\act m\colon \C\to\M$ is an exact functor and therefore it admits a right adjoint $\uHom_{\,\subM}^{\supC}(m,-)$, i.e.\ there are natural isomorphisms
\begin{equation*}
\Hom_\M(c\act m,n) \xsim \Hom_\C(c,\uHom_{\,\subM}^{\supC}(m,n))
\end{equation*}
for $c\in\C$ and $m,n\in\M$. This extends to a functor left exact in both entries  
\begin{equation*}
\uHom_{\,\subM}^{\supC}(-,-)\Colon\M\opp\times\M \longrightarrow \C
\end{equation*}
called the \emph{internal Hom functor} of ${}_\subC\M$. A module category is exact iff its internal Hom is an exact functor \cite[Cor.\ 7.9.6 \& Prop.\ 7.9.7]{etingof2016tensor}.

%%%%%%%%%%%%%%%%%%%%%%%%%%%%%%%%%%%%%%%%%%%%%%%%%%%%%%%%%%%%%%%
%%%%%%%%%%%%%%%%%%%%%%%%%%%%%%%%%%%%%%%%%%%%%%%%%%%%%%%%%%%%%%%
%%%%%%%%%%%%%%%%%%%%%%%%%%%%%%%%%%%%%%%%%%%%%%%%%%%%%%%%%%%%%%%

\subsection{Nakayama functors and relative Serre functors}
Any finite $\kk$-linear category $\X$ becomes a left $\Vect$-module category via the action $\btr$ defined by:
\[ \Hom_{\kk}(V,\Hom_{\X}(x,y)) \cong  \Hom_{\X}(V\btr x,y) ,\quad (V\in\Vect, \; x,y\in\X)\,. \]
The left and right exact Nakayama functors of $\X$ are defined as the image of the identity functor under the Eilenberg-Watts correspondence \cite[Def. 3.14]{fuchs2020eilenberg}, explicitly given by the (co)ends
\begin{equation*}
\dN_{\subX}^l (x) \coloneqq \int_{y\in \X} \Hom_{\X}(y,x) \btr y, \hspace{1cm}\dN^r_{\subX} (x) \coloneqq\int^{y\in \X} \Hom_{\X}(x,y)^* \;\btr y\,.
\end{equation*}
The right exact Nakayama functor is equipped with a family of natural isomorphisms
\begin{equation}  \label{eq:Nakayama_twisted_functor}
  \mathfrak{n}_F \Colon \dN^r_\subY\circ F\xRightarrow{\;\sim~}F^{\rra} \circ \dN^r_\subX
\end{equation}
for every right exact functor $F\colon\X\longrightarrow\Y$ admitting a double-right adjoint $F^{\rra}$ that is exact. There are similar natural isomorphisms involving the left exact Nakayama functor.

Given a finite multitensor category $\C$, the Nakayama functors of the $\kk$-linear category underlying $\C$ determine \textit{distinguished invertible objects} \cite{etingof2004analogue}
\begin{equation}\label{eq:DIO}
D_{\subC}\coloneqq\dN_{\subC}^l(\unit) \quad \text{and} \quad D_{\subC}^{-1}\coloneqq\dN_{\subC}^r(\unit).
\end{equation}
These come with a monoidal natural isomorphism
\begin{equation}\label{eq:Radford-C}
\mathcal{R}_\subC^{}\Colon D_\subC^{}\otimes-
\xRightarrow{\;\sim~} (-){\rv\rv\rv\rv}\otimes D_\subC
\end{equation}
called the \emph{Radford isomorphism}. A finite multitensor category $\C$ is \textit{unimodular} iff $D_\subC$ is isomorphic to the monoidal unit $\unit$.

Given a $\C$-module category $\M$, the natural isomorphisms \eqref{eq:Nakayama_twisted_functor} endow the Nakayama functor of $\M$ with the structure of a twisted module functor of the form 
\begin{equation} \label{eq:module_Nakayama_twisted_functor}
\fn_{c,m}\Colon  \dN^r_\subM(c\act m)\xsim \ldd c\act\dN^r_\subM(m).
\end{equation}
Similarly, the left exact Nakayama functor becomes a twisted module functor with the appropriate appearance of double-duals.

Let $\C$ be a finite tensor category and $\M$ a left $\C$-module category. A (right)  \textit{relative Serre functor} of $\M$ \cite[Def.\,4.22]{fuchs2020eilenberg} is an endofunctor $\dS_\subM^\supC:\M\longrightarrow\M$ equipped with a natural isomorphism
\begin{equation*} 
\phi_{m,n}\Colon\uHom_{\,\subM}^\supC(m,n){\rv} \xsim \uHom_{\,\subM}^\supC(n,\Se_\subM^\supC(m)), \qquad (m,n\in\M)\,.
\end{equation*}
A relative Serre functor $(\dS_\subM^\supC,\phi)$ exists iff $\M$ is exact \cite{fuchs2020eilenberg}. In this case, it is unique up to a unique isomorphism \cite[Lem.\ 3.5]{shimizu2023relative} and it is an equivalence of categories. The (left) relative Serre functor \cite[Def.\,4.22]{fuchs2020eilenberg}, denoted by $\overline{\dS}_\subM^\supC \colon\M\longrightarrow\M$, serves as quasi-inverse of $\dS_\subM^\supC$.
Relative Serre functors come with a twisted module functor structure, i.e. there are natural isomorphisms
\begin{equation}\label{eq:Serre_twisted1}
    \fs_{c,m}\Colon \Se_\subM^\supC(c\act m)\xsim c{\dd} \act \Se_\subM^\supC(m)
\end{equation}
for $c\in\C$ and $m\in\M$, obeying an appropriate pentagon axiom. Furthermore,
given a $\C$-module functor $H\colon \M\longrightarrow\N$ there is a natural isomorphism
\begin{equation}\label{eq:Serre_twisted}
    \fs_{H}\Colon \Se_\subN^\supC\circ H \xRightarrow{\;\sim~}H^\rra\circ\Se_\subM^\supC
\end{equation}
of twisted module functors.

The relative Serre functors of an exact $(\C,\D)$-bimodule category $\M$ are related to its Nakayama functors by the isomorphisms
\begin{equation}
  D_{\subC}\act\dN^r_\subM \cong \Se_\subM^\supC \qquad\text{and}\qquad
  D_{\subC}^{-1}\act\dN^l_\subM \cong \lSe_\subM^\supC
  \label{eq:Nakayama_Serre}
\end{equation}
of twisted bimodule functors \cite[Thm.\,4.26]{fuchs2020eilenberg}. Similarly, there exist isomorphisms of
twisted bimodule functors for the relative Serre functors of ${}\M_\subD$. These isomorphisms lead to an extension of Radford's theorem to module categories:

\begin{theorem}\cite[Cor.\,4.16]{spherical2025}
\label{thm:mod_Radford}
Let ${}_\subC\M$ be an exact module category. There exists a natural isomorphism
\begin{equation}
  \mathcal{R}_\subM^{}\Colon D_\subC^{}\act
  {-}\xRightarrow{~\cong~\,} \Se_\subM^\supC\circ\Se_\subM^\supC \actr D_{\C_\subM^*}= D_{\C_\subM^*}\circ \;\Se_\subM^\supC\circ\Se_\subM^\supC 
  \label{eq:Radford_mod}
\end{equation}
of twisted bimodule functors.
\end{theorem}

%%%%%%%%%%%%%%%%%%%%%%%%%%%%%%%%%%%%%%%%%%%%%%%%%%%%%%%%%%%%%%%
%%%%%%%%%%%%%%%%%%%%%%%%%%%%%%%%%%%%%%%%%%%%%%%%%%%%%%%%%%%%%%%
%%%%%%%%%%%%%%%%%%%%%%%%%%%%%%%%%%%%%%%%%%%%%%%%%%%%%%%%%%%%%%%
%%%%%%%%%%%%%%%%%%%%%%%%%%%%%%%%%%%%%%%%%%%%%%%%%%%%%%%%%%%%%%%
%%%%%%%%%%%%%%%%%%%%%%%%%%%%%%%%%%%%%%%%%%%%%%%%%%%%%%%%%%%%%%%
%%%%%%%%%%%%%%%%%%%%%%%%%%%%%%%%%%%%%%%%%%%%%%%%%%%%%%%%%%%%%%%
%%%%%%%%%%%%%%%%%%%%%%%%%%%%%%%%%%%%%%%%%%%%%%%%%%%%%%%%%%%%%%%
%%%%%%%%%%%%%%%%%%%%%%%%%%%%%%%%%%%%%%%%%%%%%%%%%%%%%%%%%%%%%%%
%%%%%%%%%%%%%%%%%%%%%%%%%%%%%%%%%%%%%%%%%%%%%%%%%%%%%%%%%%%%%%%

\section{\Frob functors}\label{sec:perfect-tensor}
A $\kk$-linear functor with isomorphic adjoints is called a Frobenius functor. In the case of a tensor functor, its adjoint functors are endowed with additional structure, namely the structure of bimodule functors. In this section, we introduce the notion of \Frob tensor functors, which account for the additional data. We provide a list of characterizations and show some relations to Frobenius monoidal functors.

\subsection{Perfect tensor functors and relative modular objects}
The relative modular function of a Hopf algebra extension studied in \cite{fischman1997frobenius} finds a categorical analogue in the notion of the relative modular object of a tensor functor defined in \cite{shimizu2017relative}. Under certain conditions, this object measures the failure to which a tensor functor preserves the distinguished invertible object \eqref{eq:DIO} and the Radford isomorphism of a finite tensor category. We first need the notion of a perfect tensor functor:
\begin{definition}\cite[\S2.1]{bruguieres2014centralexact}
A tensor functor $F\colon\C\!\longrightarrow \!\D$ between multitensor categories is called \textit{perfect} iff it admits a right adjoint that is exact.
\end{definition}
Being perfect is a property of a tensor functor. In the finite case, this notion has several characterizations listed in the lemma that follows.
\begin{lemma}\cite[Lemma~4.3]{shimizu2017relative}\label{lem:perfect_functor}
Let $\C$ and $\D$ be finite multitensor categories and $F\colon\C\!\longrightarrow \!\D$ a tensor functor.
Then the following are equivalent:
\begin{enumerate}[{\rm (i)}]
\item $F$ admits a double-right adjoint.
\item $F$ admits a double-left adjoint.
\item $F$ preserves projective objects, i.e. $F(p)\in\D$ is projective for every projective $p\in\C$.
\item There exists an object $\chi_{\subF} \in \D$
such that $F^\la\cong F^\ra(-\otimes \chi_{\subF})\cong F^\ra(\chi_{\subF}\otimes-)$ as $\kk$-linear functors.
\item There exists $c\in\C$ such that $F(c)\in\D$ is projective.
\item $F$ is perfect.
\end{enumerate}
\end{lemma}
\begin{proof}
The equivalence (i)$\iff$(ii)$\iff$(iii)$\iff$(iv) is proven in \cite[Lemma~4.3]{shimizu2017relative} where it is also shown that the object $\chi_{\subF}$ can be realized as $F^\rra(\unit)$ (note that $\chi_{\subF}$ of this paper is $\chi_{\subF}^*$ of \cite{shimizu2017relative}). As $\C$ is finite and therefore admits projective objects, (iii)$\implies$(v) is immediate. Using the same argument as in \cite[Lemma~2.4]{etingof2017exact}, (v)$\implies$(iii) follows. Finally, since $\C$ and $\D$ are finite, (vi) is equivalent to (i).
\end{proof}

A first consequence of Lemma \ref{lem:perfect_functor} is that surjective tensor functors are perfect, since they preserve projective objects \cite[Thm.\ 2.5]{etingof2004finite}. The converse is not true: for instance, consider any tensor functor $F$ between semisimple tensor categories that is not surjective (like $\Vect_H\subset \Vect_G$ for a proper subgroup $H\subset G$), then $F$ preserves projectives, but is not surjective. In fact, any tensor functor from a non-semisimple tensor category to a semisimple tensor category is perfect. Also notice that not every tensor functor is perfect: for a finite tensor category $\C$ that is not semisimple, the tensor functor $\Vect\to \C$ given by $\kk\mapsto\unit$ does not preserve projective objects.

\begin{definition}\cite[Def.\ 4.4]{shimizu2017relative}
The object \(\chi_{\subF} \in\D\) from Lemma \ref{lem:perfect_functor}(iv) is called \textit{the relative modular object of $F$}.
\end{definition}

\begin{proposition}\label{prop:relative_mod_obj_properties}\cite{shimizu2017relative}
Let $F\colon\C\longrightarrow \D$ be a perfect tensor functor. The relative modular object of $F$ obeys the following properties:
\begin{enumerate}[{\rm (i)}]

\item $\chi_{\subF}$ is an invertible object in $\D$.
\item There is a natural isomorphism $F^\rra\cong F(-)\ot \chi_{\subF}$ of $\kk$-linear functors.
\item There is a natural isomorphism $F^\lla\cong F(-)\ot \chi_{\subF}^{-1}$ of $\kk$-linear functors.
\item There is an isomorphism $\chi_{\subF}\cong F^\rra(\unit_\subC)$.
\item There is an isomorphism $\chi_{\subF}\cong F(D_\subC)\otimes D_\subD^{-1}$.
\end{enumerate}
\end{proposition}
\begin{proof}
Part (i) is proven in \cite[Prop.~4.7(2)]{shimizu2017relative}.
By taking right adjoints to the isomorphism $F^\la\cong F^\ra(-\otimes \chi_{\subF})$ from Lemma \ref{lem:perfect_functor}(iv), we obtain that
\[ F \cong (F^\la)^\ra \cong \left( F^\ra \circ (-\otimes \chi_{\subF}) \right)^\ra \cong (-\otimes \chi_{\subF})^\ra \circ F^\rra \cong F^\rra(-)\otimes \chi_{\subF}^{-1}.\]
Tensoring by $\chi_{\subF}$ on both sides yields (ii). 
Part (iii) is proved in a similar manner. Lastly, since $F(\unit_\subC)\cong\unit_\subD$, part (iv) follows by (ii).
According to \cite[Rem.\ 4.17]{fuchs2020eilenberg} $F^\lla(D_\subC)\cong D_\subD$. The isomorphism in (v) is then obtained by using (iii).
\end{proof}

%%%%%%%%%%%%%%%%%%%%%%%%%%%%%%%%%%%%%%%%%%%%%%%%%%%%%%%%%%%%%%%
%%%%%%%%%%%%%%%%%%%%%%%%%%%%%%%%%%%%%%%%%%%%%%%%%%%%%%%%%%%%%%%
%%%%%%%%%%%%%%%%%%%%%%%%%%%%%%%%%%%%%%%%%%%%%%%%%%%%%%%%%%%%%%%

\begin{definition}\cite{caenepeel1997doi}
\label{def:Frobenius_functor}
A ($\kk$-linear) functor $F\colon\X\to\Y$ is called \textit{Frobenius} if it admits a two-sided adjoint, i.e. it admits both a left adjoint and a right adjoint and $F^{\la}\cong F^{\ra}$.
\end{definition}

Given an adjoint pair $F\dashv G$, note that $F$ is Frobenius if and only if $G$ is Frobenius. Another immediate consequence of Definition \ref{def:Frobenius_functor} is that a tensor functor $F\colon\C\longrightarrow\D$ that is Frobenius must be perfect as well: $F$ admits a double-right adjoint, namely $F^{\rra}\cong F$. 
However, the converse is not true as shown by the following example.
\begin{example}
For a finite tensor category $\C$, the forgetful functor $U\colon\Z(\C)\rightarrow\C$ is surjective and therefore perfect. However, $U$ is Frobenius if and only if $\C$ is unimodular \cite[Thm.\ 4.10]{shimizu2016unimodular}.
\end{example}

The following result relates the property of $F$ being Frobenius and $F^\ra$ being Frobenius monoidal.
\begin{proposition}\label{prop:Frob-mon-implies-Frobenius}
Let $F\colon\C\longrightarrow\D$ be a tensor functor between tensor categories admitting a right adjoint $F^\ra\colon\D\longrightarrow\C$.
\begin{enumerate}[{\rm (i)}]
    \item If the lax monoidal functor $F^{\ra}$ admits a Frobenius monoidal structure, then $F$ is Frobenius. 
    \item If $F^\ra$ is exact and faithful, then $F$ being Frobenius implies that the induced lax and oplax monoidal structures turn $F^\ra$ into a Frobenius monoidal functor.
\end{enumerate}
\end{proposition}
\begin{proof}
(i) The left and right adjoints of $F$ are related by $F^{\la}(d) \cong \rv (F^{\ra}(d\rv))$ \cite[Lem.~4.1]{shimizu2017relative}. As $F^\ra$ is Frobenius monoidal it preserves duals. Thus, there is a natural isomorphism
\[ F^{\la}(d) \cong \rv(F^{\ra}(d\rv)) \cong \rv(F^{\ra}(d)\rv) \cong F^{\ra}(d),\]
thereby proving that $F$ is Frobenius. 

(ii) Notice that $F$ being Frobenius implies that $F^\ra(\unit_\subD)$ is a Frobenius algebra in $\C$: 
according to \cite[Prop.\ 6.1]{bruguieres2011exact}, $\D$ can be identified with the category $\mathrm{mod}_A(\C)$ with $A=F^\ra(\unit_\subD)$. Under this identification, $F^\ra\colon\mathrm{mod}_A(\C)\longrightarrow\C$ is the forgetful functor. By \cite[Lemma 2.3]{shimizu2016unimodular}, we get that $A$ is a Frobenius algebra iff $F^\ra$ is a Frobenius functor iff $F$ is a Frobenius functor. Then, as explained in \cite[\S3.1]{yadav2024frobenius}, we can apply \cite[Prop.~4.5]{balan2017hopf} to obtain that $F^\ra$ is endowed with a Frobenius monoidal structure.
\end{proof}

\begin{question}
    Note that by \cite[Lem.~3.1]{bruguieres2011exact} the assumption that $F^\ra$ is faithful is equivalent to asking that $F$ is \textit{dominant}, that is, every object in $\D$ is a quotient of $F(c)$ for some $c\in\C$.
Does the statement in Proposition \ref{prop:Frob-mon-implies-Frobenius} $($ii$\,)$ hold without assuming that $F$ is dominant?
\end{question}

\begin{proposition}
\label{prop:frobenius_uni}
Given a perfect tensor functor $F\colon\C\!\longrightarrow\!\D$, the following hold:
\begin{enumerate}[{\rm (i)}]
    \item $F$ is Frobenius if and only if $\chi_{\subF}\cong\unit_\subD$.
    \item If $\C$ is unimodular, then $F$ is Frobenius if and only if $\D$ is unimodular. 
    \item If $\C$ is semisimple, then $F$ is Frobenius.
\end{enumerate}
\end{proposition}
\begin{proof}
The first statement is proven in \cite[Prop.\ 4.7(i)]{shimizu2017relative}. If $\C$ is unimodular, by Proposition~\ref{prop:relative_mod_obj_properties}(v), $D_\subD^{-1}\cong\chi_\subF$. Then, (ii) follows from (i). 
For (iii), note that when $\C$ is semisimple, $\unit_\subC$ is projective. As $F$ preserves projectives, $\unit_\subD \cong F(\unit_\subC)$ is projective. Consequently, $\D$ is semisimple. It follows that both $\C$ and $\D$ are unimodular. Therefore, $\chi_{\subF}\cong\unit_\subD$ and by part (i), $F$ is Frobenius. 
\end{proof}
\begin{remark}
The unimodularity of $\C$ in Proposition \ref{prop:frobenius_uni}(ii) is crucial. As  Example \ref{ex:tensor-Frob} shows, it might be the case that $F$ is Frobenius and $\D$ is unimodular, but $\C$ is not unimodular.
\end{remark}

Perfect, Frobenius and surjective are properties of a tensor functor that behave well under composition of tensor functors.
\begin{proposition}
\label{prop:comp_rel_obj}
Let $F\colon\C\!\longrightarrow\!\D$ and $G\colon\B\!\longrightarrow\!\C$ be tensor functors, we have that:
\begin{enumerate}[\rm (i)]
\item If $F$ and $G$ are perfect, their composition $F\circ G:\B\to\D$ is perfect and its relative modular object is given by
\begin{equation*}
\chi_{\scriptscriptstyle F\circ G}\cong F(\chi_{\scriptscriptstyle G})\otimes\chi_{\subF}\,.
\end{equation*}
\item If $F$ and $G$ are Frobenius, the composition $F\circ G$ is Frobenius.
\item If $F$ and $G$ are surjective, the composition $F\circ G$ is surjective.
\end{enumerate}
\end{proposition}
\begin{proof}
If both $F$ and $G$ admit double-adjoints, then their composition also admits double-adjoints and from Lemma \ref{lem:perfect_functor} it follows that $F\circ G$ is perfect showing (i). We can compute relative modular objects in terms of double-adjoints according to Proposition \ref{prop:relative_mod_obj_properties}(ii):
\begin{equation*}
\chi_{\scriptscriptstyle F\circ G}\cong (F\circ G)^\rra(\unit_{\scriptscriptstyle \B}) 
\cong (F^\rra\circ G^\rra)(\unit_{\scriptscriptstyle \B})
\cong  F^\rra(\chi_{\scriptscriptstyle G})\cong F(\chi_{\scriptscriptstyle G})\otimes\chi_{\subF}\,. 
\end{equation*}
Assertion (ii) follows from (i) and Proposition \ref{prop:frobenius_uni}(i). 
Now, in order to prove (iii), notice that since $F$ is surjective, any $z\in\D$ is a subquotient of $F(y)$ for some $y\in \C$. Also, since $G$ is surjective, $y$ is a subquotient of $G(x)$ for some $x\in \B$. 
Now, left exactness of $F$ implies that it preserves kernels (hence subobjects). Also, right exactness of $F$ implies that it preserves cokernels (hence quotient objects). Hence, $F$ preserves subquotients. Thus, $F(y)$ is a subquotient of $FG(x)$. Since the subquotient of a subquotient  of some object $p$ is also a subquotient of $p$ \cite[Ex.\ 1.3.6]{etingof2016tensor}, we conclude that $z$ is a subquotient of $FG(x)$.
\end{proof}

%%%%%%%%%%%%%%%%%%%%%%%%%%%%%%%%%%%%%%%%%%%%%%%%%%%%%%%%%%%%%%%
%%%%%%%%%%%%%%%%%%%%%%%%%%%%%%%%%%%%%%%%%%%%%%%%%%%%%%%%%%%%%%%
%%%%%%%%%%%%%%%%%%%%%%%%%%%%%%%%%%%%%%%%%%%%%%%%%%%%%%%%%%%%%%%

\subsection{Center of a perfect tensor functor}\label{subsec:center}
In this section, we recall the notion of the center $\Z(F)$ of a tensor functor $F$, and describe its relative modular object in case it is perfect.

Let $\C$ be a finite tensor category and $\M$ a $\C$-bimodule category. The \textit{center of $\M$}, denoted by $\Z(\M)$, is the category whose objects are pairs $(m,\gamma)$, where $m\in\M$ and $\gamma_c\colon m\tl c\rightarrow c\tr m$ is a natural isomorphism obeying a hexagon axiom. The hom spaces of $\Z(\M)$ are given by
\[\Hom_{\Z(\M)}\left( (m,\gamma),(m',\gamma') \right) 
\coloneqq 
\left\{ f\in\Hom_{\M}(m,m') \, | \, (f\otimes \id_{c})\circ \gamma_c = \gamma'_c \circ(\id_{c}\otimes f)\, \; \forall\, c\in\C \right\}. \] 

The center construction extends to a strict pseudo-functor \cite[\S3.6]{shimizu2020further}, \cite[Lem.~4.3]{flake2024projection} 
\begin{equation}\label{eq:center_cons}
\Z\Colon \mathrm{\textbf{Bimod}}(\C)\longrightarrow \mathrm{FinCat}_\kk,\qquad \M\longmapsto \Z(\M)
\end{equation}
from the $2$-category of $\C$-bimodule categories, bimodule functors and bimodule natural transformations to the $2$-category of finite $\kk$-linear categories, $\kk$-linear functors and natural transformations.

Let $F,\, G\colon\C\longrightarrow\D$ be tensor functors between finite tensor categories. We use the notation $\FDG$ to denote the $\C$-bimodule category $\D$ with left and right actions given by 
\[c\triangleright d \triangleleft c' := F(c)\otimes d \otimes G(c') .\]
The \textit{$F$-centralizer of $\D$} is defined as the center $\ZFD$ of the $\C$-bimodule category $\FDF$. Explicitly, its objects are pairs $(d,\gamma)$ consisting of an object $d\in\D$ and a \emph{$F$-half-braiding} $\gamma$, i.e.\
a natural isomorphism $\gamma_{c} \colon d\,{\otimes}\, F(c)\rightarrow F(c)\,{\otimes}\, d$ 
for $c\in\C$, obeying the respective hexagon axiom. 
The $F$-centralizer of $\D$ is endowed with a monoidal structure defined via 
\[(d,\gamma)\otimes (d',\gamma')=(d\otimes d', \rho)\] 
where the $F$-half-braiding $\rho_c$ is the composition
\[d\otimes d' \otimes F(c)\xrightarrow{\id\otimes\gamma'_c} d\otimes F(c)\otimes d' \xrightarrow{\gamma_c\otimes\id} F(c)\otimes d\otimes d'\]
and the monoidal unit $\unit_{\Z(\FDF)}$ is given by $\unit_\subD$ together with the trivial half-braiding.
Moreover, $\Z(\FDF)$ is rigid \cite[Thm.~3.3]{majid1991representations} and finite \cite[Thm.~3.8]{shimizu2023ribbon}, therefore $\Z(\FDF)$ is a finite tensor category. By considering the identity tensor functor $F=\id_{\,\subC}$, we recover the \textit{Drinfeld center} $\Z(\C)$ of $\C$, which is further endowed with the structure of a braiding.

Now, the functor $F$ can be seen as a $\C$-bimodule functor $F\colon \C\longrightarrow\FDF$. The center construction \eqref{eq:center_cons} provides a $\kk$-linear functor
\begin{equation*}
\Z(F)\Colon\Z(\C)\longrightarrow \ZFD\,, \qquad (c,\nu) \longmapsto (F(c),\gamma)
\end{equation*}
where $\gamma_{c'}$ is given by the composition
\[ F(c)\otimes F(c') \xrightarrow{\;F_2 \;} F(c\otimes c') \xrightarrow{F(\nu_{c'})} F(c'\otimes c) \xrightarrow{\;F_2^{-1}\;} F(c')\otimes F(c)\,. \]
In addition, the strong monoidal structure of $F$ induces a strong monoidal structure on $\Z(F)$. We summarize the properties of $\Z(F)$ in the following lemma.

\begin{lemma}
Let $F\colon\C\longrightarrow\D$ be a tensor functor between finite tensor categories.
\begin{enumerate}[\rm (i)]
    \item $\Z(\FDF)$ is a finite tensor category.
    \item $\Z(F)\colon\Z(\C)\longrightarrow \Z(\FDF)$ is a tensor functor.
    \item $\Z(F^{\ra})$ is right adjoint to $\Z(F)$.
\end{enumerate}
\end{lemma}
\begin{proof}
Proof of part (i) was given above. Since both $\C$ and $\D$ are finite and $F$ is exact, $F$ admits a right adjoint $F^\ra$, which is equipped with a lax monoidal structure making the adjunction $F\dashv F^\ra$ monoidal. By \cite[Prop.~4.8]{flake2024projection}, $\Z(F^\ra)$ is right adjoint to $\Z(F)$, and the adjunction $\Z(F)\dashv \Z(F^\ra)$ is monoidal; this establishes (iii). For part (ii), note that $\Z(F)$ is a strong monoidal functor between rigid categories and admits a right adjoint, so it also has a left adjoint. Since both $\Z(\C)$ and $\Z(\FDF)$ are finite, it follows that $\Z(F)$ is exact, and hence a tensor functor.
\end{proof}

\begin{proposition}{\cite[Corollary~6.9]{shimizu2022nakayama}}\label{prop:Z(F)-rmo}
Let $F\colon\C\longrightarrow\D$ be a perfect tensor functor. The tensor functor $\Z(F)$ is perfect. Moreover, $\chi_{\subF}$ is endowed with a $F$-half-braiding
\begin{equation}\label{eq:chi_F-braiding}
\sigma_c\Colon\chi_{\subF}\otimes F(c)\xsim F(c)\otimes \chi_{\subF}\,,\qquad \text{for } c\in\C\,,
\end{equation}
such that the object $\chi_{\scriptscriptstyle{\Z(F)}}:=(\chi_{\subF},\sigma)$ in $\ZFD$ is the relative modular object of $\Z(F)$.
\end{proposition}

The pair $\chi_{\scriptscriptstyle{\Z(F)}}=(\chi_{\subF},\sigma)$ measures the extent to which the Radford isomorphism is preserved under $F$, in the following sense:
\begin{proposition}{\cite[Thm.\ 4.4]{shimizu2023ribbon}}\label{prop:rel-mod-obj}
Let $F\colon\C\longrightarrow \D$ be a perfect tensor functor. There exists an isomorphism 
\begin{equation}\label{eq:relative-distinguished-iso}
F(D_\subC)\xsim \chi_{\subF} \otimes D_\subD
\end{equation}
fulfilling the commutativity of the following diagram
\begin{equation}\label{eq:chi-central-structre}    
    \begin{tikzcd}[row sep = 2em, column sep = 6em]
        F(D_\subC\otimes c)\ar[d,swap,"F(\mathcal{R}_\subC)"]\ar[r,"\cong"]
        &F(D_\subC)\otimes F(c)\ar[r,"\eqref{eq:relative-distinguished-iso} "]
        &\chi_{\subF} \otimes D_\subD\otimes F(c)\ar[d,"\id\otimes \mathcal{R}_\subD"]
        \\
        F(c{\rv\rv\rv\rv}\otimes D_\subC)\ar[d,"\cong",swap] & & \chi_{\subF}\otimes F(c){\rv\rv\rv\rv}\otimes D_\subD \ar[d,"\cong"]
        \\
        F(c{\rv\rv\rv\rv})\otimes F(D_\subC)\ar[r," \eqref{eq:relative-distinguished-iso} "] &   F(c{\rv\rv\rv\rv})\otimes\chi_{\subF} \otimes D_\subD 
        &\ar[swap,l,"\sigma_{\scriptscriptstyle c{{\rv^4}}}\otimes\id"] \chi_{\subF}\otimes F(c{\rv\rv\rv\rv})\otimes D_\subD
    \end{tikzcd}
\end{equation}
where $\mathcal{R}_\subC$ $($resp. $\mathcal{R}_\subD$$)$ is the Radford isomorphism, and $\sigma$ is the isomorphism \eqref{eq:chi_F-braiding}.
\end{proposition}
\begin{proof}
A similar statement is proven in \cite[Thm.\ 4.4]{shimizu2023ribbon}, where an isomorphism of the 
form $F(D_\subC)\xsim D_\subD\otimes \chi_{\subF} $ is considered instead and the diagram \eqref{eq:chi-central-structre} is appropriately adapted. In fact, we can obtain diagram \eqref{eq:chi-central-structre} from the diagram in \cite[Thm.\ 4.4]{shimizu2023ribbon} by means of the Radford isomorphism $(\mathcal{R}_\subD)_\chi\colon D_\subD\otimes \chi_{\subF}\cong \chi_{\subF}\otimes D_\subD$. It is enough to check that the following diagram commutes
\begin{equation*}
\begin{tikzcd}[row sep = 4em, column sep = 4em]
\chi_{\subF} \otimes D_\subD\otimes F(c)\ar[r,"\id\otimes \mathcal{R}_\subD"]&
\chi_{\subF}\otimes F(c{\rv\rv\rv\rv})\otimes D_\subD
\ar[r,"\sigma\otimes\id"]\ar[ld," \mathcal{R}_\subD"]
&F(c{\rv\rv\rv\rv})\otimes\chi_{\subF} \otimes D_\subD \ar[ld," \mathcal{R}_\subD",swap]\\
D_\subD\otimes\chi_{\subF} \otimes  F(c)
\ar[u," \mathcal{R}_\subD\otimes\id"]
\ar[r,"\id\otimes \sigma",swap]&D_\subD\otimes F(c)\otimes\chi_{\subF}
\ar[r,"\mathcal{R}_\subD\otimes\id",swap]&
F(c{\rv\rv\rv\rv})\otimes D_\subD\otimes\chi_{\subF}\ar[u,"\id\otimes \mathcal{R}_\subD",swap]  
\end{tikzcd}
\end{equation*}
where the identification $F(c{\rv\rv\rv\rv})\cong F(c){\rv\rv\rv\rv}$ is implicit. The commutativity of the triangles follows from the monoidality of the Radford isomorphism and the remaining middle square commutes due to naturality of the $F$-half-braiding $\sigma$.
\end{proof}

%%%%%%%%%%%%%%%%%%%%%%%%%%%%%%%%%%%%%%%%%%%%%%%%%%%%%%%%%%%%%%%
%%%%%%%%%%%%%%%%%%%%%%%%%%%%%%%%%%%%%%%%%%%%%%%%%%%%%%%%%%%%%%%
%%%%%%%%%%%%%%%%%%%%%%%%%%%%%%%%%%%%%%%%%%%%%%%%%%%%%%%%%%%%%%%

\subsection{The properties of \Frob functors}
A tensor functor $F\colon\C\longrightarrow\D$ can be viewed as a $\C$-bimodule functor $F\colon\C\longrightarrow\FDF$. Moreover, it admits left and right adjoints that are also $\C$-bimodule functors \cite[Cor.\ 2.13]{douglas2019balanced}. This motivates the following stronger version of Definition \ref{def:Frobenius_functor} that accounts for the additional structure. 

\begin{definition}\label{def:-Frob}
A tensor functor $F\colon \C\longrightarrow\D$ between finite multitensor categories is called \Frob iff there exists a $\C$-bimodule natural isomorphism $F^\la\cong F^\ra\colon \FDF \longrightarrow \C$.
\end{definition}

\begin{example}
Consider a finite tensor category $\C$ and an action $T\colon G\to \Aut_\otimes(\C)$ of a finite group $G$ on $\C$. The finite tensor category of fixed points $\C^G$ or \textit{equivariantization} \cite[Def.\ 2.7.2]{etingof2016tensor} comes with a forgetful functor
\begin{equation}
    \mathtt{forg}\Colon \C^G\longrightarrow \C,\quad (X, u)\longmapsto X\,.
\end{equation}
As pointed out in \cite[Lemma 4.6]{drinfeld2009braided} the induction functor
\begin{equation}
    \mathtt{Ind}\Colon \C\longrightarrow \C^G,\quad Y\longmapsto \oplus_{g\in G}\, T_g(Y)
\end{equation}
serves as left and right adjoint to $\mathtt{forg}$. Moreover, one can check that the $\C^G$-bimodule structures on $\mathtt{Ind}$ induced by the adjunctions $\mathtt{forg}\dashv \mathtt{Ind}$ and $ \mathtt{Ind}\dashv\mathtt{forg}$ agree and are given by
\begin{align*}
    (X,u)\tr \Ind(Y) \tl (Z,v)=\bigoplus_{g\in G}\,X\otimes T_g(Y)\otimes Z&\xrightarrow{\oplus_{g\in G}u_g\otimes\id\otimes v_g} \bigoplus_{g\in G}\,T_g(X)\otimes T_g(Y)\otimes T_g(Z)\\
    &\cong \bigoplus_{g\in G}\,T_g(X\otimes Y\otimes Z)=\Ind(X\otimes Y\otimes Z)\,.
\end{align*}
Hence, $\mathtt{forg}$ is a \Frob functor.
\end{example}

The notion of a \Frob tensor functor $F\colon\C\longrightarrow\D$ is an instance of an ambi-adjunction (see \cite[Def.\ 2.7]{flake2024frobenius1} for a definition) internal to the $2$-category $\mathrm{\textbf{Bimod}}(\C)$.
Clearly, a bimodule natural isomorphism $F^\la\cong F^\ra$ is in particular a natural isomorphism. Therefore, any \Frob functor is Frobenius and thus perfect. The converse does not hold, an example of a (surjective) Frobenius functor that is not \Frob can be found in Example \ref{ex:tensor-Frob}. It is worth pointing out that \Frob is a property of a tensor functor $F$, in the sense that it asks for the existence and not for a specific choice of bimodule isomorphism between the left and right adjoint functors of $F$. 

\begin{remark}
Given \Frob functors $F\colon \C\to\D$ and $G\colon \B\to\C$, the composition $F\circ G$ is also \Frob\!\!. Indeed, if $F\cong F^\rra$ and $G\cong G^\rra$ as bimodule functors, then $F\circ G\cong F^\rra\circ G^\rra \cong (F\circ G)^\rra$ as bimodule functors. The second isomorphism is obtained by the uniqueness of adjoints in the $2$-category $\mathrm{\textbf{Bimod}}(\C)$.
\end{remark}

Another consequence of being \Frob (and thus Frobenius) is the following: as $F$ is strong monoidal, $F^\ra$ is endowed with a lax monoidal structure and $F^\la$ with an oplax monoidal structure. Using the isomorphism $F^\la\cong F^\ra$, $F^\ra$ also inherits an oplax structure. Next, directly following the results in \cite{flake2024frobenius1}, we obtain that these two structures  form a Frobenius monoidal functor.

\begin{proposition}{\cite[Thm.~3.18]{flake2024frobenius1}}\label{prop:tensor-Frob-ra-FrobMon}
Let $F\colon\C\longrightarrow\D$ be a \Frob functor. Then $F^{\ra}$ is endowed with a Frobenius monoidal structure.
\end{proposition}
\begin{proof}
Since $F$ is \Frob\!\!, $F^{\la}\cong F^{\ra}$ in the $2$-category $\mathrm{\textbf{Bimod}}(\C)$. Thus, by the implication (2)$\implies$(1) in \cite[Thm.~3.17]{flake2024frobenius1}, the Hopf operators \eqref{eq:Hopf-operators} of $F^\ra \dashv F$ and the coHopf operators \eqref{eq:coHopf-operators} of $F\dashv F^\ra$ satisfy $H^l_{d,c} = (h^l_{d,c})^{-1}$ and $H^r_{d,c} = (h^r_{d,c})^{-1}$. By \cite[Thm.~3.18]{flake2024frobenius1}, the second relation implies the claim.
\end{proof}

We show next that the notion of a \Frob tensor functor is closely related to relative modular objects. 
To make this relation precise, we need the following additional fact. The $F$-centralizer of $\D$ can be realized as the category of $\C$-bimodule functors from $\C$ to $\FDF$, denoted by $\Fun_{\C|\C}(\C, \FDF)$. Explicitly, there is an equivalence of categories \cite[Lemma 3.1]{shimizu2023ribbon}
\begin{equation}\label{eq:equi_bimod_centralizer}
\Fun_{\C|\C}(\C, \FDF)\xsim \Z(\FDF),\quad H\longmapsto (H(\unit),\gamma)
\end{equation}
where 
\[ \gamma_c\colon H(\unit)\otimes F(c)=H(\unit)\tl^{\scriptscriptstyle F} c\xsim H(c)\xsim c\trF H(\unit)=F(c)\otimes H(\unit) \] 
is given by the bimodule functor structure of $H$.

The following assignment serves as a quasi-inverse of \eqref{eq:equi_bimod_centralizer}. For $(d,\gamma)\in \Z(\FDF)$ consider the functor $F(-)\otimes d$ together with $\C$-bimodule functor structure defined via 
\begin{equation*}
F(c_1\otimes c'\otimes c_2)\otimes d\xrightarrow{~\;F_2\;~}
F(c_1)\otimes F(c')\otimes F(c_2)\otimes d\xrightarrow{~\;\gamma^{-1}\;~} F(c_1)\otimes F(c')\otimes d\otimes F(c_2)\,.
\end{equation*}

Now, using the equivalence \eqref{eq:equi_bimod_centralizer} and our previous results, we obtain the following list of characterizations of the notion of a \Frob functors.

\begin{theorem}\label{thm:tFrob-char}
Let $F\colon \C\longrightarrow\D$ be a perfect tensor functor. The following are equivalent:
\begin{enumerate}[\rm (i)]
    \item $F$ is \Frob\!\!.
    \item There is a natural isomorphism $F\cong F^\rra$ of $\C$-bimodule functors.
    \item The relative modular object $\chi_{\scriptscriptstyle{\Z(F)}}=(\chi_{\subF},\sigma)$ is trivial in $\Z(\FDF)$, i.e. there is $\lambda:\chi_{\subF}\cong\unit_\subD$ such that $\id_{F(c)}\otimes\lambda\circ\sigma_c=\lambda\otimes\id_{F(c)}$ for all $c\in \C$.
    \item The tensor functor $\Z(F)$ is Frobenius.
    \item The finite tensor category $\Z(\FDF)$ is unimodular. 
    \item The lax monoidal functor $\Z(F^{\ra})$ admits a Frobenius monoidal structure.
    \item $F$ preserves the Radford isomorphism, i.e. there exists an isomorphism $\eta\colon F(D_\subC)\xsim D_\subD$ such that the following diagram
\begin{equation}
    \begin{tikzcd}[row sep = 2.5em, column sep = 2em]
        \ar[d,swap,"F(\mathcal{R}_\subC)"]F(D_\subC\otimes c)\ar[r,"\cong"]& F(D_\subC)\otimes F(c)
        \ar[r,"\eta"]
        &D_\subD\otimes F(c)\ar[d," \mathcal{R}_\subD"]
        \\
        F(c{\rv\rv\rv\rv}\otimes D_\subC)\ar[r,"\cong",swap]& F(c{\rv\rv\rv\rv})\otimes F(D_\subC)\cong F(c){\rv\rv\rv\rv}\otimes F(D_\subC)\ar[r," \eta",swap] 
        &  F(c){\rv\rv\rv\rv}\otimes D_\subD
    \end{tikzcd}
\end{equation}
commutes for every $c\in\C$.
\end{enumerate}
\end{theorem}
\begin{proof}
Since $F$ is perfect, it admits double-adjoints. Then, by taking adjoints in $\mathrm{\textbf{Bimod}}(\C)$, there exists a $\C$-bimodule natural isomorphism $F^\la\cong F^\ra$ if and only if there exists a bimodule natural isomorphism $F\cong F^\rra$, which proves (i)$\iff$(ii). 

To prove (ii)$\iff$(iii), notice that the category equivalence \eqref{eq:equi_bimod_centralizer} assigns $(\chi_{\subF}, \sigma)$ to $F^\rra$ as defined in Proposition \ref{prop:Z(F)-rmo}. Moreover, \eqref{eq:equi_bimod_centralizer} sends $F$ to the unit of $\Z(\FDF)$. Then, $F\cong F^\rra$ in $\Fun_{\C|\C}(\C, \FDF)$ if and only if $\unit\cong (\chi_{\subF}, \sigma)$ in $\Z(\FDF)$.

Since, by Proposition~\ref{prop:Z(F)-rmo}, $(\chi_{\subF},\sigma)$ is the relative modular object of $\Z(F)$, the equivalence (iii)$\iff$(iv) follows by Proposition~\ref{prop:frobenius_uni}(i). 
Next, note that
\[ (\chi_{\subF},\sigma) = \chi_{\scriptscriptstyle{\Z(F)}} \cong F(D_{\Z(\C)}) \otimes D_{\Z(\FDF)}^{-1} \cong D_{\Z(\FDF)}^{-1},\]
where the last isomorphism holds because $\Z(\C)$ is unimodular. This implies that (iv)$\iff$(v).

(i)$\implies$(vi) follows by a direct application of \cite[Thm.~3.17]{flake2024frobenius1} and \cite[Prop.~4.15]{flake2024frobenius1}. 
The implication (vi)$\implies$(iv) follows by Proposition~\ref{prop:Frob-mon-implies-Frobenius}. Lastly, by Proposition \ref{prop:rel-mod-obj}, it follows that (iii)$\iff$(vii).
\end{proof}

Note that there is a strong monoidal functor $\mathrm{res}\colon\Z(\D)\longrightarrow\Z(\FDF)$ given by restricting the half-braiding to objects of the form $F(c)$. Using this, we obtain the lax monoidal functor:
\begin{equation}\label{eq:tilde-Z-F-ra}
    \Z(F^\ra)' \Colon \Z(\D) \xrightarrow{~\mathrm{res}~} \Z(\FDF) \xrightarrow{~\Z(F^\ra)~} \Z(\C)\,.
\end{equation}
Since $F\dashv F^{\ra}$ is a monoidal adjunction between rigid monoidal categories, by \cite[Cor.~4.11]{flake2024projection}, the functor $\Z(F^\ra)'$ becomes a braided lax monoidal functor. Then we obtain the following result which may be considered as a generalization of Proposition~\ref{prop:tensor-Frob-ra-FrobMon}.

\begin{corollary}\label{cor:tensor-Frob-Z'-Frob-mon}
Let $F\colon\C\longrightarrow\D$ be a \Frob functor. Then, the induced functor $\Z(F^{\ra})'$ from \eqref{eq:tilde-Z-F-ra} is endowed with a $($braided$)$ Frobenius monoidal structure. 
\end{corollary}
\begin{proof}
The functor $\mathrm{res}$ is strong monoidal and therefore Frobenius monoidal. Moreover, by Theorem~\ref{thm:tFrob-char}, $\Z(F^\ra)$ is Frobenius monoidal. As Frobenius monoidal functors compose, the claim follows.
\end{proof}

%%%%%%%%%%%%%%%%%%%%%%%%%%%%%%%%%%%%%%%%%%%%%%%%%%%%%%%%%%%%%%%
%%%%%%%%%%%%%%%%%%%%%%%%%%%%%%%%%%%%%%%%%%%%%%%%%%%%%%%%%%%%%%%
%%%%%%%%%%%%%%%%%%%%%%%%%%%%%%%%%%%%%%%%%%%%%%%%%%%%%%%%%%%%%%%

Let $\B$ be a braided finite tensor category and $\D$ a finite tensor category. A \textit{central tensor functor} consists of a tensor functor $F\colon \B\longrightarrow\D$ together with a braided functor $\widetilde{F}\colon \B\longrightarrow\Z(\D)$ and a monoidal natural isomorphism $F\cong U\circ \widetilde{F}$, where $U\colon\Z(\D)\to\D$ is the forgetful functor.

\begin{proposition}\label{prop:central-tensor-Frobenius} 
Let $\B$ be a braided finite tensor category, $\D$ a finite tensor category and $F\colon \B\longrightarrow\D$ a central tensor functor. Then, the following are equivalent:
\begin{enumerate}[{\rm (i)}]
\item $F$ is \Frob\!\!.
\item $F$ is Frobenius.
\item The lax monoidal functor $F^{\ra}$ admits a Frobenius monoidal structure.
\end{enumerate}
\end{proposition}
\begin{proof}
That $F$ is \Frob automatically implies that $F$ is Frobenius. To prove the converse, we make use of \cite[Thm.\ 5.6]{shimizu2023ribbon} that states that for $b\iN \B$, the associated component of the $F$-half-braiding
$\sigma_b\colon \chi_{\subF}\otimes F(b)\xsim  F(b)\otimes\chi_{\subF} $
agrees with the half-braiding of $\widetilde{F}(b)$ coming from the central structure of $F$. According to Proposition \ref{prop:frobenius_uni}, we have that $\chi_{\subF}\cong\unit$. Since the component of the half-braiding of $\widetilde{F}(b)$ associated to $\unit$ is the identity, it follows that $\sigma_b=\id_{F(b)}$. We conclude, by means of Theorem~\ref{thm:tFrob-char}, that $F$ is \Frob\!\!. This establishes (i)$\iff$(ii). Proposition~\ref{prop:tensor-Frob-ra-FrobMon} states that (i)$\implies$(iii) and Proposition~\ref{prop:Frob-mon-implies-Frobenius}(i) shows (iii)$\implies$ (ii).
\end{proof}

\begin{example}
For a finite tensor category $\C$, consider the central functor $U\colon\Z(\C)\longrightarrow\C$. Then:
\[
\C \; \text{is unimodular} \stackrel{\textnormal{\cite{shimizu2016unimodular}}}{\Longleftrightarrow} 
\;  U\;\text{is Frobenius} \; 
\stackrel{\textnormal{Prop.~\ref{prop:central-tensor-Frobenius}}}{\Longleftrightarrow} 
\; U\; \text{is \Frob}\!\!. 
\]
This generalizes an observation from the Hopf algebra case in \cite[Corollary~5.1]{flake2024frobenius2}.
\end{example}

\begin{remark}
Proposition~\ref{prop:central-tensor-Frobenius} implies that a braided tensor functor that is Frobenius is also \Frob\!. This has also been observed in \cite{flake2024frobenius1}.
\end{remark}

A semisimple finite tensor category is called a {\it fusion category}. 
Moreover, a fusion category is called {\it separable} if its global dimension is nonzero. If char$(\kk)=0$, then all fusion categories are separable according to \cite[Thm.~2.3]{etingof2005fusion}.
This definition of separability is equivalent \cite[Thm.~2.6.7]{douglas2018dualizable} to the alternate notion given in \cite{douglas2018dualizable} in terms of algebras, which can also be extended to (bi)module categories.

\begin{theorem}\label{thm:separable-implies-tensor-Frob}
If $\C$ and $\D$ are separable fusion categories and $F\colon\C\longrightarrow\D$ is a tensor functor, then $F$ is \Frob\!\!.
\end{theorem}
\begin{proof}
We have the following equivalence of categories:
\[ \Z(\FDF) \simeq \Fun_{\C | \C}(\C,\FDF) \simeq \Fun_{\C \boxtimes\overline{\C}}(\C,\FDF) \simeq \overline{\C} \mathop{\boxtimes}\limits_{\scriptscriptstyle {\C \boxtimes\overline{\C}}}^{} \FDF .\]
As the categories $\C$ and $\FDF$ are semisimple, they are separable $\C$-bimodule categories because of \cite[Prop.~2.5.10]{douglas2018dualizable}.
Moreover, the relative Deligne product of separable bimodules is separable \cite[Thm.~2.5.5]{douglas2018dualizable}. Since, by \cite[Prop.~2.5.3]{douglas2018dualizable}, separable module categories are semisimple, we obtain that $\Z(\FDF)$ is semisimple and, in particular, unimodular. Thus, the claim follows from Theorem~\ref{thm:tFrob-char}.
\end{proof}

\begin{remark}
Let $F\colon\C\!\longrightarrow\!\D$ be a strong monoidal $\kk$-linear functor between separable fusion categories. Assuming that $F$ admits adjoints, then it is exact and thus a tensor functor. By Theorem \ref{thm:separable-implies-tensor-Frob}, we have that $F$ is \Frob\!\!.
We can conclude from Corollary~\ref{cor:tensor-Frob-Z'-Frob-mon} that the functor $\Z(F^\ra)'$ from \eqref{eq:tilde-Z-F-ra} is Frobenius monoidal, thereby positively answering \cite[Questions~5.22 \& 5.23]{flake2024frobenius2} with $\kk$ being algebraically closed of arbitrary characteristic.
\end{remark}

%%%%%%%%%%%%%%%%%%%%%%%%%%%%%%%%%%%%%%%%%%%%%%%%%%%%%%%%%%%%%%%
%%%%%%%%%%%%%%%%%%%%%%%%%%%%%%%%%%%%%%%%%%%%%%%%%%%%%%%%%%%%%%%
%%%%%%%%%%%%%%%%%%%%%%%%%%%%%%%%%%%%%%%%%%%%%%%%%%%%%%%%%%%%%%%
%%%%%%%%%%%%%%%%%%%%%%%%%%%%%%%%%%%%%%%%%%%%%%%%%%%%%%%%%%%%%%%
%%%%%%%%%%%%%%%%%%%%%%%%%%%%%%%%%%%%%%%%%%%%%%%%%%%%%%%%%%%%%%%
%%%%%%%%%%%%%%%%%%%%%%%%%%%%%%%%%%%%%%%%%%%%%%%%%%%%%%%%%%%%%%%

\section{Twisted module categories}\label{sec:F-twisted}
Let $F\colon\C\!\longrightarrow \!\D$ be a tensor functor between finite tensor categories. Given a left $\D$-module category $\M$, we obtain a left $\C$-module category $\FM$ by pulling the $\D$-action back along $F$, i.e.  $\FM$ consists of the same underlying category as $\M$, but with the action defined by 
\begin{equation*}
    c\trF m := F(c)\act m 
\end{equation*}
for every $c\in\C$ and $m\in\M$. We refer to the so obtained $\C$-module category $\FM$ as an \textit{$F$-twisted module category}.
Given a $\D$-module functor $H\colon\M\to\N$, denote by 
\begin{equation}\label{eq:F*_module_functor}
    F^* (H)\Colon \FM \longrightarrow \FN
\end{equation}
the $\C$-module functor given by $H$ as a functor and with module structure 
\begin{equation*}
\varphi^*_{c,m}:=\varphi_{F(c),m}\Colon H(F(c)\tr m)\xsim F(c)\tr H(m).
\end{equation*} 
A module functor involving twisted module categories is also called a \textit{twisted module functor}.
\begin{example}
Consider the double-dual tensor functor $\dD\coloneqq(-)\dd:\D\longrightarrow\D$. The relative Serre functor of a $\D$-module category $\M$
\begin{equation}
    \dS_\subM^\supD\Colon \M\longrightarrow {}_{\scriptscriptstyle\dD}\M
\end{equation}
becomes a twisted $\D$-module functor with constraints given by \eqref{eq:Serre_twisted1}.
\end{example}

A monoidal natural transformation $\alpha\colon F\Longrightarrow G$ between tensor functors from $\C$ to $\D$ determines a $\C$-module equivalence $\id^{\alpha}\colon \FM \xsim \GM$. This is defined by the assignment $\id^\alpha(m)=m$ together with the $\C$-module structure
\[ \alpha_c\tr \id_m=(\id^\alpha)_{c,m}   \Colon \id^{\alpha}(c\trF m)=F(c)\tr m  \longrightarrow G(c)\tr m = c\trG \id^\alpha(m) \,,\]
for $c\in\C$ and $m\in\M$. Notice that since $\C$ is rigid, every such $\alpha$ is necessarily invertible.

\begin{example}
Given a tensor functor $F\colon\C\rightarrow\D$ and a $\D$-module category $\M$, we can use the monoidal natural isomorphism $\xi_c\colon F(c\rv\rv) \longrightarrow F(c)\rv\rv$ \cite[Lem.~1.1]{ng2007higher} to get an equivalence between the $\C$-module categories ${}_\subF({}_\subdD\M)$ and ${}_\subdD({}_\subF \M)$. In the following, we will often use this identification.
\end{example}
The $\C$-module category $\FD$ captures the procedure of twisting actions along $F$. For any $\D$-module category $\M$ the assignment
\begin{equation}\label{eq:FDM=FM}
\FD\btD \M \xsim\FM,\quad d\boxtimes m\longmapsto d\act m
\end{equation}
is an equivalence of $\C$-module categories.

The internal Hom of an $F$-twisted module category $\FM$ can be described in terms of the internal Hom of the original module category $\M$ by the formula 
\begin{equation}\label{eq:twisted_inner_hom}
    \uHom_{\,\subFM}^{\supC}(m,n)\cong F^{\ra}(\,\uHom_{\,\subM}^{\supD}(m,n)\,)\,.
\end{equation}
This description will be useful to investigate the exactness of twisted module categories.
\begin{remark}
If $\FM$ is an indecomposable $\C$-module category, then $\M$ is indecomposable as a $\D$-module category. Indeed, assume that we have a non-trivial $\D$-module category decomposition,
\begin{equation*}
\M\simeq\bigoplus_i\M_i
\end{equation*}
then by twisting the components of $\M$, we obtain a non-trivial decomposition
\begin{equation*}
\FM\simeq\bigoplus_i\FM_i
\end{equation*}
as $\C$-module categories. The converse does not hold. Consider a non-trivial finite group $G$ and the unique tensor functor $\iota\colon\Vect\longrightarrow\Vect_G$. The regular module category $\Vect_G$ is an indecomposable $\Vect_G$-module category, but ${}_\iota\Vect_G=\Vect_G$ is decomposable as a $\kk$-linear category.  
\end{remark}

%%%%%%%%%%%%%%%%%%%%%%%%%%%%%%%%%%%%%%%%%%%%%%%%%%%%%%%%%%%%%%%
%%%%%%%%%%%%%%%%%%%%%%%%%%%%%%%%%%%%%%%%%%%%%%%%%%%%%%%%%%%%%%%
%%%%%%%%%%%%%%%%%%%%%%%%%%%%%%%%%%%%%%%%%%%%%%%%%%%%%%%%%%%%%%%

\subsection{Exactness of module categories}
Tensor categories are regarded as the categorification of rings and exact module categories may be seen as categorical analogues of projective modules. A classical result concerning the preservation of projectivity under the pullback of ring actions via homomorphisms states: let $f:A\rightarrow B$ be a ring homomorphism. If $B$ is projective as an $A$-module then, a $B$-module $M$ is projective if and only if $M$ is projective as an $A$-module. This result finds an analog in the categorical setting of twisting actions along tensor functors $F\colon\C\!\longrightarrow \!\D$. 
\begin{proposition}\label{prop:perfect-iff}
Let $F\colon\C\!\longrightarrow \!\D$ be a perfect tensor functor. A $\D$-module category $\M$ is exact if and only if the $\C$-module category $\FM$ is exact.
\end{proposition}
\begin{proof}
Since, $F$ is perfect, $F(p)\in\D$ is projective for any projective $p\in \C$. If $\M$ is exact as $\D$-module, then for every $m\in\M$, we have that $p\trF m=F(p)\tr m$ is projective in $\M$, and thus $\FM$ is an exact $\C$-module. For the converse, pick a non-zero projective object $p\in \C$. Since $F$ is faithful, it reflects zero objects and thus $F(p)$ is non-zero \cite[Lemma~5.1]{shimizu2016unimodular}. Assuming that $\FM$ is an exact $\C$-module, $F(p)\tr m$ is projective for all $m\in \M$. The claim follows by \cite[Lemma~2.4]{etingof2017exact}.
\end{proof}

This means that twisting by perfect functors preserves exactness of module categories.

\begin{corollary}\label{cor:Fra_algebra}
Let $F\colon\C\!\longrightarrow \!\D$ be a tensor functor between finite tensor categories. If $F$ is perfect, the lax monoidal functor $F^\ra\colon\D\!\longrightarrow \!\C$ preserves exactness of algebras.
\end{corollary}
\begin{proof}
Given an exact algebra $A\in\D$, consider the module category $\M\!=\!\moD_A(\D)$. Since, $\FM\simeq{\moD}_{F^\ra(A)}(\C)$, Proposition \ref{prop:perfect-iff} implies that $F^\ra(A)$ is exact if $F$ is perfect.
\end{proof}

\begin{remark}
The corestriction of a tensor functor $F$ to its image is surjective and therefore perfect. Thus, by Corollary~\ref{cor:Fra_algebra}, $F^{\ra}(A)$ is exact in $\C$ for any exact algebra $A\in\mathrm{im}(F)$. For instance, since $\unit_{\subD}\in \mathrm{im}(F)$ is an exact algebra, we have that $F^{\ra}(\unit_{\subD})$ is an exact algebra in $\C$.
\end{remark}

\begin{proposition}\label{prop:FMexact}
Let $F\colon\C\!\longrightarrow \!\D$ be a tensor functor. The $\C$-module category $\FD$ is exact if and only if the tensor functor $F$ is perfect. 
\end{proposition}
\begin{proof}
If $F$ is perfect, it preserves projective objects according to Lemma \ref{lem:perfect_functor}, and then $\FD$ is exact. On the other hand, exactness of $\FD$ implies that $F^\ra\cong\uHom_{{}_F\!\D}(\unit,-)$ is exact, and therefore $F$ admits a double right adjoint, which means that $F$ is perfect.
\end{proof}

\begin{remark}\label{rem:fullyexact}
Recent results on the interaction between the relative Deligne product and exactness of module categories appear in \cite{gainutdinov2026fullyexact}. The authors define the classes of \textit{fully exact} and \textit{perfect} module categories. It would be interesting to investigate the preservation of these properties under twisting along \Frob functors and perfect functors. We thank the referee for this remark.
\end{remark}

For a finite tensor category $\C$, let $\mathbf{Mod}^{\rm{ex}}(\C)$ denote the $2$-category with exact left $\C$-module categories as objects, $\C$-module functors as $1$-morphisms and module natural transformations as $2$-morphisms.

\begin{corollary}
A perfect tensor functor $F\colon\C\!\longrightarrow \!\D$ between finite tensor categories, induces a pseudo-functor
\be \label{eq:ex_pseudo-functor}
\mathbf{F^*}\Colon\mathbf{Mod}^{\rm{ex}}(\D)\longrightarrow\mathbf{Mod}^{\rm{ex}}(\C),\quad \M\longmapsto\FM,\quad H\longmapsto F^*(H)
\end{equation}
between their $2$-categories of exact module categories. The pseudo-functor \eqref{eq:ex_pseudo-functor} is pseudo-naturally equivalent to the pseudo-functor \({}_\subF\D\btD-\colon\mathbf{Mod}^{\rm{ex}}(\D)\longrightarrow\mathbf{Mod}^{\rm{ex}}(\C)\,.\)
\end{corollary}
\begin{proof}
By Proposition \ref{prop:perfect-iff}, $F$-twisted module categories are exact. A routine check shows that $F^*(H_2\circ H_1)= F^*(H_2)\circ F^*(H_1)$. The equivalences \eqref{eq:FDM=FM} assemble into the desired pseudo-natural equivalence in the second part of the statement.
\end{proof}

%%%%%%%%%%%%%%%%%%%%%%%%%%%%%%%%%%%%%%%%%%%%%%%%%%%%%%%%%%%%%%%
%%%%%%%%%%%%%%%%%%%%%%%%%%%%%%%%%%%%%%%%%%%%%%%%%%%%%%%%%%%%%%%
%%%%%%%%%%%%%%%%%%%%%%%%%%%%%%%%%%%%%%%%%%%%%%%%%%%%%%%%%%%%%%%

\subsection{Action of central objects and relative Serre functors}\label{subsec:rel-Serre}
Given tensor functors $F,G\colon\C\longrightarrow\D$ and $\M$ a $\D$-module category, there is a $\kk$-linear functor

\begin{equation}\label{eq:centr_act}
\Z(\FDG)\longrightarrow\Fun_{\C}(\GM,\FM)\,,\quad(d,\gamma)\longmapsto (d\act-,\varphi)
\end{equation}
where the module functor structure $\varphi$ on $d\tr-$ is defined by the composition
\begin{equation*}
d\tr (c\trG m)=d\otimes G(c) \tr m \xrightarrow{~\;\gamma\;~} F(c)\otimes d\tr m=c\trF (d\tr m)\,.
\end{equation*}
For an object $(d,\gamma)\in\Z(\FDG)$ and a $\C$-module functor $H\colon\N\rightarrow\GM$, we denote the composite $\C$-module functor by  $d\tr H\colon \N\rightarrow \FM$.

\begin{example}
Let $\C$ be a finite tensor category and let $\dD=(-)\rv\rv$ and $\odD=\rv\rv(-)$ denote the double dual functors.
The Radford isomorphism \eqref{eq:Radford-C} together with the distinguished invertible object $D_\subC$ forms an object in $\Z( {}_{\subdD} \C_{\,\subodD})$. Therefore, for any $\C$-module category $\M$, there is a $\C$-module functor $D_\subC\tr-\colon {}_{\subodD} \M\longrightarrow {}_{\subdD} \M$. 
\end{example}

Consider a perfect tensor functor $F\colon\C\longrightarrow\D$ and a $\D$-module category $\M$. By Proposition \ref{prop:Z(F)-rmo}, we have that the relative modular object $\chi_{\scriptscriptstyle{\Z(F)}}=(\chi_{\subF},\sigma)$ is an object in $\ZFD$ and thus its action becomes a $\C$-module endofunctor on $\FM$:
\begin{equation}\label{eq:relative_modular_obj_acts}
    \chi_{\subF}\tr-\Colon\FM\longrightarrow\FM
\end{equation}
with $\C$-module functor structure given by
\begin{equation*}
    \chi_{\subF}\tr (c\trF m)=\chi_{\subF}\otimes F(c)\tr m\xrightarrow{~\sigma_c\tr\id~}F(c)\otimes \chi_{\subF} \tr m=c\trF(\chi_{\subF}\tr m)
\end{equation*}
where $\sigma_c\colon\chi_{\subF}\otimes F(c)\xsim F(c)\otimes \chi_{\subF}$ is the $F$-half-braiding of $\chi_{\subF}$.

In the case of the regular $\D$-module category $\D$, the $\C$-module functor $\chi_{\subF}\otimes-$ can be turned into a $(\C,\D)$-bimodule functor and into a $\C$-bimodule functor
\begin{equation}
    \chi_{\subF}\otimes-\Colon\FD_\subD\longrightarrow\FD_\subD, \quad\text{ and } \quad \chi_{\subF}\otimes- \colon\FDF\longrightarrow\FDF
\end{equation}
with trivial right module structures, i.e. $\chi_{\subF}\otimes (d\tlF c)=\chi_{\subF}\otimes d\otimes F(c)=(\chi_{\subF}\otimes d)\tlF c$.

\noindent Trivializing these bimodule functors provides another characterization of the \Frob property.
\begin{lemma}\label{lem:chi_otimes_bimod}
Let $F\colon\C\!\longrightarrow\!\D$ be a perfect tensor functor. The following are equivalent:
\begin{enumerate}[\rm (i)]
    \item $F$ is \Frobp
    \item The $(\C,\D)$-bimodule functor $\chi_{\subF}\otimes- \colon\FD_\subD\longrightarrow\FD_\subD$ is isomorphic to $\id_{\FD_\subD}$.
    \item The $\C$-bimodule functor $\chi_{\subF}\otimes- \colon\FDF\longrightarrow\FDF$ is isomorphic to $\id_\FDF$.
\end{enumerate}
\end{lemma}

\begin{proof}
A bimodule natural isomorphism $\gamma\colon \chi_{\subF}\otimes- \cong\id_{\FD_\subD}$ (respectively $\gamma\colon \chi_{\subF}\otimes- \cong\id_{\FDF}$) obeys
\begin{enumerate}[(a)]
    \item $\gamma_{d\otimes d'}=\gamma_d\otimes\id_{d'}$ for every $d,d'\in\D$ (respectively $\gamma_{d\tlF c}=\gamma_d\otimes\id_{F(c)}$ for $c\in \C$).
    \item $\gamma_{c\trF d}=\id_{F(c)}\otimes\gamma_d\circ \sigma_c\otimes\id_d$ for every $d\in\D$ and $c\in\C$.
\end{enumerate}
We first prove that (i) implies (ii). Since $F$ is \Frobc by Theorem \ref{thm:tFrob-char}(iii) there is $\lambda:\chi_{\subF}\cong\unit_\subD$ with $\id_{F(c)}\otimes\lambda\circ\sigma_c=\lambda\otimes\id_{F(c)}$.
A routine check shows that $\gamma_d\coloneqq\lambda\otimes\id_d\colon \chi_\subF\otimes d\to d$ satisfies (a) and (b). Now, that (ii) implies (iii) follows immediately, since a bimodule isomorphism $\gamma\colon \chi_{\subF}\otimes- \cong\id_{\FD_\subD}$ in particular fulfills condition (a) for $d'=F(c)$. Finally, we check (iii) implies (i). Consider a bimodule isomorphism $\gamma\colon \chi_{\subF}\otimes- \cong\id_{\FDF}$. Condition (a) for $d=\unit$ reads $\gamma_{F(c)}= \gamma_\unit\otimes\id_{F(c)}$ and condition (b) that $\gamma_{F(c)}= \id_{F(c)}\otimes\gamma_\unit\circ \sigma_c$. We conclude that $F$ is \Frobc by Theorem \ref{thm:tFrob-char}(iii).
\end{proof}

\begin{remark}
For any tensor functor $F\colon\C\longrightarrow \D$, Proposition~\ref{prop:rel-mod-obj} provides an isomorphism 
\[
(F(D_\subC),F(\mathcal{R}_\subC))\xsim (\,\chi_{\subF}\otimes D_\subD,\,(\sigma\otimes\id)\circ(\id\otimes\mathcal{R}_\subD)\,)
\] 
of objects in $\Z({}_{\dD\circ\dD}\D_\subF)$. 
This data can be re-written as an isomorphism 
\[
\left(F(D_\subC)\otimes D_\subD^{-1},\,(F(\mathcal{R}_\subC)\otimes\id)\circ(\id\otimes \mathcal{R}_\subD^{-1})\,\right)\xsim (\chi_{\subF},\sigma)=\chi_{\scriptscriptstyle{\Z(F)}}
\]
of objects in $\ZFD$. Under the equivalence \eqref{eq:centr_act}, we obtain an isomorphism 
\begin{equation}\label{eq:chi-D-relation}
(F(D_\subC) \tr -) \circ F^*(D_\subD^{-1} \act -) \cong \chi_{\subF} \act - \Colon 
\FM 
\longrightarrow 
\FM\,.
\end{equation} 
of $\C$-module functors, for any $\D$-module category $\M$.
\end{remark}

A consequence of Proposition \ref{prop:perfect-iff} is that the $F$-twisted module category arising from an exact module category is exact and thus admits a relative Serre functor. In order to describe the relative Serre functor of $\FM$, we first describe its Nakayama functor as a $\C$-module functor.
\begin{lemma}\label{lem:FMNak}
Given a tensor functor $F\colon\C\rightarrow\D$ and a $\D$-module category $\M$, we have that $\dN_{\subFM} = F^*(\dN_\subM)$ as twisted $\C$-module functors.
\end{lemma}
\begin{proof}
By the definition of $F^*$ we have equality as functors and on both sides the twisted module functor structure is given by the isomorphism $\fn_{F(c),m}$ from \eqref{eq:module_Nakayama_twisted_functor}. 
\end{proof}

The following proposition relates the relative Serre functor and the relative modular object of $F$.

\begin{proposition}\label{prop:FMSerre}
Let $F\colon\C\!\longrightarrow \!\D$ be a perfect tensor functor and $\M$ an exact $\D$-module category.
There is a natural isomorphism
\begin{equation}\label{eq:serre_FM}
    \chi_{\subF}\tr F^*\!\left(\;\dS_{\subM}^\supD\;\right)\;\cong\;\dS_{\subFM}^\supC
\end{equation}
of twisted $\C$-module functors, where $\chi_{\subF}\act-$ is the $\C$-module functor from \eqref{eq:relative_modular_obj_acts}.
\end{proposition}
\begin{proof}
On one hand we have an isomorphism of twisted $\C$-module functors from \eqref{eq:Nakayama_Serre}
    \be\label{Nakayama_C}
    \dN_{\subFM}\cong D^{-1}_{\subC}\trF \dS_{\subFM}^\supC
    \end{equation}   
and on the other hand an isomorphism of twisted $\D$-module functors from \eqref{eq:Nakayama_Serre}
    \be\label{Nakayama_D}
    D^{-1}_{\subD}\tr \dS_{\subM}^\supD\cong\dN_\subM\,.
    \end{equation}
Moreover, $\dN_{\subFM}=F^*(\dN_\subM)$, by Lemma~\ref{lem:FMNak}. Thus, by composing \eqref{Nakayama_C} with the image of \eqref{Nakayama_D} under $F^*$ we obtain an isomorphism of twisted $\C$-module functors
    \begin{equation*} 
    F^*(D^{-1}_{\subD}\tr -)\circ F^*\!\left(\,\dS_{\subM}^\supD\,\right)
    =F^*\!\left(\,D^{-1}_{\supD}\tr \dS_{\subM}^\supD\,\right)
    \cong D^{-1}_{\supC}\trF \dS_{\subFM}^\supC
    =F(\,D_{\supC}^{-1}\,)\tr \dS_{\subFM}^\supC
    \end{equation*}
where the first equality is the functoriality of $F^*$ and the last one is the definition of the $F$-twisted action. Post-composing both sides with the module functor $F(D_\subC)\tr-$ and using the isomorphism \eqref{eq:chi-D-relation} we obtain the desired result.
\end{proof}
We spell out some details that are implicit in the above result. Since the relative Serre functor is unique up to unique isomorphism, from here on, we can fix $\dS_{\subFM}^\supC = \chi_{\subF}\tr F^*\!\left(\;\dS_{\subM}^\supD\;\right)$. 
The twisted left $\C$-module structure of $\dS_{\subFM}^\supC$ is given by the composition
\begin{align*}
\fs_{c,m}^F\Colon
\chi_{\subF} \tr \dS_{\subM}^{\supD} (F(c)\tr m) 
\xrightarrow{\id\tr \fs_{F(c),m}} \chi_{\subF} \tr\left( F(c){\dd} \tr \dS_{\subM}^\supD(m) \right)
\xrightarrow{\sigma_{c\rv\rv} \tr \id}
(F(c{\dd})\otimes \chi_{\subF}) \tr \dS_{\subM}^\supD(m)\,.
\end{align*}

As we have seen, the relative modular object of a perfect tensor functor measures its failure of preserving distinguished invertible objects and, in particular, its failure of preserving unimodularity. It turns out, that the relative modular object also relates the distinguished invertible objects of the dual tensor categories with respect to an exact module category.

\begin{proposition}\label{DIO_dual_cats} 
Let $F\colon\C\!\longrightarrow \!\D$ be a perfect tensor functor and $\M$ an exact $\D$-module category.
There is a natural isomorphism of $\C$-module functors
\begin{equation}\label{eq:DIO_dual_cats}
\chi_{\subF}\act F^*\!\left(D^{-1}_{\D_\subM^*}\right) \cong D_{\C_{\!\FM}^*}^{-1}
\end{equation}
relating the distinguished invertible objects of the dual categories of $\M$ and $\FM$ and the $\C$-module functor  $\chi_{\subF}\act-$ defined in \eqref{eq:relative_modular_obj_acts}.
\end{proposition}
\begin{proof}
The desired natural isomorphism arises as the composition
\begin{equation*}
\begin{aligned}
D^{-1}_{\C_{\subFM}^*}\cong D_\subC^{-1}\trF\; \dS_{\subFM}^\supC\circ \dS_{\subFM}^\supC =F(D_\subC^{-1})\act\; \dS_{\subFM}^\supC\circ \dS_{\!\FM}^\supC 
\cong \chi_{\subF}\act F^*\!\left(\;D_\subD^{-1}\tr \; \dS_{\subM}^\supD\circ \dS_{\subM}^\supD\right)\cong \chi_{\subF}\act F^*\!\left(D^{-1}_{\D_\subM^*}\right) 
\end{aligned}
\end{equation*}
where the first isomorphism comes from \eqref{eq:Radford_mod}, next is the definition of the $F$-twisted action, the second to last isomorphism is \eqref{eq:serre_FM} and lastly we use isomorphism \eqref{eq:Radford_mod} once more.
\end{proof}

In light of the above results, it is important to understand the functor $\chi_{\subF}\act-:\FM \longrightarrow\FM$. 

\begin{definition}\label{def:relative_frobenius}
Let $F\colon\C\!\longrightarrow\!\D$ be a perfect tensor functor and $\M$ an exact $\D$-module category. We say $F$ is \textit{Frobenius with respect to $\M$} in case $\chi_{\subF}\tr-\colon\FM\to\FM$ is isomorphic to $\id_\FM$ as a $\C$-module functor.
\end{definition}
\begin{remark}
The assignment  \eqref{eq:F*_module_functor} defines a tensor functor
\be\label{eq:tensor_func_F!}
F^*_\subM\Colon \D^*_{\subM} \longrightarrow \C^*_{{}_F\subM}\,,
\end{equation}
The terminology used in Definition \ref{def:relative_frobenius} is inspired by the fact that in case $F^*_\subM$ is perfect, then $F$ is Frobenius with respect to $\M$ if and only if $F^*_\subM$ is a Frobenius functor. This follows from Proposition~\ref{prop:relative_mod_obj_properties}(v) and the isomorphism \eqref{eq:DIO_dual_cats}.
\end{remark}

\begin{proposition}
    \label{prop:chi-action}
Let $F\colon\C\longrightarrow\D$ be a perfect tensor functor. Then:
\begin{enumerate}[\rm (i)]
    \item If $F$ is \Frobc then $F$ is Frobenius with respect to any $\D$-module category $\M$.
    \item If $F$ is Frobenius with respect to $\D$, then $F$ is Frobenius.
\end{enumerate}
\end{proposition}
\begin{proof}
To prove (i), by Lemma~\ref{lem:chi_otimes_bimod}(ii), we assume that $\chi_{\subF}\otimes-\cong\id_{{}_\subF \D_\subD}$. Then, given a $\D$-module category $\M$, the diagram of $\C$-module equivalences
\begin{equation*}
\begin{tikzcd}[column sep=2.1cm,row sep=1cm]
    \FD\btD\M\ar[r,"\chi_{\subF}\otimes-\,\btD\; \id_\subM"]
    \ar[d,"\eqref{eq:FDM=FM}",swap]
    &\FD\btD\M \ar[d,"\eqref{eq:FDM=FM}"]\\
    \FM\ar[r,"\chi_{\subF}\tr-",swap]&\FM
\end{tikzcd}
\end{equation*}
commutes up to natural isomorphism (given by the module associativity constraints $(\chi_{\subF}\otimes d)\tr m\cong \chi_{\subF}\tr (d\tr m)$ of $\M$). Thus, we obtain a trivialization of $\chi_{\subF}\tr-$. Lastly, (ii) follows by evaluating $\chi_{\subF}\otimes-\cong\id_{\!\FD}$ at the unit object to obtain an isomorphism $\chi_\subF\cong \unit$. By Proposition \ref{prop:frobenius_uni}, this means that $F$ is Frobenius.
\end{proof}
Examples~\ref{ex:uqsl2-kG} and~\ref{ex:not-f-Frob} later show that the converses of (i) and (ii), respectively do not hold in general.

%%%%%%%%%%%%%%%%%%%%%%%%%%%%%%%%%%%%%%%%%%%%%%%%%%%%%%%%%%%%%%%

\subsection{Unimodular structures}

\begin{definition}\cite[Def.\ 3.2]{yadav2023unimodular}
Let $\C$ be a finite tensor category, and $\M$ an exact $\C$-module category. A \textit{unimodular structure} on $\M$ consists of a module natural isomorphism
\begin{equation*}
\tfu\Colon\id_\subM\xRightarrow{\;\sim\;}D_{\C_\subM^*}
\end{equation*}
A module category endowed with a unimodular structure is called a \textit{unimodular module category}. 
\end{definition}
\begin{remark}
Notice that for a module category ${}_\subC\M$ to be unimodular $\C$ is not required to be unimodular. The dual tensor category $\C_\subM^*$ must be unimodular instead \cite[Rem.\ 3.7]{yadav2023unimodular}.
\end{remark}

Twisting module actions by perfect tensor functors that are \Frob allows for the transfer of unimodular structures on module categories, as we see next. 

%-----------------------------------

\begin{proposition}\label{prop:uni-tensor-Frob}
Let $F\colon\C\rightarrow\D$ be a perfect tensor functor and $\M$ a unimodular $\D$-module category. The $\C$-module category $\FM$ is unimodular if and only if $F$ is Frobenius with respect to $\M$.
In that case, the $\C$-module category $\FM$ is endowed with a unimodular structure via
\begin{equation*}
{}_{{}_F}\tfu\Colon\id_\subM=F^*\!\left(\id_\subM\right)\xrightarrow{\quad F^*\!(\,\tfu\,) \quad} F^*\!\left(\,D_{\D_{\!\!\M}^*}\,\right)\cong D_{\C_{\!\FM}^*}\,.
\end{equation*}
In particular, if $F$ is \Frobc then the $\C$-module category $\FM$ is unimodular.
\end{proposition}
\begin{proof}
Since, by assumption, $\chi_{\subF}\act-$ is trivial, the result follows from Proposition \ref{DIO_dual_cats}. The second part of the statement follows by Proposition \ref{prop:chi-action}(i).
\end{proof}

\begin{remark}
We call a tensor category $\C$ {\it weakly unimodular} iff there exists a unimodular tensor category in the Morita class of $\C$, or equivalently if $\C$ admits a unimodular module category. An example of a non weakly unimodular category is category of representations of the Taft Hopf algebra \cite{yadav2023unimodular}.
Proposition~\ref{prop:uni-tensor-Frob} shows that there are no \Frob functors from a finite tensor category that is not weakly unimodular into one that is weakly unimodular.
\end{remark}

%%%%%%%%%%%%%%%%%%%%%%%%%%%%%%%%%%%%%%%%%%%%%%%%%%%%%%%%%%%%%%%
%%%%%%%%%%%%%%%%%%%%%%%%%%%%%%%%%%%%%%%%%%%%%%%%%%%%%%%%%%%%%%%
%%%%%%%%%%%%%%%%%%%%%%%%%%%%%%%%%%%%%%%%%%%%%%%%%%%%%%%%%%%%%%%

\subsection{Pivotal structures and sphericality}

Let $\C$ be a pivotal finite tensor category and $\M$ an exact left $\C$-module category. The pivotal structure $\fp\colon\id_{\,\subC}\xRightarrow{\;\sim\;}(-)\dd$ turns the relative Serre functor of $\M$ into a $\C$-module functor via
\begin{equation*}
\Se_\subM^\supC(c\act m)\xrightarrow{\;\eqref{eq:Serre_twisted1}\;} c\dd\act \Se_\subM^\supC(m)\xrightarrow{\;\fp_c\act\id\;}
c\act \Se_\subM^\supC(m)
\end{equation*}
for $c\in\C$ and $m\in\M$.
\begin{definition}{
\cite[Def. 3.11]{shimizu2017relative}}
\label{def:pivmodule}
A \emph{pivotal structure} on an exact left $\C$-module category $\M$ is a module natural isomorphism 
\begin{equation*} 
\widetilde{\fp}\Colon\id_\subM\;{\xRightarrow{\;\sim\;}}\;\Se_\subM^\supC\,.
\end{equation*}
\noindent A module category together with a pivotal structure is called a \emph{pivotal module category}. 
\end{definition}

Just as a perfect functor $F$ preserves exactness of module categories, requiring 
$F$ to be \Frob ensures the transfer of pivotal structures on module categories.

\begin{proposition}\label{prop:piv-tensor-Frob}
Let $F\colon\C\longrightarrow\D$ be a pivotal tensor functor between pivotal finite tensor categories and $\M$ a pivotal $\D$-module category. The $\C$-module category $\FM$ is pivotal if and only if Frobenius with respect to $\M$.
In that case, the $\C$-module category $\FM$ is endowed with a  pivotal structure via
\begin{equation}\label{eq:piv_transported}
{}_{{}_F}\widetilde{\fp}\Colon\id_\subM= F^*\!\left(\id_\subM\right)\xrightarrow{\quad F^*\!(\,\widetilde{\fp}\,) \quad} F^*\!\left(\;\dS_{\subM}^\supD\;\right)\cong\dS_{\subFM}^\supC\,.
\end{equation}
In particular, if $F$ is \Frobc then the $\C$-module category $\FM$ is pivotal.
\end{proposition}
\begin{proof}
By assumption, $\chi_{\subF}\tr-$ is trivial. Thus, by Proposition \ref{prop:FMSerre}, the pivotal structure on $\M$ furnishes a pivotal structure on $\FM$ as desired. The second part of the statement follows by Proposition \ref{prop:chi-action}(i).
\end{proof}

\begin{remark}
In the setting of Proposition \ref{prop:piv-tensor-Frob}, the tensor functor $F^*\colon\D_\subM^*\longrightarrow \C_{\!\FM}^*$ is pivotal. The $2$-categorical extension of this fact is proven in Proposition~\ref{prop:pivotal-2-functor}.
\end{remark}

\begin{theorem}
\label{thm:piv_symm_Frob}
Let $\C$ and $\D$ be pivotal finite tensor categories and $F\colon\C\!\longrightarrow \!\D$ a pivotal tensor functor.
\begin{enumerate}[\rm (i)]
    \item Assume that $F$ is \Frobc then the lax monoidal functor $F^\ra\colon\D\!\longrightarrow \!\C$ preserves exact symmetric Frobenius algebras.
    \item If $\C$ is endowed with a braiding, $F$ has the structure of a central functor and $F$ is Frobenius, then $F^\ra(\unit)$ is a haploid commutative exact symmetric Frobenius algebra.
\end{enumerate}
\end{theorem}
\begin{proof}
Given an exact algebra $A\in\D$, pivotal structures on the module category $\M=\moD_A(\D)$ correspond to symmetric Frobenius structures on $A$ \cite[Thm.\ 6.4]{shimizu2024quasi}. From Corollary \ref{cor:Fra_algebra} we know that $F^\ra$ preserves exact algebras, thus $F^\ra(A)$ is exact. Moreover, since $F$ is \Frobc Proposition \ref{prop:piv-tensor-Frob} implies that $\FM=\moD_{\!F^\ra(A)}(\C)$ inherits a pivotal structure, which in turn induces a symmetric Frobenius structure on $F^\ra(A)$.

For assertion (ii), observe first that by Proposition~\ref{prop:central-tensor-Frobenius}, the fact that $F$ is Frobenius implies it is indeed \Frob\!\!. Consequently, the claim follows from part (i) and \cite[Prop.\ 8.8.8]{etingof2016tensor}.
\end{proof}

\begin{remark}
A non-degenerate ribbon finite tensor category is called a modular tensor category (MTC).
By \cite[Thm.~5.21(d)]{shimizu2024commutative}, given an MTC $\C$ and a haploid commutative exact symmetric Frobenius algebra $A$ in $\C$, the category of local modules $\C_A^{\loc}$ is also an MTC. Thus, the above result shows that central tensor functors $F\colon\C\rightarrow\D$ that are \Frob and pivotal can be used to obtain new MTCs. 
Let $A=F^{\ra}(\unit_\subD)$ and denote by $G:\C\rightarrow\mathrm{im}(F) \cong \mathrm{mod}_A(\C)$ the restriction of $F$ to its image. Then, $G$ is also central and we can describe the category $\C_A^\loc$ as the M\"uger centralizer $\Z_{(2)}(\C \subset \Z(\mathrm{im}(F))$.
\end{remark}

\begin{remark}
In \cite{shimizu2024quasi}, the notions of relative Serre functors and pivotal structures are extended to module categories that are not necessarily exact. We expect Propositions~\ref{DIO_dual_cats} and \ref{prop:piv-tensor-Frob} to generalize to this setting.
Moreover, by \cite[Thm.\ 6.4]{shimizu2024quasi}, symmetric Frobenius algebra structures on a (not necessarily exact) algebra $A$ in $\C$ are in correspondence with pivotal structures on the $\C$-module category $\moD_A(\C)$. Using this, Theorem~\ref{thm:piv_symm_Frob} can be generalized to non-exact algebras.
\end{remark}

The previous results have a $2$-categorical description. In order to fully reformulate them, we need additional structure: given a $2$-category $\mathscr{B}$, the existence of dualities  for $1$-morphisms extends to a pseudo-functor
  \begin{equation*}
  \begin{aligned}[c]
  (-)\rv\Colon \mathscr{B} & \,\longrightarrow\, \mathscr{B}^{\,\text{op},\text{op}},
  \\
  \hspace*{2.5em} x & \,\longmapsto\, x \,,
  \\
  \hspace*{2.5em} (a\colon x\,{\to}\, y) & \,\longmapsto\, (a\rv\colon y\,{\to}\, x) \,.
  \end{aligned}
  \end{equation*}
A \emph{pivotal structure} on a $2$-category $\mathscr{B}$ with dualities is a 
pseudo-natural equivalence 
  \begin{equation}\label{pivotal_st_bicategory}
  \textbf{P} \Colon \id_\mathscr{B} \xRightarrow{\;\sim~\,} (-){\dd}
  \,.
  \end{equation}
A $2$-category together with a pivotal structure is called a \emph{pivotal $2$-category}.

\begin{definition} 
Let $\mathscr{B}_1$ and $\mathscr{B}_2$ be pivotal $2$-categories. A \textit{pivotal pseudo-functor} consists of a pseudo-functor $\mathscr{F}\colon\mathscr{B}_1\longrightarrow\mathscr{B}_2$ admitting an invertible modification filling the diagram:
\begin{equation}\label{eq:pivotal_str_pseudo-functor}
\begin{tikzcd}[column sep=3em,row sep=1em]
    &&|[alias=Z]|\mathscr{F}\circ (-){\dd} 
    \ar[dd,"\sim", Rightarrow]\\
\mathscr{F}\ar[rru,"\id_\mathscr{F}\,\circ\, \textbf{P}^{\mathscr{B}_1}",sloped, Rightarrow]\ar[rrd,swap,"\textbf{P}^{\mathscr{B}_2}\,\circ\,\id_\mathscr{F}",sloped, Rightarrow,""{name=U,below}]&&\\
    &&(-){\dd}\circ \mathscr{F}
        \arrow[Rightarrow, from=Z,to=U,  "\sim", shorten=5mm]
\end{tikzcd}
\end{equation}
\end{definition}

Denote by $\textbf{Mod}^{\text{piv}\!}(\C)$ 
the $2$-category that has pivotal $\C$-module categories as objects, $\C$-module 
functors as $1$-morphisms and module natural transformations as $2$-morphisms. 
Pivotal modules are exact and therefore every module functor between pivotal modules has adjoints.
These turn $\textbf{Mod}^{\text{piv}\!}(\C)$ into a $2$-category with dualities 
for $1$-morphisms. Moreover, $\textbf{Mod}^{\text{piv}\!}(\C)$ is endowed with a pivotal structure \eqref{pivotal_st_bicategory} with components given for any object $\M$ by $\textbf{P}_\subM\coloneqq\id_\subM$ and for a $1$-morphism
$H\colon {}_{\subC}\N_1 \longrightarrow {}_{\subC}\N_2$ by
\begin{equation*}
  \textbf{P}_{\!H}^{}\Colon H
  \xRightarrow{\;\id\circ\widetilde{\fp}_1~} H \circ \Se_{\subN_1}^\supC
  \xRightarrow{\;\eqref{eq:Serre_twisted}~} \Se_{\subN_2}^\supC \circ H^\lla
  \xRightarrow{\;(\widetilde{\fp}_2)^{-1}\circ\id~\,} H^\lla ,
  \end{equation*}
where $\widetilde{\fp}_i$ are the pivotal structures of the module categories $\N_i$.

\begin{proposition}\label{prop:pivotal-2-functor}
A pivotal tensor functor $F\colon\C\!\longrightarrow \!\D$ between pivotal finite tensor categories, that is \Frob\!\!, induces a pivotal pseudo-functor
\begin{equation}\label{eq:piv_pseudo-functor}
F^*\Colon\mathbf{Mod}^{\rm{piv}}(\D)\longrightarrow\mathbf{Mod}^{\rm{piv}}(\C),\quad (\M,\widetilde{\fp})\longmapsto (\FM,{}_{{}_F}\widetilde{\fp}),\quad H\longmapsto F^*(H)
\end{equation}
between their pivotal $2$-categories of pivotal module categories, where ${}_{{}_F}\tfp$ is given by \eqref{eq:piv_transported}.
\end{proposition}
\begin{proof}
From \eqref{eq:ex_pseudo-functor} we have that $F$-twisting of module categories extends to a pseudo-functor. 
By Proposition \ref{prop:piv-tensor-Frob} pivotal structures on $\D$-module categories induce pivotal structures on the associated $F$-twisted $\C$-module categories, and thus the assignment in \eqref{eq:piv_pseudo-functor} is well-defined. It remains to be checked that the pseudo-functor $F^*$ is pivotal. First, on objects we have that $\textbf{P}^{\mathscr{B}_1}_\subM=\id_\subM=\textbf{P}^{\mathscr{B}_2}_\subM$. Now, given a module functor $H\colon {}_{\subD}\N_1 \longrightarrow {}_{\subD}\N_2$ the diagram
\begin{equation*}
\begin{tikzcd}
&  F^*(H) \circ F^*(\Se_{\subN_1}^\supD)
  \ar[r,"\eqref{eq:Serre_twisted}",Rightarrow]
  \ar[dd,Rightarrow,swap,"\eqref{eq:serre_FM}"]&
  F^*(\Se_{\subN_2}^\supD) \circ F^*(H^\lla)\ar[dd,"\eqref{eq:serre_FM}",Rightarrow]
  \ar[rd,"F^*(\tfp_2^{-1}\circ\id)",Rightarrow]&
 \\
    F^*(H)
  \ar[ru,"F^*(\id\circ\tfp_1)",Rightarrow]\ar[rd,"\id\circ {}_{{}_F}\tfp_1",Rightarrow,swap] &&&  F^*(H^\lla)  \\
&  F^*(H) \circ \Se_{\FN_1}^\supC
  \ar[r,"\eqref{eq:Serre_twisted}",Rightarrow,swap]&
  \Se_{\FN_2}^\supC \circ F^*(H^\lla)
  \ar[ru,"{}_{{}_F}\widetilde{\fp}_{2}^{-1}\circ\id",swap,Rightarrow]&
\end{tikzcd}
\end{equation*}
commutes by the definition of the transported pivotal structures on $\FN_{\!1}$ and $\FN_{\!2}$ and the fact that \eqref{eq:serre_FM} is an isomorphism of twisted module functors. It follows that the diagram \eqref{eq:pivotal_str_pseudo-functor} commutes on the nose (i.e. the identity modification is a filler), and thus \eqref{eq:piv_pseudo-functor} is pivotal.
\end{proof}

We now turn to investigate the preservation of sphericality. Given a unimodular finite tensor category $\C$, a pivotal structure $\fp\colon\id_{\,\subC}\xRightarrow{\;\sim\;}(-)\dd$ on $\C$ is called \textit{spherical} iff $\fp\dd\circ \fp=\mathcal{R}_\C$ as defined in \cite[Def.\ 3.5.2]{douglas2018dualizable}. In analogy, the sphericality property for a pivotal module category can be defined in terms of the module Radford isomorphism \eqref{eq:Radford_mod}.
\begin{definition}\cite[Def.\ 5.22]{spherical2025}\label{spherical_mod}~\\
Let $\M$ be a pivotal module category over a $($unimodular$)$ spherical finite tensor category $\D$. Consider further an isomorphism $\tfu_{\subD}\colon\unit\xsim D_\subD^{-1}$ and a unimodular structure $\tfu_{\subM}\colon\id_{\subM}
\xRightarrow{\;\sim\;}D_{\D_\subM^*}^{-1}$ on $\M$.\\
The pivotal module ${}_\subD\M$ is called $(\tfu_{\subD},\tfu_{\subM})$-\emph{spherical} iff the diagram 
\begin{equation*}
  \begin{tikzcd}[row sep=1.5em,column sep=2.5em]
    \Se^\supD_\subM  \ar[rr,"\overline{\mathcal{R}}_\subM",Rightarrow]  &~
    & \overline{\Se}^\supD_\subM\ar[dl,"\widetilde{\fp}",Rightarrow] \\
    ~& \id_{\!\scriptscriptstyle \M}  \ar[ul,"\widetilde{\fp}",Rightarrow]&~
  \end{tikzcd}
\end{equation*}
commutes, where $\tfp$ is the pivotal structure of $\M$, and $\overline{\mathcal{R}}_{\!\scriptscriptstyle \M}$ is the composition
\begin{equation}
\Se^\supD_\subM\xRightarrow{\;\tfu_\subD\tr\id\;}
D_\subD^{-1}\tr\Se^\supD_\subM\xRightarrow{\mathcal{R}_{\subM}\;~}
D_{\D_{\!\M}^*}^{-1}\circ\overline{\Se}^{\supD}_\subM\xLeftarrow{\;~\tfu_\subM\circ\id\;~}\overline{\Se}^{\supD}_\subM\,,
\end{equation}
and $\mathcal{R}_{\subM}$ is the module Radford isomorphism \eqref{eq:Radford_mod}.
\end{definition}

The following technical lemma shows that the procedure of $F$-twisting module categories preserves the module Radford isomorphism.

\begin{lemma}\label{lem:RadfordM_pres}
Let $F\colon \C\longrightarrow\D$ be a perfect tensor functor and $\M$ an exact $\D$-module category. Then the following diagram commutes.
\begin{equation*}
\begin{tikzcd}[row sep=2em,column sep=5.5em]
D_\subC^{-1}\trF \;\Se^\supC_\subFM\ar[d,"\mathcal{R}_\subFM",swap,Rightarrow]\ar[r,"\eqref{eq:relative-distinguished-iso}\tr\eqref{eq:serre_FM}",Rightarrow]& F^*(D_\subD^{-1}\tr \Se^\supD_\subM)
\ar[d,Rightarrow,"F^*(\mathcal{R}_\subM)"]\\
D_{\C_{\subFM}^*}^{-1}\circ \overline{\Se}^\supC_\subFM
\ar[r,"\eqref{eq:DIO_dual_cats}\circ\eqref{eq:serre_FM}",swap,Rightarrow]
&
F^*\left(D_{\D_\subM^*}^{-1}\circ \Se^\supD_\subM\right)
\end{tikzcd}
\end{equation*}
\end{lemma}

\begin{proof}
The commutativity follows after noticing that we can rewrite the diagram as
\begin{equation}
\begin{tikzcd}[row sep=2em,column sep=5.5em]
F(D_\subC^{-1})\tr \;\Se^\supC_\subFM\ar[dd,"\mathcal{R}_\subFM",swap,Rightarrow]
\ar[dr,"\cong",Rightarrow,swap]
\ar[r,"\eqref{eq:serre_FM}",Rightarrow]&F(D_\subC^{-1})\otimes \chi_{\subF}^{-1}\tr F^*(\Se^\supD_\subM) \ar[r,"\eqref{eq:relative-distinguished-iso}",Rightarrow]
& F^*(D_\subD^{-1}\tr \Se^\supD_\subM)
\ar[dd,Rightarrow,"F^*(\mathcal{R}_\subM)"]\\
&\dN_{\!\FM}=F^*(\dN_\subM)\ar[ur,"\cong",Rightarrow,swap]\ar[dr,"\cong",Rightarrow]&\\
D_{\C_{\subFM}^*}^{-1}\circ \overline{\Se}^\supC_\subFM
\ar[ur,"\cong",Rightarrow]
\ar[r,"\eqref{eq:serre_FM}",swap,Rightarrow]
&D_{\C_{\subFM}^*}^{-1}\circ\chi_{\subF}^{-1}\tr F^*\!\left(\;\overline{\Se}_{\subM}^\supD\;\right)\ar[r,"\eqref{eq:DIO_dual_cats}",swap,Rightarrow]
&
F^*\left(D_{\D_\subM^*}^{-1}\circ \Se^\supD_\subM\right)
\end{tikzcd}
\end{equation}
where each of the triangles commutes due to the definition of the isomorphisms involved.
\end{proof}

\begin{proposition} \label{prop:sph-Frob}
Let $F\colon\C\to\D$ be a pivotal tensor functor between $($unimodular$)$ spherical tensor categories such that $F$ is \Frobc  and let $\M$ be a pivotal $\D$-module. Given a trivialization $\tfu_{\subC}:\unit\xsim D_\subC$ and a unimodular structure $\tfu_{\subM}$ on $\M$, if $\M$ is $(F(\tfu_{\subC}),\tfu_{\subM})$-spherical, then $\FM$ is $(\tfu_{\subC},{}_{{}_F}\!\tfu_\subM)$-spherical.
\end{proposition}
\begin{proof}
From Proposition \ref{prop:uni-tensor-Frob} it follows that $F^*(\tfu_\subM)$ is a unimodular structure on $\FM$. The assertion holds due to the commutativity of the following diagram.
\begin{equation*}
\begin{tikzcd}[row sep=2em,column sep=4.2em]
\Se^\supC_\subFM\ar[d,"\eqref{eq:serre_FM}",Rightarrow,swap]\ar[r,"\tfu_\subC\tr\id",Rightarrow]
&D_\subC^{-1}\trF \Se^\supC_\subFM\ar[d,"\eqref{eq:serre_FM}",Rightarrow]\ar[r,"\mathcal{R}_\subFM",Rightarrow]
& D_{\C_{\subFM}^*}^{-1}\circ \overline{\Se}^\supC_\subFM\ar[d,"\eqref{eq:serre_FM}",Rightarrow]
&\ar[l,swap,"F^*(\tfu_\subM)\circ\id",Rightarrow]\overline{\Se}^\supC_\subFM\ar[d,"\eqref{eq:serre_FM}",Rightarrow]\\
 F^*\!\left(\,\dS_{\subM}^\supD\,\right)  
\ar[r,"\tfu_\subC\tr\id",Rightarrow]
\ar[rd,"F^*\!(F(\tfu_\subC)\tr\id)",swap,Rightarrow]
    &F(D_\subC^{-1})\tr F^*\!\left(\,\dS_{\subM}^\supD\,\right)\ar[d,"\eqref{eq:relative-distinguished-iso}",Rightarrow]
    &    D_{\C_{\subFM}^*}^{-1}\circ F^*\!\left(\overline{\Se}^{\supD}_\subM\right)\ar[d,"\eqref{eq:DIO_dual_cats}",Rightarrow]
    &\ar[l,swap,"F^*(\tfu_\subM)\circ\id",Rightarrow]F^*\!\left(\overline{\Se}^\supD_\subM\right)\ar[dd," F^*\!(\widetilde{\fp})",Rightarrow]\ar[ld,"{}_{{}_F}\!\tfu_\subM\circ\id",Rightarrow]
     \\
    &
    F^*\!\left(\,D_\subD^{-1}\tr \dS_{\subM}^\supD\,\right)
    \ar[r,"F^*\!(\mathcal{R}_\subM)",swap,Rightarrow]&
    F^*\!\left(\,D_{\D_{\subM}^*}^{-1}\circ \overline{\Se}^{\supD}_\subM\,\right)
    &\\
    \id_{\subM}\ar[rrr,equal]\ar[uu," F^*\!(\widetilde{\fp})",Rightarrow]&&&\id_{\subM}
  \end{tikzcd}
\end{equation*}    
where the upper left and right squares commute due to naturality of $\tr$ and $\circ$, respectively. The middle diagram commutes according to Lemma \ref{lem:RadfordM_pres}. The bottom diagram is the condition that $\M$ is $(F(\tfu_{\subC}),\tfu_{\subM})$-spherical. Finally, the remaining triangles commute from the definitions of $F(\tfu_\subC):\unit\xsim F(D_\subC)\cong D_\subD$ and ${}_{{}_F}\tfu_\subM$, respectively.
\end{proof}

%%%%%%%%%%%%%%%%%%%%%%%%%%%%%%%%%%%%%%%%%%%%%%%%%%%%%%%%%%%%%%%
%%%%%%%%%%%%%%%%%%%%%%%%%%%%%%%%%%%%%%%%%%%%%%%%%%%%%%%%%%%%%%%
%%%%%%%%%%%%%%%%%%%%%%%%%%%%%%%%%%%%%%%%%%%%%%%%%%%%%%%%%%%%%%%
%%%%%%%%%%%%%%%%%%%%%%%%%%%%%%%%%%%%%%%%%%%%%%%%%%%%%%%%%%%%%%%
%%%%%%%%%%%%%%%%%%%%%%%%%%%%%%%%%%%%%%%%%%%%%%%%%%%%%%%%%%%%%%%
%%%%%%%%%%%%%%%%%%%%%%%%%%%%%%%%%%%%%%%%%%%%%%%%%%%%%%%%%%%%%%%

\section{Applications to Hopf algebras}\label{sec:Hopf-examples}
In this section, we illustrate the results from the previous section for tensor categories coming from Hopf algebras.
Let $(H,m,u,\Delta,\varepsilon,S)$ be a finite-dimensional Hopf algebra over $\kk$. We will use the Sweedler notation $\Delta(h)=h_1\otimes h_2$ for the coproduct. Since $H$ is finite-dimensional, the antipode is invertible and we denote its inverse as $\overline{S}$.
We refer the reader to \cite{radford2011hopf} for basics of Hopf algebras.

Given a map $\ff\colon A\rightarrow B$ of $\kk$-algebras and $M$ a left $B$-module, we use the notation $M_{\ff}$ to denote the left $A$-module $M$ with action defined by $a \cdot_\ff m \coloneqq \ff(a)\cdot m$ for $m\in M$ and $a\in A$.
Note that for algebra maps $\ff\colon A\rightarrow B$ and $\ff':B \rightarrow C$ and a $C$-module $M$, we have that $(M_{\ff'})_\ff = M_{\ff'\,\ff}$ as left $A$-modules.

%%%%%%%%%%%%%%%%%%%%%%%%%%%%%%%%%%%%%%%%%%%%%%%%%%%%%%%%%%%%%%%
%%%%%%%%%%%%%%%%%%%%%%%%%%%%%%%%%%%%%%%%%%%%%%%%%%%%%%%%%%%%%%%
%%%%%%%%%%%%%%%%%%%%%%%%%%%%%%%%%%%%%%%%%%%%%%%%%%%%%%%%%%%%%%%

\subsection{Preliminaries about the tensor category \texorpdfstring{$\Rep(H)$}{Rep(H)}}
The category $\Rep(H)$ of finite-dimensional representations of $H$ is a finite tensor category. 
For $X,\, Y\in\Rep(H)$, recall that $X\otimes Y = X\ok Y$ and $X\rv=X^*:=\Hom(X,\kk)$ as $\kk$-vector spaces. The $H$-action is given by 
\[ h\cdot(x\ok y) = h_1\cdot x\ok h_2\cdot y, \qquad h\cdot \langle f,-\rangle = \langle f,S(h) -\rangle\]
for $x\iN X$, $y\iN Y$, $f\iN X^*$ and $h\iN H$.
The unit object is the one-dimensional $H$-representation $\kk_{\varepsilon}$.

For $X\in \Rep(H)$, using the canonical isomorphism $\phi_X: X \rightarrow X\rv\rv$ and thinking of elements in $X^{**} = X\rv\rv$ as $\phi_X(x)$ for some $x\in X$ we get that the $H$-action is given by $h\cdot \phi_X(x) = \phi_X(S^2(h)\cdot x)$. Thus, $X\rv\rv \cong X_{S^2}$.
In a similar manner, $X\rv\rv\rv\rv \cong X_{S^4}$.

Let $\Lambda\in H$ be a non-zero left integral of $H$, that is an element satisfying $h \Lambda =\langle\varepsilon,h\rangle \Lambda$. Then, \textit{the (left) modular function} $\alpha:H\rightarrow \kk$ of $H$ is defined to be the unique algebra map satisfying $\Lambda h = \langle \alpha,h\rangle \Lambda$ for all $h\in H$. Set $\oalpha:=\alpha\circ S$, this is the inverse of $\alpha$ in the algebra $H^*$. 
This matches the conventions used in prior work \cite{yadav2023unimodular} whose results we will use next. One can calculate that $D_{\Rep(H)}=\dN^l_{\Rep(H)}(\kk_{\varepsilon}) = \kk_{\oalpha}$ (see \cite[\S4.2]{yadav2023unimodular} for details).
We call $H$ \textit{unimodular} if $\alpha=\varepsilon$. 
As $D_{\Rep(H)}$ is invertible, its dual $(D_{\Rep(H)})\rv$ is $\kk_{\alpha}$.

A map $\lambda:H\rightarrow\kk$ is called a {\it right cointegral} of $H$ if it satisfies $\langle\lambda,h_1\rangle h_2 = \langle \lambda,h \rangle 1_H$ for all $h\in H$. Such a cointegral always exists and from here on we fix a left integral $\Lambda$ and a right cointegral $\lambda$ satisfying $\langle\lambda,\Lambda\rangle =1$. There exists a unique element $\gH\in H$ satisfying 
\begin{equation}\label{eq:g_H-defn}
h_1 \langle \lambda, h_2\rangle = \langle\lambda,h\rangle \gH.    
\end{equation} 
We will use the notation $\ogH = \gH^{-1}$. Using this data, we can describe the Radford isomorphism \eqref{eq:Radford-C} 
$D_{\Rep(H)} \otimes X \rightarrow X\rv\rv\rv\rv \otimes D_{\Rep(H)} $
of $\Rep(H)$ explicitly as follows
\begin{equation}\label{eq:Hopf-Radford-1}
\kk_{\oalpha} \otimes X \xrightarrow{\R_X} X_{S^4} \otimes \kk_{\oalpha}, 
\quad 
1\ok x \mapsto \gH \cdot x \ok 1.   
\end{equation}
This can be derived using \cite[Prop.\ 4.23]{yadav2023unimodular}. 

A \textit{pivotal element} of a Hopf algebra is a grouplike element $\sg_{\piv}\in H$ satisfying $\sg_{\piv} h \sg_{\piv}^{-1} = S^2(h)$ for all $h\in H$. Pivotal structures on $\Rep(H)$ are in bijection with pivotal elements of $H$. Given a pivotal element $\sg_{\piv}$, the pivotal structure $\fp:X\rightarrow X\dd=X_{S^2}$ is given by $x\mapsto \sg_{\piv}\cdot x$.

%%%%%%%%%%%%%%%%%%%%%%%%%%%%%%%%%%%%%%%%%%%%%%%%%%%%%%%%%%%%%%%

Let $\ff\colon H'\rightarrow H$ be a bialgebra map between two finite-dimensional Hopf algebras over $\kk$. This means in particular that $\ff$ commutes with the antipodes, that is, 
\begin{equation}\label{eq:S-commutes-f}
    \ff\circ S_{H'} = S_H \circ \ff\,.
\end{equation}
Moreover, $\ff$ induces a tensor functor $F_{\ff}: \Rep(H) \rightarrow\Rep(H')$  defined via the assignment
\begin{equation*}
\begin{aligned}
 X\mapsto X_{\ff}, \qquad(f:X\rightarrow Y) \mapsto (f:X_{\ff} \rightarrow Y_{\ff})    
\end{aligned}
\end{equation*}
called \textit{the restriction of scalars} functor. 
We call $\ff$ \textit{perfect} iff the tensor functor $F_{\ff}$ is perfect. 
Using Lemma~\ref{lem:perfect_functor}, one can obtain alternate characterizations of perfect bialgebra maps, for instance:

\begin{lemma}
A bialgebra map $\ff\colon H' \rightarrow H$ is perfect if and only if $H_\ff$ is a projective $H'$-module.
\end{lemma}

\begin{remark}
As Hopf algebras are free (and therefore projective) over their Hopf subalgebras \cite{nichols1989hopf}, if $H'$ is a Hopf subalgebra of $H$ (that is, $\ff$ is injective), then $\ff$ is perfect.
\end{remark}

We also discuss the pivotality of the functor $F_\ff$.
Let $(H', \sg'_{\piv})$ and $(H,\sg_{\piv})$ be pivotal Hopf algebras. We call a bialgebra map $\ff\colon H'\rightarrow H$ \textit{pivotal} if it satisfies $\ff(\sg'_{\piv})=\sg_{\piv}$.

\begin{lemma}
    The tensor functor $F_{\ff}$ between the pivotal categories $\Rep(H')$ and $\Rep(H)$ is pivotal if and only if $\ff$ is pivotal.
\end{lemma}
\begin{proof}
Being a tensor functor, $F_\ff$ yields a canonical isomorphism
\[ \zeta_X \Colon (X^{\vee})_{\ff} = F_{\ff}(X^{\vee}) \longrightarrow F_{\ff}(X)^{\vee} = (X_{\ff})^{\vee} .\]
From \eqref{eq:S-commutes-f}, one can check that $\zeta_X(f) = f$ for all $f\in X^{\vee}$.
Since the isomorphisms $\zeta_X$ are just identity maps, this follows from the definition \eqref{eq:piv-functor} of functor pivotality.
\end{proof}

%%%%%%%%%%%%%%%%%%%%%%%%%%%%%%%%%%%%%%%%%%%%%%%%%%%%%%%%%%%%%%%
%%%%%%%%%%%%%%%%%%%%%%%%%%%%%%%%%%%%%%%%%%%%%%%%%%%%%%%%%%%%%%%
%%%%%%%%%%%%%%%%%%%%%%%%%%%%%%%%%%%%%%%%%%%%%%%%%%%%%%%%%%%%%%%

\subsection{\Frob functors induced by bialgebra maps}

Let $\ff\colon H'\rightarrow H$ be a bialgebra map. Define the {\it relative modular function} 
\[ \chi_{\ff}:H'\rightarrow\kk,  \qquad \chi_{\ff}(h') = \langle\alpha_H, S_H (\ff(h'_1))\rangle \langle\alpha_{H'},h'_2\rangle.\] 
The relative modular object of the perfect tensor functor $F_{\ff}$ is determined by the relative modular function of $\ff$ as we show next.

\begin{proposition}\label{prop:rel-modular-Hopf}
The relative modular object $\chiFf$ of $F_{\ff}$ is the $H'$-module $\kk_{\chi_{\ff}}$. The $F_{\ff}$-half-braiding of $\chi_{\ff}$ is given for $X\in \Rep(H)$ by the isomorphism 
\begin{equation}\label{eq:Hopf-sigma}
\begin{aligned}
\sigma_X\Colon \chi_{\scriptscriptstyle F_{\ff}}\otimes F_{\ff}(X) 
&\longrightarrow F_{\ff}(X) \otimes \chi_{\scriptscriptstyle F_{\ff}},\\ 
1 \ok x 
&\longmapsto \gH \,\ff(\overline{\sg_{H'}}) \cdot x \ok 1 .
\end{aligned}
\end{equation}
\end{proposition}
\begin{proof}
By Proposition~\ref{prop:relative_mod_obj_properties}(v), we know that 
\[ \chi_{\scriptscriptstyle F_{\ff}}= F_\ff(D_{\Rep(H)}) \otimes D_{\Rep(H')}^{-1} =
F_{\ff}(\kk_{\oalpha_H})\otimes (\kk_{\oalpha_{H'}})^{-1} 
= \kk_{\oalpha_H \ff}\otimes \kk_{\alpha_{H'}} .\]
Thus, the first claim follows. Now, consider $X\in\Rep(H)$ and denote $X_\ff=F_\ff(X)\in \Rep(H')$. Then, the commutative diagram \eqref{eq:chi-central-structre} in Proposition~\ref{prop:rel-mod-obj} gives a description of $\sigma_X$. Namely, it is given by the following compositions of Radford isomorphisms
\begin{align*}
\kk_{\oalpha_{H} \ff} \otimes \kk_{\alpha_{H'}} \otimes X_\ff 
& \xrightarrow{\id \otimes \R_{X_\ff}} \kk_{\oalpha_H \ff} \otimes X_{\ff \overline{S}^4_{H'}} \otimes \kk_{\alpha_{H'}} 
\\
& \qquad \qquad = \kk_{\oalpha_H \ff} \otimes X_{\overline{S}^4_{H} \ff} \otimes \kk_{\alpha_{H'}} 
= (\kk_{\oalpha_H} \otimes X_{\overline{S}^4_{H}})_\ff \otimes \kk_{\alpha_{H'}} 
\\
& \xrightarrow{(\R_X)_\ff \otimes \id} (X \otimes \kk_{\oalpha_H})_\ff \otimes \kk_{\alpha_{H'}} 
= X_\ff \otimes \kk_{\oalpha_H \ff} \otimes \kk_{\alpha_{H'}}
\end{align*} 
Using the description of $\R$ \eqref{eq:Hopf-Radford-1}  provided earlier, the above map is given by
\[ 1\ok 1 \ok x 
\mapsto 
1 \ok (\overline{\sg_{H'}}\cdot_{\ff} x) \ok 1 = 1 \ok \ff(\overline{\sg_{H'}})\cdot x \ok 1 
\mapsto 
\gH \ff(\overline{\sg_{H'}}) \cdot x \ok 1 \ok 1 . \]
This completes the proof.
\end{proof}

A consequence of the above proposition is the following characterization of \Frob functors, which generalizes \cite[Corollary 1.8]{fischman1997frobenius} to perfect bialgebra maps.

\begin{theorem}\label{thm:Hopf-Frob-description}
Let $\ff\colon H'\rightarrow H$ be a perfect bialgebra map.
    \begin{enumerate}[{\rm (i)}]
        \item $F_\ff$ is Frobenius $\iff \chi_\ff = \varepsilon_{H'} \iff \alpha_H\circ \ff  = \alpha_{H'}$.
         
        \item $F_{\ff}\, \text{is \Frob} \, 
        \iff \, \chi_{\ff} = \varepsilon_{H'} \,\text{and} \,\, \ff(\sg_{H'})=\gH \,
        \iff \, \alpha_H\circ \ff = \alpha_{H'} \, \text{and} \,\, \ff(\sg_{H'})=\gH$.
    \end{enumerate}
\end{theorem}
\begin{proof}
(i) As $\chiFf\cong \kk_{\chi_\ff}$, by Proposition~\ref{prop:frobenius_uni}(i), the first equivalence is clear.
For the second equivalence, note that we can rewrite the relative modular function as
\[ \chi_{\ff} = (\oalpha_{H}\circ \ff) \times \alpha_{H'} \in  H^*.\]
where $\times$ denotes the product in $H^*$.
As $\varepsilon\in H^*$ is its unit, we obtain that $\chi_{\ff}=\varepsilon$ if and only if $\oalpha_{H}\circ \ff = (\alpha_{H'})^{-1} = \oalpha_{H'}$. As the antipode is bijective, this is equivalent to $\alpha_H\circ \ff = \alpha_{H'}$.

(ii) By the equivalence (i)$\iff$(iii) of Theorem~\ref{thm:tFrob-char}, $F_{\ff}$ is \Frob if and only if $\chiFf\cong \unit$ as an object of $\Rep(H)$ and the half-braiding is the identity map. By Proposition~\ref{prop:rel-modular-Hopf}, the former is equivalent to
$\chi_\ff=\varepsilon_{H'}$ and the latter to $\ff(\sg_{H'})=\gH$.
This proves the first equivalence. The second equivalence follows because of part (i).
\end{proof}

As a corollary, we obtain the following result.

\begin{corollary}\label{cor:Hopf-Frob-description}
Suppose that $H'\subseteq H$ is a Hopf subalgebra, then:
\begin{enumerate}[{\rm (i)}]
    \item \cite[Corollary 1.8]{fischman1997frobenius} $F_{\ff}$ is Frobenius if and only if $\alpha_H(h')=\alpha_{H'}(h')$ for all $h'\in H'$.
    \item  $F_\ff$ is \Frob if and only if $\sg_{H'}=\gH$ and $\alpha_H(h')=\alpha_{H'}(h')$ for all $h'\in H'$.
\end{enumerate} 
\end{corollary}

We end this subsection with some examples and remarks.

\begin{example}\label{ex:uqsl2-Hopf-example}
    Fix an odd integer $n=2h+1$ for $h>0$ and let $q$ be a primitive $n$-th root of unity. The Hopf algebra $u_q(\fsl_2)$ is generated, as an algebra, by $E$, $F$ and $K$ subject to the relations 
    \begin{equation*}
    K^n=1,\quad E^n=F^n=0, \quad KE = q^2 EK, \quad KF = q^{-2} FK, \quad EF - FE = \frac{K-K^{-1}}{q-q^{-1}}. 
    \end{equation*}
    The coalgebra structure is given by $\varepsilon(K) = 1 $, $\varepsilon(E) = \varepsilon(F) = 0$ and
    \begin{align*}
        \Delta(K) = K\otimes K, \; \quad & \Delta(E) = E\otimes K + 1\otimes E, \quad \Delta(F) = F\otimes 1 + K^{-1}\otimes F.
    \end{align*}
    The antipode, on the generators, assigns the following values:  $S(K) = K^{-1}$, $S(E) = -EK^{-1}$ and $S(F) = -KF$. Also, the element $\Lambda =  (\sum_{i=0}^{n-1} K^i)E^{n-1} F^{n-1}$ is a left integral. From this description, one obtains that the modular function $\alpha:u_q(\fsl_2)\rightarrow\kk$ is given by $\alpha(E)=\alpha(F)=0$ and $\alpha(K)=1$. Moreover, the distinguished grouplike element is given by $g_{u_q(\fsl_2)}=K^2$.
    Lastly, $S^2(h) = K h K^{-1}$ for all $h\in u_q(\fsl_2)$, so $\sg_{\piv} = K$ is a pivotal element of $u_q(\fsl_2)$. 
    
    \begin{itemize}
        \item If one considers the subalgebra of $u_q(\fsl_2)$ generated by $E$ and $K$, one obtains the Taft (Hopf) algebra $T\coloneqq T_n(q^2)$. The right modular function $\alpha_T$ is given by $\alpha_T(E)=0$ and $\alpha_T(K)=q^{-2}$. As $\alpha|_T\neq \alpha_T$, the functor $F\colon\Rep(u_q(\fsl_2))\rightarrow \Rep(T)$ is not Frobenius. 
        \item If one considers the subalgebra of $u_q(\fsl_2)$ generated by $K$, one obtains a Hopf subalgebra isomorphic to the group algebra, $T'\coloneqq \kk \mathbb{Z}_n$, of $\mathbb{Z}_n$. The right modular function $\alpha_{T'}$ is given by $\alpha_{T'}(K)=1$. As $\alpha|_{T'} = \alpha_{T'}$, the restriction of scalars functor $F\colon\Rep(u_q(\fsl_2)) \rightarrow \Rep(\kk \mathbb{Z}_n)$ is a Frobenius functor. However, as $g_{u_q(\fsl_2)} = K^2$ and $g_{\kk\mathbb{Z}_n} = 1$ are not the same, $F$ is not \Frobp
    \end{itemize}
\end{example}

\begin{example}\label{ex:tensor-Frob}
For a Hopf algebra $H$, take $\ff$ to be the unit $u:\kk\rightarrow H$ of $H$. Then, $\ff$ is a bialgebra map whose induced functor is just the fiber functor $F_\ff\colon\Rep(H)\rightarrow\Vect$.
\begin{enumerate}[{\rm (i)}]
    \item  As $\alpha_H$ is an algebra map, it satisfies $\alpha_H(1_H)=1$. Therefore, $\alpha_H|_{\kk} = \alpha_{\kk}$. This shows that the fiber functor $F_{\ff}$ is always a Frobenius functor.
    \item Also, note that $g_{\kk}=1$. Thus, $F_\ff$ is \Frob if and only if $\gH=\ff(\sg_\kk)=1_H$. This happens if and only if $H^*$ is unimodular. 
\end{enumerate}
For instance, for $H=u_q(\fsl_2)^*$, its dual $H^* \cong u_q(\fsl_2)$ is unimodular. Thus, the fiber functor $F\colon\Rep(u_q(\fsl_2)^*)\rightarrow \Vect$ is \Frob\!\!. On the other hand, the fiber functor $F\colon\Rep(u_q(\fsl_2))\rightarrow\Vect$ is Frobenius, but not \Frob because $u_q(\fsl_2)^*$ is not unimodular.
\end{example}

\begin{remark}\label{rem:tFrobenius_Hopf}
Let $\ff \colon H'\rightarrow H$ be a bialgebra map such that $\chi_\ff=\varepsilon$. If $H$ is unimodular, then so if $H'$. This follows from Proposition~\ref{prop:frobenius_uni}(ii) or can be checked directly. However, the converse is not true as illustrated by the following example. Take $\ff=u:\kk\rightarrow H$ for a non-unimodular Hopf algebra $H$ (like the Taft algebra or $u_q(\fsl_2)^*$). By Example~\ref{ex:tensor-Frob}, $F_\ff$ is Frobenius (and thus, $\chi_\ff=\varepsilon$) even though $H$ is not unimodular. 
\end{remark}

\begin{remark}\label{rem:tFrobenius_Hopf2}
A bialgebra map \(\ff\colon H' \to H\) is called a \textit{Frobenius extension} if \(\ff\) is injective and \(F_\ff\) is Frobenius. In \cite{flake2024frobenius2}, the concept of a {\it central Frobenius extension} for Hopf algebras \(H' \subset H\) is introduced. 
In the finite-dimensional setting, a Frobenius extension \(H' \subset H\) is central Frobenius if and only if \(F_\ff\) is \Frob\!\!. Indeed, combining \cite[Thm.~4.9]{flake2024frobenius2} and \cite[Thm.~3.17]{flake2024frobenius1}, the extension \(H \subset H'\) is central Frobenius if and only if the left and right adjoints of \(F_\ff\colon \Rep(H) \to \Rep(H')\) are isomorphic as \(\Rep(H)\)-bimodule functors - this is precisely the definition of being \Frob\!\!. 
Using this observation, many results in \cite{flake2024frobenius2} about extensions of finite-dimensional Hopf algebras can be obtained using our results. For instance, \cite[Thm.~4.21]{flake2024frobenius2} is a particular case of Corollary~\ref{cor:Hopf-Frob-description}(ii).
\end{remark}

%%%%%%%%%%%%%%%%%%%%%%%%%%%%%%%%%%%%%%%%%%%%%%%%%%%%%%%%%%%%%%%
%%%%%%%%%%%%%%%%%%%%%%%%%%%%%%%%%%%%%%%%%%%%%%%%%%%%%%%%%%%%%%%
%%%%%%%%%%%%%%%%%%%%%%%%%%%%%%%%%%%%%%%%%%%%%%%%%%%%%%%%%%%%%%%

\subsection{Applications to module categories}
In this section, we apply the results from \S\ref{sec:F-twisted} to $\Rep(H)$-module categories. 
Left module categories over $\Rep(H)$ are understood in terms of left $H$-comodule algebras $(L,\delta: L\rightarrow H \ok L)$. Namely, the category $\Rep(L)$ is endowed with a left $\Rep(H)$-module category structure, defined as follows: for $X\in\Rep(H)$ and $M\in\Rep(L)$, $X\tr M\in\Rep(L)$ is just $X\ok M$ as a vector space with $L$-action
\begin{equation*} 
l\cdot (x\ok m) = l_{-1}\cdot x \ok l_0\cdot m.
\end{equation*}
where we have used the Sweedler notation $\delta(l) = l_{-1}\ok l_0$.
We call $L$ \textit{exact} (resp. \textit{indecomposable}) if the $\Rep(H)$-module category $\Rep(L)$ is exact (resp. indecomposable).

As before, $\ff\colon H'\rightarrow H$ is a bialgebra map. Let $L$ be a left $H'$-comodule algebra with coaction $\delta: L \rightarrow H' \otimes L$. Define the map $\delta_\ff$ by the composition:
\[  L \xrightarrow{\;\delta\;} H' \otimes L \xrightarrow{\ff \otimes \id_L} H \otimes L, \qquad a \mapsto \ff(a_{-1})\ok a_0. \]
As a consequence of Proposition~\ref{prop:perfect-iff}, we obtain:
\begin{lemma}
Let $\ff\colon H'\rightarrow H$ be a perfect bialgebra map. Then $(L,\delta)$ is an exact $H'$-comodule algebra if and only if $(L,\delta_\ff)$ is an exact $H$-comodule algebra.
\end{lemma}

Next, we introduce a notion of Frobenius elements for $H$-comodule algebras.
\begin{definition}\label{defn:f-Frob-element}
Let $\ff\colon H'\rightarrow H$ be a perfect bialgebra map and $L$ an exact $H'$-comodule algebra. An \emph{$\ff$-Frobenius element} of $L$ is an invertible element $\sa \in L$ satisfying 
\begin{align}
    \sa \, l\, \sa^{-1} & = \chi_{\ff}(l_{-1})l_0, \qquad (\forall \; l\in L)  \label{eq:F-Frobenius-element-1} \\
    \ff(\sa_{-1}) \otimes \sa_0 & =   \ff(\gHp)\,\overline{\gH} \otimes \sa. \label{eq:F-Frobenius-element-2}
\end{align}
We say that $L$ is $\ff$-Frobenius if it admits an $\ff$-Frobenius element.
\end{definition}

Our choice of terminology is motivated by the following result.
\begin{proposition}\label{prop:Hopf-F-Frobenius-description}
The tensor functor $F_\ff$ is Frobenius with respect to $\Rep(L)$ if and only if $L$ admits an $\ff$-Frobenius element.
\end{proposition}
\begin{proof}
If $F_\ff$ is Frobenius with respect to $\Rep(L)$, we have a $\C$-module natural isomorphism $\eta\colon \id_{\FM}\Rightarrow\chiFf\act-$. Set $\sa=\eta_{L}(1_L)\in L$. Then, for every $M\in \Rep(L)$ and $m\in M$, by the naturality of $\eta$ applied to the $L$-module map 
\[ L\rightarrow M, \qquad l\mapsto l\cdot m, \]
we get that $\eta_M(m) = \sa\cdot m$. As $\eta$ is an isomorphism, $\sa$ is invertible. 
Moreover, as $\eta_L:L\rightarrow \chi_{\ff}\act L$ is a morphism in $\Rep(L)$, we obtain that $\sa$ satisfies \eqref{eq:F-Frobenius-element-1}.
Lastly, since $\eta$ is a $\Rep(H)$-module natural isomorphism, it satisfies 
\[ \eta_{H\act^F L} = (\sigma_H^{-1} \otimes \id_{L}) (\id_{H}\act^F \eta_{L}). \]
Upon evaluating both sides at $1_{H}\otimes 1_{L}$ (using the description \eqref{eq:Hopf-sigma} of $\sigma_H$), we get $\sa$ satisfies \eqref{eq:F-Frobenius-element-2}.

Conversely, let $\sa\in L$ be a $\ff$-Frobenius element. Then, the natural isomorphism $\eta: \id_{\Rep(H')}\Rightarrow \chi_{\ff}\otimes-$ given by $\eta_X(x) = \sa\cdot x$ is a $\Rep(H)$-module natural isomorphism. Thus,  $F_\ff$ is Frobenius with respect to $\Rep(L)$.
\end{proof}

Note that the space of $\ff$-Frobenius elements is $1$-dimensional. Thus, it is desirable to find some easier criteria for the existence of such elements. Next, we provide one such criteria.

\begin{corollary}\label{cor:f-Frobenius-condition}
If $F_\ff$ is Frobenius with respect to $\Rep(H')$, then $\sg_H\in \ff(H')\subset H$.
\end{corollary}
\begin{proof}
By our assumption, a $\ff$-Frobenius element $\sa\in H'$ exists. Applying $\id \otimes \varepsilon_{H'}$ to \eqref{eq:F-Frobenius-element-2} yields $\ff(\sa)=\ff(\sg_{H'})\overline{\sg_H} \,\varepsilon_{H'}(\sa)$. As $\sa$ invertible, $\varepsilon_{H'}(\sa)\neq 0$. Thus, $\sg_H = \ff(h')$ for some $h'\in H'$.
\end{proof}

\begin{remark}
\begin{enumerate}
    \item In Proposition~\ref{prop:chi-action}(i), we noted that if $F$ is \Frobc then it is Frobenius with respect to every module category $\M$. This is evident in the case $F=F_\ff$, because $\sa=1_L$ is an $\ff$-Frobenius element of $L$.
    \item According to Proposition~\ref{prop:chi-action}(ii), if $F_\ff$ is Frobenius with respect to $\M=\Rep(H')$, then $\chi_{\ff}=\varepsilon_{H'}$. Indeed, in this case, $H'$ admits a $\ff$-Frobenius element $\sa$ satisfying \eqref{eq:F-Frobenius-element-1}. Then, applying $\varepsilon_{H'}$ to both sides, the claim follows. 
\end{enumerate}
\end{remark}

\begin{proposition}\label{prop:pivotal-comod-subalg}
Let $\ff\colon H'\rightarrow H$ be a bialgebra map and $L$ be a $H'$-comodule algebra that admits an $\ff$-Frobenius element. If $\Rep(L)$ is endowed with a pivotal (resp., unimodular) $\Rep(H')$-module category structure, then it also inherits a pivotal (resp., unimodular) as a $\Rep(H)$-module category.
\end{proposition}
\begin{proof}
The assumption on $\ff$ ensures that $F_\ff$ is Frobenius with respect to the module category $\Rep(L)$, see Proposition~\ref{prop:Hopf-F-Frobenius-description}. Thus, the claim follows by combining Propositions~\ref{prop:uni-tensor-Frob} and Proposition~\ref{prop:piv-tensor-Frob}.
\end{proof}

Next we provide an example of a functor that is Frobenius with respect to $\M$, but not \Frobp
\begin{example}\label{ex:uqsl2-kG}
Let $q$ be a root of unity of order $n$ and consider the small quantum group $u_q(\fsl_2)$ from Example~\ref{ex:uqsl2-Hopf-example}.
Take $H'\coloneqq\langle K \rangle$ be the Hopf subalgebra of $H=u_q(\fsl_2)$ generated by $K$ and consider the inclusion $\ff \colon H'\rightarrow u_q(\fsl_2)$. Recall that $F_\ff$ is Frobenius, but not \Frobp 

Consider $L=H'$ as a left $H'$-comodule algebra. Then, $\sa=K^{-2}$ is a $\ff$-Frobenius element. Indeed, $F_\ff$ being Frobenius implies that $\chiFf=\varepsilon_{H'}$. Thus condition \eqref{eq:F-Frobenius-element-1} is satisfied because $\sa$ is central and invertible in $H'$. Moreover,  $\ff(\sa_{-1})\otimes \sa_0 = K^{-2}\otimes K^{-2}$ and $\ff(\sg_{H'})\overline{\sg_{u_q(\fsl_2)}} \otimes \sa = K^{-2}\otimes K^{-2}$. Thus, \eqref{eq:F-Frobenius-element-2} also holds. 
Note that $K$ is a pivotal element for $H'$. Thus, taking $\M=\Rep(H')$, we get a pivotal left $\Rep(H')$-module category. By Proposition~\ref{prop:pivotal-comod-subalg}, $\Rep(H')$ is a pivotal $\Rep(u_q(\fsl_2))$-module category.
\end{example}

We end with an example of $\ff$ such that $F_\ff$ is Frobenius, but not Frobenius with respect to $\D$. 
\begin{example}\label{ex:not-f-Frob}
Let $q=e^{2\pi \iota/9}$. Consider the small quantum group $H=u_q(\fsl_2)$ and its Hopf algebra generated $H'=\langle K^3\rangle$ generated by $K^3$. We observed earlier that in this case $F_\ff$ is Frobenius. However, as $\sg_H=K^2$ does not belong to $H'$, Corollary~\ref{cor:f-Frobenius-condition} implies that $F_\ff$ is not Frobenius with respect to $\Rep(H')$.
\end{example}

%%%%%%%%%%%%%%%%%%%%%%%%%%%%%%%%%%%%%%%%%%%%%%%%%%%%%%%%%%%%%%%
%%%%%%%%%%%%%%%%%%%%%%%%%%%%%%%%%%%%%%%%%%%%%%%%%%%%%%%%%%%%%%%
%%%%%%%%%%%%%%%%%%%%%%%%%%%%%%%%%%%%%%%%%%%%%%%%%%%%%%%%%%%%%%%
%%%%%%%%%%%%%%%%%%%%%%%%%%%%%%%%%%%%%%%%%%%%%%%%%%%%%%%%%%%%%%%
%%%%%%%%%%%%%%%%%%%%%%%%%%%%%%%%%%%%%%%%%%%%%%%%%%%%%%%%%%%%%%%
%%%%%%%%%%%%%%%%%%%%%%%%%%%%%%%%%%%%%%%%%%%%%%%%%%%%%%%%%%%%%%%

\section{Applications to internal natural transformations}
\label{sec:internal-natural}
In \cite{fuchs2021internal} the authors studied an internalized version of the vector space of natural transformations between two functors. Given a finite tensor category $\C$, an appropriate internal Hom involving two module functors $H,H'\colon {}_\subC\M\rightarrow{}_\subC\N$ gives an object inside the Drinfeld center of $\C$ called the \textit{object of internal natural transformations} from $H$ to $H'$. Under certain assumptions, this object is endowed with the structure of a symmetric Frobenius algebra internal to $\Z(\C)$. In this section, the results from \S\ref{sec:F-twisted} allow us to reprove and strengthen many statements about internal natural transformations.

%%%%%%%%%%%%%%%%%%%%%%%%%%%%%%%%%%%%%%%%%%%%%%%%%%%%%%%%%%%%%%%
%%%%%%%%%%%%%%%%%%%%%%%%%%%%%%%%%%%%%%%%%%%%%%%%%%%%%%%%%%%%%%%
%%%%%%%%%%%%%%%%%%%%%%%%%%%%%%%%%%%%%%%%%%%%%%%%%%%%%%%%%%%%%%%

\subsection{Exactness and unimodularity of the \texorpdfstring{$\Z(\C)$}{Z(C)}-module category \texorpdfstring{$\Fun_{\C}(\M,\N)$}{FunC(M,N)}}\label{subsec:Fun2}

Let $\C$ be a finite tensor category and $\M$ and $\N$ exact $\C$-module categories. Recall, that the category of module functors $\Fun_{\C}(\M,\N)$ becomes an exact $(\C_\subN^*,\C_\subM^*)$-bimodule category, with action given by composition of module functors \eqref{eq:mod_fun_comp}. The internal Hom's can be described using adjoint functors:

\begin{lemma}\label{lem:intHom}
For $H,H'\in \Fun_{\C}(\M,\N)$, we have that the internal Hom's are given by
\begin{equation*}
\uHom^{\C_\subN^*}_{\Fun_{\C}(\M,\N)}(H,H')\cong H'\circ H^{\la}\;\qquad \text{ and }\qquad \uHom^{\overline{\C_\subM^*}}_{\Fun_{\C}(\M,\N)}(H,H')\cong H^{\ra}\circ H'\,.
\end{equation*}
\end{lemma}
\begin{proof}
For $F\in\C_\subM^*$, observe that there is a natural isomorphism
\begin{equation*}
\varphi  \Colon \Hom_{\Fun_{\C}(\M,\N)}(H\circ F,H') \cong \Hom_{\Fun_{\C}(\M,\M)}(F,H^{\ra} \circ H')\,.
\end{equation*}
The assignment $\varphi$ and its inverse are defined via the formulas
\begin{equation*}
    \varphi(\alpha)= (\id_{H^{\ra}} \circ \alpha)(\eta\circ \id_F): F\Rightarrow H^{\ra}\circ H', \;\;\;\;\;\; 
    \varphi^{-1}(\beta) =  (\varepsilon\circ \id_F)(\id_H\circ \beta):H\circ F\Rightarrow H'.
\end{equation*}
where $\eta\colon\id_{\subM}\Rightarrow H^{\ra} \circ H$ and $\varepsilon\colon H\circ H^{\ra} \Rightarrow \id_{\subN}$ 
denote the unit and counit of the adjunction $H\dashv H^{\ra}$. Since $\alpha,\eta$ are $\C$-module natural transformations, we have that $\varphi(\alpha)$ is a $\C$-module functor. 
The statement for the $\C_\subN^*$-internal Hom follows analogously.
\end{proof}

Given an exact $\C$-module category $\N$ we have a tensor functor
\begin{equation}\label{psiN}
\begin{aligned}
    \PsiN\Colon\Z(\C) &\longrightarrow \C_\subN^*\\
                (c,\nu)&\longmapsto c\tr-
\end{aligned}
\end{equation}
where the half-braiding $\nu$ endows the functor $c\tr-\colon\N\to\N$ with a $\C$-module structure. The tensor structure on the functor $\PsiN$ is given by the module associativity constraints of $\N$
\begin{equation*}
    c\tr(d\tr n)\cong (c\otimes d)\tr n\;.
\end{equation*}

\begin{remark}
In \cite[\S3.1]{davydov2015functoriality}, the tensor functor \eqref{psiN} is called \(\alpha\)-induction. The \(\Z(\C)\)-action on \(\Fun_{\C}(\M, \N)\) is also studied there. However, in the literature \cite{ostrik2003module,fuchs2002tft}, \(\alpha\)-induction typically assumes that \(\C\) is braided, a condition we do not impose. Consequently, we avoid this terminology in this article.
\end{remark}

For an exact $\C$-module category $\M$, the same assignment given by \eqref{psiN} can be made into a tensor functor whose target is the monoidal opposite of the dual tensor category with respect to $\M$
\begin{equation*}
    \PsiM\Colon\Z(\C) \longrightarrow \overline{\C_\subM^*}\;, 
\end{equation*}
but with tensor structure defined by half-braidings
\begin{equation*}
c\tr(d\tr n)\cong (c\otimes d)\tr n \xrightarrow{\nu_{d}^c} (d\otimes c)\tr n \;.
\end{equation*}
Now, we can consider the twist of the $\C_\subN^*$-module category $\Fun_{\C}(\M,\N)$ by $\PsiN$. Explicitly, the $\Z(\C)$-action is given by
\[
\begin{array}{rccc}
    \tr^{\PsiN}  :& \Z(\C) \times \Fun_{\C}(\M,\N)& \rightarrow & \Fun_{\C}(\M,\N) \\ 
    & ((c,\nu)\,,\,H) & \mapsto & c\tr H  
\end{array}    
\]
the composition of the $\C$-module functors $c\tr-:\N\to\N$ and $H$, and module structure \eqref{eq:mod_fun_comp}. Similarly, if we consider $\Fun_{\C}(\M,\N)$ as a (left) $\overline{\C_\subM^*}$-module category, we obtain a $\Z(\C)$-module structure by pulling back the action along $\PsiM$:
\[
\begin{array}{rccc}
    \tr^{\PsiM}  :& \Z(\C) \times \Fun_{\C}(\M,\N)& \rightarrow & \Fun_{\C}(\M,\N) \\ 
    & ((c,\nu)\,,\,H) & \mapsto &  H(c\tr-)
\end{array}    \,.
\]
These two actions of the Drinfeld center $\Z(\C)$ on $\Fun_{\C}(\M,\N)$ are equivalent.
\begin{lemma}\label{lem:center-action-functors}
${}_{\PsiN}\!\Fun_{\C}(\M,\N)$ and ${}_{\PsiM}\!\Fun_{\C}(\M,\N)$ are equivalent as $\Z(\C)$-module categories.
\end{lemma}
\begin{proof}
We have for every $(c,\nu)\in\Z(\C)$ and $H\in\Fun_{\C}(\M,\N)$ a module natural isomorphism
\begin{equation*}
    (c,\nu)\tr^{\PsiN}H=c\tr H\xrightarrow{\quad\varphi_{c,-}\;} H(c\tr-) =(c,\nu)\tr^{\PsiM}H
\end{equation*}
given by the $\C$-module functor structure of $H$. A routine check shows that these isomorphisms furnish a $\Z(\C)$-module functor structure on the identity functor of the category $\Fun_{\C}(\M,\N)$.
\end{proof}

\begin{remark}\label{factor_psi}
Notice that the functor $\PsiN$ \eqref{psiN} can be factorized as $U\circ\mathrm{S}$, where $U\colon \Z(\C_\subN^*)\to \C_\subN^*$ is the forgetful functor, and
\begin{equation}\label{schauenburg}
\begin{aligned}
    \mathrm{S}\Colon \Z(\C)&\longrightarrow\Z(\C_\subN^*), \\
                (c,\nu)&\longmapsto c\tr-
\end{aligned}
\end{equation}
where $c\tr-\colon\N\to\N$ is the $\C$-module functor assigned by \eqref{psiN}, and the $\C_\subN^*$-half-braiding on $c\tr-$ is defined for $H\in \C_\subN^*$ by the $\C$-module structure $\varphi_{c,-}\colon c\tr -\circ H\xsim H\circ c\tr-$ of $H$. Moreover, $\mathrm{S}$ is braided and thus endows $\PsiN$ with the structure of a central tensor functor.
\end{remark}

\begin{lemma}\label{lem:psi}
The functor \eqref{psiN} is surjective. Thus, the tensor functors $\PsiN$ and $\PsiM$ are perfect.
\end{lemma}
\begin{proof} In view of Remark \ref{factor_psi}, we have a factorization $U\circ\mathrm{S}$.
According to \cite[Thm.~3.13]{shimizu2020further} $\mathrm{S}$ is in particular an equivalence of $\kk$-linear categories. Furthermore, by \cite[Cor.~7.13.11]{etingof2016tensor}, the functor $U$ is surjective.
Thus, it follows by Proposition~\ref{prop:comp_rel_obj}(iii), that \eqref{psiN} is surjective. Since, composition of tensor functors is a tensor functor, considering appropriately adapted tensor structures on \eqref{schauenburg} implies that $\PsiN$ and $\PsiM$ are perfect.
\end{proof}

The internal (co)Hom's of the $\Z(\C)$-module category $\Fun_{\C}(\M,\N)$ can be described in terms of a (co)end as proved in \cite{fuchs2021internal}. In the case that $\C=\Vect$, this construction recovers, by the Yoneda lemma, the vector space of natural transformations between the functors involved.

\begin{definition}\cite[Def.\ 4]{fuchs2021internal}
Let $\M$ and $\N$ be exact $\C$-module categories. Given $H,H'\in \Fun_{\C}(\M,\N)$, their object of \textit{internal natural transformations} is defined as the internal hom
\begin{equation*}
    \uNat_{(\M,\N)}(H,H')\coloneqq \uHom^{\Z(\C)}_{\Fun_{\C}(\M,\N)}(H,H')\,.
\end{equation*}
When clear from the context, we will drop the indices $\M,\N$.
\end{definition}
In view of Lemma~\ref{lem:center-action-functors}, the $\Z(\C)$-module structure on $\Fun_{\C}(\M,\N)$ can be obtained by twisting either via the perfect tensor functor $\PsiN$ or via $\PsiM$. As a consequence, the internal natural transformations object can also be described in terms of either of these perfect tensor functors.

\begin{lemma}
\label{lem:uNatPsi}
Let $\M$ and $\N$ be exact $\C$-module categories. There is a natural isomorphism 
\begin{equation*}
\uNat_{(\M,\N)}(H,H')\cong\Phi_\subM^{\ra}(H^{\ra} \circ H')\cong\Psi_\subN^{\ra}(H' \circ H^{\la})    
\end{equation*}
for $H,\;H'\in \Fun_{\C}(\M,\N)$. 
\end{lemma}
\begin{proof}
The result follows from Lemma~\ref{lem:intHom} and isomorphism \eqref{eq:twisted_inner_hom}.
\end{proof}

We obtain the following result generalizing \cite[Thm.\ 18]{fuchs2021internal}, where we do not require the additional assumption that $\C$ is pivotal.
\begin{theorem}\label{thm:FunCexact}
Let $\C$ be a finite tensor category, and $\M$ and $\N$ be exact $\C$-module categories. The $\Z(\C)$-module category $\Fun_{\C}(\M,\N)$ is exact.
\end{theorem}
\begin{proof}
By \cite[Prop.~7.12.14]{etingof2016tensor}, $\Fun_{\C}(\M,\N)$ is an exact left $\C_\subN^*$-module category.
Also, by Lemma~\ref{lem:psi}, $\PsiN$ is a perfect tensor functor.
It follows from Proposition \ref{prop:perfect-iff} that ${}_{\PsiN}\Fun_{\C}(\M,\N)$ is an exact left $\Z(\C)$-module category. 
\end{proof}

\begin{proposition}\label{prop:psi-Frobenius}
The tensor functor $\PsiN$ is \Frob iff $\N$ is unimodular iff $\Psi_{\subN}^\ra$ is Frobenius monoidal. Similarly, the tensor functor $\PsiM$ is \Frob iff $\M$ is unimodular iff $\Phi_{\subM}^\ra$ if Frobenius monoidal.
\end{proposition}
\begin{proof}
From Remark \ref{factor_psi} we have that $\PsiN=U\circ\mathrm{S}$. Thus, according to Proposition \ref{prop:comp_rel_obj}, the relative modular object of $\PsiN$ can be computed as $U(\chi_{\mathrm{S}})\circ\chi_U$. But, since $\mathrm{S}$ is an equivalence $\chi_{\mathrm{S}}\cong\id_\subN$, and since $\Z(\C_\subN^*)$ is unimodular, $\chi_U\cong D_{\C_\subN^*}^{-1}$ by Proposition \ref{prop:rel-mod-obj}. It follows that the relative modular object of $\PsiN$ is $D_{\C_\subN^*}^{-1}$, which together with Proposition \ref{prop:frobenius_uni} proves that $\PsiN$ is Frobenius iff $\N$ is unimodular. Since $\PsiN$ is a central tensor functor, Proposition~\ref{prop:central-tensor-Frobenius} ensures that $\PsiN$ is \Frob iff $\PsiN$ is Frobenius iff $\Psi_{\subN}^\ra$ is Frobenius monoidal. 
\end{proof}

\begin{proposition}\label{prop:FunCunim}
Let $\C$ be a finite tensor category, and $\M$ and $\N$ be unimodular $\C$-module categories. The $\Z(\C)$-module category $\Fun_{\C}(\M,\N)$ is unimodular.
\end{proposition}
\begin{proof}
$\PsiN$ is perfect by Lemma \ref{lem:psi}. Now, assuming that $\N$ is unimodular, Proposition~\ref{prop:psi-Frobenius} implies that $\PsiN$ is \Frob\!\!. 
From \cite[Thm.\ 3.11]{spherical2025} the dual-tensor category of $\Fun_{\C}(\M,\N)$ with respect to $\C_\subN^*$ is equivalent to  $\C_\subM^*$. This means that a unimodular structure on ${}_\subC\M$ induces a unimodular structure on the $\C_\subN^*$-module category $\Fun_{\C}(\M,\N)$. Thus, by Proposition \ref{prop:uni-tensor-Frob} we can pull back the unimodular structure along $\PsiN$ to obtain a unimodular structure on the $\Z(\C)$-module category
$\Fun_{\C}(\M,\N)$.
\end{proof}

\begin{proposition}\label{prop:Frobenius_algebra}
Let $\C$ be a finite tensor category and $\M$ and $\N$ be exact $\C$-module categories. Additionally, assume that either $\M$ or $\N$ is unimodular.
\begin{enumerate}[{\rm (i)}]
    \item Let $H\in\Fun_{\C}(\M,\N)$ and $H^\ra\cong H^\la$ a $\C$-module natural isomorphism. Then the algebra $\uNat_{(\M,\N)}(H,H)$ is endowed with the structure of a Frobenius algebra in $\Z(\C)$.
    \item The algebra $\uNat(\id_{\,\subM},\id_{\,\subM})$ has the structure of a commutative Frobenius algebra in $\Z(\C)$.
\end{enumerate}
\end{proposition}
\begin{proof}
We prove (i) assuming that $\M$ is unimodular.
The assumption that $H^\la\cong H^\ra$ as $\C$-module functors implies that the composition $H^{\ra} \circ H$ has a canonical structure of a Frobenius algebra in $\C_\subN^*$. Moreover, $\M$ being unimodular implies that $\Phi_{\subM}^{\ra}$ is Frobenius monoidal according to Proposition~\ref{prop:psi-Frobenius}. Thus it preserves Frobenius algebras. 
Consequently, $\Phi_\subM^\ra(H^{\ra} \circ H)$ inherits a Frobenius algebra structure. The statement follows from Lemma~\ref{lem:uNatPsi}. The commutativity in assertion (ii) follows from \cite[Prop.~8.8.8]{etingof2016tensor}.
A similar argument shows the claim if we assume $\N$ to be unimodular instead.
\end{proof}

%%%%%%%%%%%%%%%%%%%%%%%%%%%%%%%%%%%%%%%%%%%%%%%%%%%%%%%%%%%%%%%
%%%%%%%%%%%%%%%%%%%%%%%%%%%%%%%%%%%%%%%%%%%%%%%%%%%%%%%%%%%%%%%
%%%%%%%%%%%%%%%%%%%%%%%%%%%%%%%%%%%%%%%%%%%%%%%%%%%%%%%%%%%%%%%

\subsection{Pivotal setting}
Next we explore the implications on the pivotal setting. Let $\C$ be a pivotal finite tensor category.
The pivotal structure $\fp_c \colon c\xsim c\dd$ on $\C$ induces a pivotal structure on $\Z(\C)$ via $\fp_{(c,\sigma)} \,{:=}\, \fp_c$ \cite[Ex.\,7.13.6]{etingof2016tensor}. This choice on the Drinfeld center renders the forgetful functor $U\colon\Z(\C)\to\C$ a pivotal tensor functor.
\begin{lemma}
Given a pivotal $\C$-module category $\N$, the tensor functor $\PsiN$ is pivotal.    
\end{lemma}
\begin{proof}
As described in Remark \ref{factor_psi} we have a factorization $\PsiN=U\circ \rm{S}$. By \cite[Prop.\ 5.15]{spherical2025}, $\rm{S}$ is pivotal, which leads to the result.
\end{proof}

Now, consider pivotal $\C$-module categories $\M$ and $\N$. The relative Serre functors of the $(\C_\subN^*,\C_\subM^*)$-bimodule category $\Fun_{\C}(\M,\N)$ can be expressed in terms of those of $\M$ and $\N$ according to \cite[Lemma 5.3]{spherical2025} and \cite[Cor.\ 4.12]{spherical2025} by
\begin{equation*}
    \Se_\subFun^{\scriptscriptstyle\C_\subN^*}(H)\cong \overline{\Se}_\subN^{\supC}\circ H\circ\Se_\subM^{\supC}\;,
    \qquad
    \Se_\subFun^{\scriptscriptstyle\overline{\C_\subM^*}}(H)\cong \Se_\subN^{\scriptscriptstyle\C}\circ H\circ\overline{\Se}_\subM^{\scriptscriptstyle\,\C}    
\end{equation*}
for $H\in \Fun_{\C}(\M,\N)$.

\begin{proposition}\label{prop:int-natural-piv-1}
Let $\M$ and $\N$ pivotal $\C$-module categories. 
\begin{enumerate}[{\rm (i)}]
\item $\Fun_{\C}(\M,\N)$ inherits the structure of a pivotal $(\C_\subN^*,\C_\subM^*)$-bimodule category via the isomorphisms
\begin{equation*}
\begin{aligned}
    \mathrm{P}_H&\Colon H\xRightarrow{\;\widetilde{\fq}^{-1}\,\circ \,\id\,\circ\, \widetilde{\fp}\;}\overline{\Se}_\subN^{\scriptscriptstyle\,\C}\circ H\circ\Se_\subM^{\scriptscriptstyle\C}
    \cong\Se_\subFun^{\scriptscriptstyle\C_\subN^*}(H)\\
    \mathrm{Q}_H&\Colon H\xRightarrow{\;\tfq\,\circ \,\id\,\circ\, \tfp^{-1}\;}\Se_\subN^{\scriptscriptstyle\,\C}\circ H\circ\overline{\Se}_\subM^{\scriptscriptstyle\C}
    \cong\Se_\subFun^{\scriptscriptstyle\overline{\C_\subM^*}}(H)
\end{aligned}
\end{equation*}
where $\tfp$ and $\tfq$ are the pivotal structures of ${}_\subC\M$ and ${}_\subC\N$ respectively.

\item For $H\in\Fun_\C(\M,\N)$, the compositions $H\circ H^\la$ and $H^\ra\circ H$ are symmetric Frobenius algebras in $\C^*_\subN$ and $\C^*_\subM$, respectively.
\end{enumerate} 
\end{proposition}
\begin{proof}
Part (i) follows by applying \cite[Prop.\ 5.4]{spherical2025} to the bimodule categories ${}_\subC\M_{\scriptscriptstyle\overline{\C_\subM^*}}$ and ${}_\subC\N_{\scriptscriptstyle\overline{\C_\subN^*}}$, that are pivotal bimodules in view of \cite[Lemma 5.8 (i)]{spherical2025}. By \cite[Thm.~3.1]{shimizu2023relative}, the internal ends, which are $H\circ H^\la$ and $H^\ra\circ H$ due to Lemma~\ref{lem:intHom}, are symmetric Frobenius algebras.
\end{proof}

\begin{proposition}\label{prop:int-natural-piv-2}
Let $\C$ be a pivotal finite tensor category, and $\M$ and $\N$ pivotal $\C$-module categories. Additionally, assume that either $\M$ or $\N$ or both are unimodular. The $\Z(\C)$-module category $\Fun_\C(\M,\N)$ is endowed with a pivotal structure.
\end{proposition}
\begin{proof}
Since $\N$ is unimodular, by Proposition~\ref{prop:psi-Frobenius}, $\PsiN$ is \Frob \!\!. On the other hand, $\PsiN$ is pivotal. As the $\C^*_\subN$-module $\Fun_\C(\M,\N)$ is pivotal, the claim follows from Proposition \ref{prop:piv-tensor-Frob}. 
\end{proof}

In the pivotal setting, the objects of internal natural transformations are not only algebras, but symmetric Frobenius algebras.
\begin{corollary}
\label{cor:uNatsymm}
Let $\C$ be a pivotal finite tensor category and $\M,\N$ be pivotal $\C$-module categories. Assume that $\N$ (or $\M$) is unimodular. 
\begin{enumerate}[{\rm (i)}]
    \item The algebra $\uNat(H,H)$ has the structure of a symmetric Frobenius algebra in $\Z(\C)$ for any $H\in \Fun_{\C}(\M,\N)$.
    \item $\uNat(\id_{\,\subM},\id_{\,\subM})$ has the structure of a commutative symmetric Frobenius algebra in $\Z(\C)$. 
    \item The functors $\Psi_\subN^\ra$ and $\Phi_\subM^\ra$ are pivotal Frobenius monoidal. 
\end{enumerate}
\end{corollary}
\begin{proof}
By Lemma~\ref{lem:uNatPsi}, $\uNat(H,H)=\Psi_\subN^\ra(H\circ H^\la)$. So part (i) follows by combining Proposition~\ref{prop:int-natural-piv-1}(ii) which tells that $H\circ H^\la$ is symmetric Frobenius and Theorem~\ref{thm:piv_symm_Frob}(i) which tells that $\Psi_\subN^\ra$ preserves symmetric Frobenius algebras. Part (ii) follows is a similar manner using Theorem~\ref{thm:piv_symm_Frob}(ii). Lastly, by \cite[Thm.~1.2(ii)]{yadav2024frobenius}, part (ii) implies part (iii).
\end{proof}

Under the additional requirement that $\C$ is unimodular and spherical, any pivotal $\C$-module category is unimodular \cite[Thm.\ 5.18]{spherical2025}.

\begin{corollary}\label{cor:int-nat-spherical}
    Let $\C$ be a spherical $($unimodular$)$ tensor category and $\M,\N$ are pivotal $\C$-module categories. Then $\Fun_{\C}(\M,\N)$ is a pivotal $\Z(\C)$-module category.
\end{corollary}
\begin{proof}
According to \cite[Thm.\ 5.18]{spherical2025}, $\C_\subN^*$ and $\C_\subM^*$ are automatically unimodular under unimodularity and sphericality of $\C$. Thus, $\N$ is unimodular as a $\C$-module category and the claim follows by Proposition~\ref{prop:int-natural-piv-2}.
\end{proof}
\begin{remark}
In the situation of Corollary \ref{cor:int-nat-spherical}, a trivialization $\tfu_{\subC}:\unit\xsim D_\subC$ together with the pivotal structure of $\M$ form a unimodular structure on $\M$ via the composition
\begin{equation*}
\tfu_{\subM}^{\tfp}\Colon\id_\subM\xRightarrow{~\;{\tfp}^{\,2}\;~}
\left(\Se^\supC_\subM\right)^2\xRightarrow{\;{}^{\tfu{}_{\subC}}\tr\id\;}
D_\subC^{-1}\tr\left(\Se^\supC_\subM\right)^2\xRightarrow{\;\eqref{eq:Radford_mod}\;}D_{\overline{\C_\subM^*}}^{-1} \;.
\end{equation*}
With this choice of unimodular structure, $\M$ is $(\tfu_{\subC},\tfu^{\tfp}_\subM)$-spherical. Moreover, the pivotal $(\C_\subN^*,\C_\subM^*)$-bimodule category $\Fun_{\C}(\M,\N)$ becomes $(\tfu_{\subN}^{\tfq},\tfu_{\subM}^{\tfp})$-spherical. It follows from Proposition \ref{prop:sph-Frob} that $\Fun_{\C}(\M,\N)$ is $({}_{{}_\Psi}\!\tfu_{\subN}^{\tfq},{}_{{}_\Phi}\!\tfu_{\subM}^{\tfp})$-spherical as a $\Z(\C)$-bimodule category.
\end{remark}

%%%%%%%%%%%%%%%%%%%%%%%%%%%%%%%%%%%%%%%%%%%%%%%%%%%%%%%%%%%%%%%
%%%%%%%%%%%%%%%%%%%%%%%%%%%%%%%%%%%%%%%%%%%%%%%%%%%%%%%%%%%%%%%
%%%%%%%%%%%%%%%%%%%%%%%%%%%%%%%%%%%%%%%%%%%%%%%%%%%%%%%%%%%%%%%
%%%%%%%%%%%%%%%%%%%%%%%%%%%%%%%%%%%%%%%%%%%%%%%%%%%%%%%%%%%%%%%
%%%%%%%%%%%%%%%%%%%%%%%%%%%%%%%%%%%%%%%%%%%%%%%%%%%%%%%%%%%%%%%
%%%%%%%%%%%%%%%%%%%%%%%%%%%%%%%%%%%%%%%%%%%%%%%%%%%%%%%%%%%%%%%

\bibliographystyle{abbrv}
\bibliography{references}

\end{document}